\documentclass[11pt]{amsart}
\usepackage[utf8]{inputenc}
\usepackage{bm}
\usepackage{fontenc}
\usepackage{comment}
\usepackage{amsfonts}
\usepackage{amssymb}
\usepackage{amsmath}
\usepackage{amsthm}\usepackage{mathtools}
\usepackage{enumerate}
\usepackage{enumitem}
\usepackage[pagebackref,colorlinks,linkcolor=blue,citecolor=blue,urlcolor=blue,hypertexnames=true]{hyperref}
\usepackage[pagebackref,hypertexnames=true]{hyperref}
\usepackage{enumitem}
\usepackage{mathrsfs}
\usepackage{tikz}
\usetikzlibrary{calc}
\usepackage{marginnote}
\usepackage{xcolor,enumitem}
\usepackage{soul}\usepackage{tikz}
\usepackage[all,cmtip]{xy} \usepackage{caption}

\newcommand{\Z}{\mathbb{Z}}

\newcommand{\N}{\mathbb{N}}
\newcommand{\C}{\mathbb{C}}

\newcommand{\eps}{\varepsilon}

\usepackage{bbm}
\newcommand{\Hawaii}{Hawai\kern.05em`\kern.05em\relax i}

\newcommand{\cstu}{\mathrm{C}^*_u}

\newtheorem*{rigprob*}{Rigidity Problem for uniform Roe Algebras}
\newtheorem*{rigprobcorona*}{Rigidity Problem for uniform Roe Coronas}

\newcommand{\cstar}{$\mathrm{C}^*$}

\newcommand{\cU}{\mathcal{U}}

\newcommand{\cF}{\mathcal{F}}

\newcommand{\cB}{\mathcal{B}}
\newcommand{\cK}{\mathcal{K}}

\newcommand{\cI}{\mathcal I}

\newcommand{\R}{\mathbb{R}}

\numberwithin{equation}{section}

\newtheorem{theorem}{Theorem}[section]
\newtheorem*{theorem*}{Theorem}
\newtheorem{proposition}[theorem]{Proposition}
\newtheorem{problem}[theorem]{Problem}
\newtheorem*{proposition*}{Proposition}
\newtheorem{lemma}[theorem]{Lemma}
\newtheorem*{lemma*}{Lemma}
\newtheorem{corollary}[theorem]{Corollary}
\newtheorem*{corollary*}{Corollar}

\newtheorem*{fact*}{Fact}
\theoremstyle{definition}
\newtheorem{definition}[theorem]{Definition}
\newtheorem*{definition*}{Definition}

\newtheorem*{acknowledgments}{Acknowledgments}
\newtheorem{claim}[theorem]{Claim}
\newtheorem*{claim*}{Claim}

\newtheorem*{conjecture*}{Conjecture}

\theoremstyle{remark}
\newtheorem{example}[theorem]{Example}
\newtheorem*{example*}{Example}
\newtheorem{remark}[theorem]{Remark}
\newtheorem*{remark*}{Remark}

\newtheorem*{note*}{Note}
\newtheorem*{question*}{Question}

\newcommand{\Orb}{\mathrm{Orb}}
\newcommand{\calOrb}{{\mathcal{O}rb}}
\newcommand{\qcalOrb}{{q\text{-}\mathcal{O}rb}}
\newcommand{\COrb}{\overline{\mathrm{Orb}}}
\newcommand{\qOrb}{\mathrm{q}\text{-}\mathrm{Orb}}
\newcommand{\pOrb}{\mathrm{p}\text{-}\mathrm{Orb}}
\newcommand{\HOrb}{\mathrm{H}\text{-}\mathrm{Orb}}

\DeclareMathOperator{\ran}{ran}
\DeclareMathOperator{\dom}{dom}

\newcommand{\cM}{\mathcal M}

\usepackage{enumitem}

\numberwithin{equation}{section}

\newcommand{\PT}{\mathrm{PT}}

\begin{document}

\title[Dynamics  in large scale geometry]{Dynamics  in large scale geometry}%

% \date{\today} 

\author[B. M. Braga]{Bruno M. Braga}
\address[B. M. Braga]{IMPA, Estrada Dona Castorina 110, 22460-320, Rio de Janeiro, Brazil}
\email{demendoncabraga@gmail.com}
\urladdr{https://sites.google.com/site/demendoncabraga}
\thanks{B. M. Braga  was partially supported by FAPERJ (Proc. E-26/200.167/2023), by CNPq (Proc. 303571/2022-5), and by  Serrapilheira (R-2501-51476).}

\author[A. Buss]{Alcides Buss}
\address[A. Buss]{Universidade Federal de Santa Catarina, 88040-970 Florianopolis SC, Brazil}
 \email{alcides.buss@ufsc.br}
 \urladdr{http://www.mtm.ufsc.br/~alcides}
 \thanks{A. Buss  was partially supported by CNPq and FAPESC}

\author[R. Exel]{Ruy Exel}
\address[R. Exel]{Universidade Federal de Santa Catarina, 88040-970 Florianopolis SC, Brazil}
 \email{ruyexel@gmail.com}
 \urladdr{http://www.mtm.ufsc.br/~exel}
 \thanks{R. Exel  was partially supported by  CNPq, Brazil}

 
 
 \maketitle

\begin{abstract}
 We investigate the large scale geometry of certain metric spaces through the lens
of dynamics. Our approach establishes a close connection  between large scale dynamical phenomena and operator algebras by characterizing various large scale dynamic behaviors in terms of GNS representations of the uniform Roe algebras  arising from natural canonical states.

Our dynamical systems are given by the Stone--\v{C}ech boundary of metric spaces
together with their inverse semigroup of partial translations. This defines a space
of orbits  and we  characterize Hausdorffness and $T_1$-ness of this space by   the failure of coarse embeddability of certain metric spaces. Surprisingly, while the orbit space has very weak separation properties, we show that it satisfies a certain ``localized version'' of Urysohn's lemma.

We show that the topology of the space of orbits and quasi-orbits are given by the space of irreducible representations of uniform   Roe algebras and by the space of their primitive ideals, respectively. As a highlight of the  theory developed herein, we provide classes of spaces such that the  prime ideals of their uniform Roe algebras are primitive. This is the case for instance of spaces whose orbit space is $T_1$.
\end{abstract}

\tableofcontents

\section{Introduction}

We study the large scale geometry of metric spaces through a
dynamical viewpoint. Our approach reveals a  close   connection  between coarse
geometry, topological dynamics on the Stone--\v{C}ech boundary, and the operator
algebraic structure of uniform Roe algebras. In a nutshell, we thoroughly
 and systematically study dynamical systems given by certain metric spaces together
with their inverse semi-group of partial translations. As orbits are usually not
closed, the \emph{space of orbits}, obtained by taking the quotient by the orbit
relation, satisfies only very weak separation properties; usually these spaces will not even be $T_0$. Our first results provide large scale characterizations of such phenomena in terms of coarse embeddability. The characterization of when orbits are closed becomes a cornerstone  for the theory.

As it turns out,   dynamics in large scale geometry is intrinsically related to
certain operator algebras defined over our metric spaces. Higson functions provide
the appropriate tools to detect the closure of the orbit relation and maximal orbit
closures. On the other hand, representations of uniform Roe algebras can   be used
to characterize orbit and quasi-orbit relations.   Moreover, our study of dynamics
is intrinsically  linked to the ideal structure on uniform Roe algebras and, more
specifically, to the spaces of primitive ideals and of irreducible representations
of these algebras. As a highlight of the theory developed herein, we
use  our results to study the celebrated problem of when prime ideals are primitive,
in the special case of uniform Roe algebras, solving it affirmatively in many
cases, including all examples whose   orbit space is $T_1$.

For the remainder of this introduction, we now describe the main findings  of our work in more detail.

\subsection{Our dynamical systems}
Large scale geometry is the study of the global  aspects of metric spaces. In this context, local properties   are not relevant. Instead, we are interested in events taking place outside every bounded subset. Moreover, events happening at a uniform distance from each other are identified and treated as being morally the same.

\begin{definition}
    Let $X$ be a   metric space, $A,B\subseteq X$, and $f\colon A\to B$ be a bijection. We say that $f$ is a \emph{partial translation} if
    \[\sup_{x\in A} d(x,f(x))<\infty.\]
   We denote by $\PT(X)$ the set of all partial translations of $X$.
\end{definition}

Given a metric space $X$, the set $\PT(X)$ has a natural inverse semigroup
structure: its product is given by composition of maps whenever the composition
is defined. Although $\PT(X)$ acts on $X$ and this yields the dynamical system
$(X,\PT(X))$, such an action is not particularly informative for understanding
large scale phenomena in $X$. Indeed, $\PT(X)$ acts transitively on $X$, and by
considering only this action we are unable to distinguish   events that
happen uniformly at once.

As our interest lies in events that escape every bounded subset of the space,
we must regard $\PT(X)$ as an inverse semigroup acting not only on $X$ but
also on the \emph{boundary} of $X$, where each point of this boundary encodes a
way of approaching infinity --- that is, a way of leaving all bounded subsets. Our model for the boundary of $X$ will be given by the Stone--\v{C}ech
compactification of $X$, denoted by $\beta X$. Since we are interested only in
large scale phenomena, we may assume that our metric spaces are discrete, by
replacing them with appropriate $\delta$-separated and $\varepsilon$-dense
subsets. We furthermore assume that our metric spaces are \emph{uniformly
locally finite} (abbreviated as \emph{u.l.f.}); that is, $X$ has the property
that for every $r>0$ there is an upper bound on the cardinality of the balls of
radius $r$ in $X$.

Since $X$ is discrete, the Stone--\v{C}ech compactification $\beta X$ can be
identified with the space of all ultrafilters on $X$ endowed with its canonical
topology. Precisely, for $A\subseteq X$ we define
\[
\overline A = \{\xi\in\beta X : A\in\xi\},
\]
and we equip $\beta X$ with the topology having the sets $(\overline A)_{A\subseteq
X}$ as a basis. Identifying $X$ with the principal ultrafilters, we regard $X$ as an open and
dense subset of $\beta X$. The \emph{Stone--\v{C}ech boundary of $X$} is
$\partial X = \beta X \setminus X$, the space of all nonprincipal ultrafilters
on $X$. It is immediate that if $A\subseteq X$, then the closure of $A$ in
$\beta X$ is precisely $\overline A$, justifying our choice of notation.

Given an ultrafilter $\xi \in \beta X$ and a partial translation $f$ of $X$ with
$\dom(f) \in \xi$, we define $\overline f(\xi) \in \beta X$ as the ultrafilter
determined by
\[
B \in \overline f(\xi)
\ \Longleftrightarrow\
\exists\, A \subseteq \dom(f)\text{ with } A\in \xi \text{ such that } f(A) \subseteq B.
\]
In other words, $\overline f$ is precisely the Stone--\v{C}ech extension of the partial
translation $f \colon \dom(f) \to \mathrm{im}(f)$ to a map
$\overline{\dom(f)} \to \overline{\mathrm{im}(f)}$. Thus we may regard the
inverse semigroup $\PT(X)$ as acting not only on $X$ but on the entire
Stone--\v{C}ech compactification $\beta X$.

\begin{definition}
Let $X$ be a u.l.f.\ metric space. The discussion above defines a dynamical
system $(\beta X,\PT(X))$. For $\xi \in \beta X$, the \emph{orbit} of $\xi$ is
\[
\Orb(\xi)
=
\left\{\overline f(\xi) \in \beta X : f \in \PT(X) \text{ and } \dom(f) \in \xi \right\}.
\]
Two elements of $\beta X$ are said to be \emph{orbit equivalent} if they lie in
the same orbit. This defines an equivalence relation, called the \emph{orbit
equivalence relation} and denoted by $\Orb$. Moreover, since $\PT(X)$ leaves
$\partial X$ invariant, the same construction yields the dynamical system
$(\partial X,\PT(X))$.
\end{definition}

The dynamical systems described above have already been extensively studied in
the literature, especially in the case where $X$ is a finitely generated group equipped with a word metric. We refer the reader to \cite{Protasov2005TopAndAppl,Protasov2008TopAndAppli,HindmanStrauss2012Book,
SpakulaWillett2017,Protasov2021TopProc}.

As the dynamical system  $(X,\PT(X))$ is    transitive, the dynamical system $(\beta X,\PT(X))$ is always \emph{topologically transitive} in the sense that for all open $U,V\subseteq \beta X$, there is $\xi\in U$ and a partial translation $f$ of $X$ such that $\overline f(\xi)\in V$. The picture changes however when we restrict our action to the boundary $\partial X$:

\begin{theorem*}[See Theorem \ref{ThmMixAction.InTheText} below]
Let $X$ be an infinite  u.l.f.\ metric space.  The dynamical system $(\partial X,\PT(X))$ is not topologically transitive.\label{ThmMixAction}
\end{theorem*}

 \subsection{Equivalence relations induced by dynamics} As our metric spaces are u.l.f., we can show that all orbits are countable (see Proposition \ref{PropOrbitsAreCountable}). Our next theorem shows that orbits are finite exactly when they are closed and, moreover, it characterizes when this happens geometrically in $X$. This geometric characterization is fundamental for the understanding of orbits and it will be essential throughout these notes. 
 
 %Below,  if $X$ is a metric space, $A\subseteq X$, and $r>0$, $N_r(A)$ denotes  the $r$-neighborhood of $A$, i.e., $N_r(A)=\{x\in X\colon d(x,A)\leq r\}$.

\begin{theorem*}[See Theorem \ref{ThmCharacterizationOrbClosed.IntheText} below]
    Let $X$ be a u.l.f.\ metric space and $\xi\in \beta X$. The following are equivalent.
    \begin{enumerate}
\item\label{Item1ThmCharacterizationOrbClosed} $\mathrm{Orb}(\xi)$ is closed.
\item\label{Item2ThmCharacterizationOrbClosed}   $\mathrm{Orb}(\xi)$ is finite.
\item\label{Item3ThmCharacterizationOrbClosed}
        There is  $r>0$ such that
    \[\lim_{x,\xi}d\left(N_r(x),X\setminus N_r(x)\right)=\infty.\footnote{Throughout these notes, if $X$ is a metric space, $A\subseteq X$, and $r>0$, $N_r(A)$ denotes  the $r$-neighborhood of $A$, i.e., $N_r(A)=\{x\in X\colon d(x,A)\leq r\}$. }\]\end{enumerate}\label{ThmCharacterizationOrbClosed}
\end{theorem*}

The dynamics on $\beta X$ induced by  the partial translations $\PT(X)$ give rise to many other  equivalence relations. In these notes, besides the orbit equivalence relation, we study
\begin{enumerate}
    \item[I.] quasi-orbit equivalence relation,
    \item[II.] Higson equivalence relation, and
    \item[III.] pseudo-orbit equivalence relation.
\end{enumerate}
For brevity, we only highlight some of our findings here and refer the reader to the precise statements in the text. Firstly, given $\xi\in \beta X$, the \emph{quasi-orbit of $\xi$} is
\[\qOrb(\xi)=\left\{\eta\in \beta X\colon \COrb(\eta)=\COrb(\xi)\right\}.\]
Elements of $\beta X$ are \emph{quasi-orbit equivalent} if they are in the same quasi-orbit. Clearly, the relation of quasi-orbit equivalence, $\qOrb$, contains $\Orb$.\footnote{We point out that $\qOrb(\xi)$ is in general smaller than $\COrb(\xi)$.}
We present a geometric characterization of when quasi-orbits are closed in Theorem \ref{ThmCharacterizationqOrbClosed}.
Higson-orbits and pseudo-orbits are automatically closed, so no such characterization is necessary. In fact, more than merely having closed orbits, these relations are closed equivalence relations.

Higson functions (see Definition \ref{DefiHigson} below) and Higson equivalence on $\partial X$  play a very important role throughout the text. Our main technical result in this direction  is Theorem \ref{ThmHigsonRel}, which characterizes Higson equivalence (see also Theorem \ref{ThmHigsonRel.Cont}). The Higson compactification is a main tool in the proof of Theorem \ref{ThmMixAction.InTheText} above as well as for Theorem \ref{ThmLocalUrysohnForOrb.InTheText} below. We refer the reader to Section \ref{SectionHigson} details.

Pseudo-orbits are dealt with in Section \ref{SectionPseudo}. Theorem \ref{ThmChainTransitiveSpaces} provides a geometric characterization of when the corona $\partial X$ has a single pseudo-orbit. This is the case for instance for $\N^n$, but not for $\mathbb F_n$ (see Corollaries \ref{CorPseudoOrbN} and \ref{CorPseudoOrbF}, respectively).

\subsection{The space of orbits}\label{SubsectionSpaceOfOrbits}
We can  endow the set of orbits with the quotient topology and obtain the \emph{orbit space of $\beta X$},
\[\calOrb(\beta X)=\beta X/\Orb.\]
However, since $\PT(X)$ acts transitively on $X$, $X$ collapses completely and becomes a single isolated point in this space. Moreover, as $X$ is dense in $\beta X$, this isolated point is itself dense in the whole space. In order to understand the
subtleties of  orbits, it is therefore better to restrict our study to the  orbits  of nonprincipal ultrafilters of $X$. We define the   \emph{space of orbits of $\partial X$} as the  topological quotient space
    \[\calOrb(\partial X)=\partial X/\Orb.\]
 So, $\calOrb(\partial X)$   is a topological subspace of $\calOrb(\beta X)$.

Since the relation $\Orb$ is not closed in general, the space of orbits has very
weak topological properties.\footnote{In contrast, if one considers the space of
Higson-orbits $(\partial X/\HOrb)$ or the space of pseudo-orbits $(\partial
X/\pOrb)$, these quotients are automatically Hausdorff, since both equivalence
relations are closed (and the quotient maps are open).}  In order to state our results in this direction, we first  introduce a family of spaces whose large scale geometric aspects translate dynamically in powerful ways. Let
\begin{equation}\label{Rq.Definition.M}\cM=\left\{(a_i)_i\in \{0,1\}^\N\colon |\{i\in \N\colon a_i=1\}|<  \infty\right\}
\end{equation}
  and endow it  with the metric
\begin{equation}\label{Eq.Definiion. Metric.M}
d(\bar a,\bar b)
   =\left|\sum_{i=1}^\infty a_i\,3^i \;-\; \sum_{i=1}^\infty b_i\,3^i\right|
\end{equation}
for all $\bar a=(a_i)_i,\bar b=(b_i)_i\in \cM$. Equivalently, $\cM$ can be identified with the subset of natural numbers   whose
\emph{ternary expansion} uses only the digits $0$ and $1$.\footnote{More precisely, $\cM$ is canonically isometric to the subset of natural numbers larger than or equal to $3$  whose
ternary expansion uses only the digits $0$ and $1$.} The space $\cM$ has
asymptotic dimension zero (see Definition \ref{Definition.Asym.Dim.d}), and it is the ``universal'' space of this
type: every u.l.f.\ metric space of asymptotic dimension zero coarsely
embeds into it (see
\cite[Theorem~3.11]{DranishnikovZarichnyi2004TopAndAppl}).  For any $k\in\N$, let
\begin{equation}\label{Rq.Definition.Mk}\mathcal M_k=\left\{(a_i)_i\in \cM\colon |\{i\in \N\colon a_i=1\}|=k\right\}.
\end{equation}
The spaces $\cM_k$ can be seen as subspaces of $\N^k$. Precisely, each $\cM_k$ is canonically coarsely equivalent to
\begin{equation}\label{Rq.Definition.Mk.concreta}\left\{(n_1^2,\ldots, n_k^2)\in \N^k\colon n_1<\ldots< n_k\right\},\end{equation}
where $\N^k$ is considered with its usual metric.

 As a consequence of Theorem \ref{ThmCharacterizationOrbClosed.IntheText} (and of Theorem \ref{ThmCharacterizationqOrbClosed}), we can characterize when the space of orbits of $\partial X$ is $T_1$. Precisely:

\begin{theorem*}[See Theorem \ref{CorClosedQuasiOrbitsCoarseDisjointUnionSing.Inthetext} below] 
   The following are equivalent for a u.l.f.\ metric space $X$:
   \begin{enumerate}
       \item\label{Item1CorClosedQuasiOrbitsCoarseDisjointUnionSing} $X$ does not contain a coarse copy of $\cM_2$.
       \item\label{Item2CorClosedQuasiOrbitsCoarseDisjointUnionSing}  $\mathrm{Orb}(\xi)$ is closed for all $\xi \in \partial X$.
       \item\label{Item3CorClosedQuasiOrbitsCoarseDisjointUnionSing} $\qOrb(\xi)$  is closed for all $\xi \in \partial X$.
       \item\label{Item4CorClosedQuasiOrbitsCoarseDisjointUnionSing} The space $\calOrb(\partial X)$ is $T_1$.
   \end{enumerate}
\label{CorClosedQuasiOrbitsCoarseDisjointUnionSing}
\end{theorem*}

The Hausdorffness of the orbit space can also be characterized in terms of coarse embeddability, but this requires a space which lies    in between $\cM_1$ and $\cM_2$. For this, let $\N=\bigsqcup_n\N_n$ be a partition into infinite subsets  and let $\cM_{3/2}$ be the union over all $n\in\N$ of all elements in $\cM_2$ whose first coordinate with a $1$   is either its $1$st or its  $m_n$th coordinate and its second coordinate with a $1$ is in $\N_n$. Alternatively, it is readily seen that this space is coarsely equivalent to
\[\{(1,k^2),(n,k^2)\in \N^2\colon n\in\N,\ k\in \N_n\}.\]
In other words, $\cM_{3/2}$ is the coarse disjoint union  over $n$ of $X_n=\cM_1\sqcup\cM_1$ where the Hausdorff distance in each $X_n$ between the disjoint copies of $\cM_1$ is finite but goes to infinite as $n\to \infty$ (see Definition \ref{Defi.Coarse.Disj.Union} for details on coarse disjoint unions).

\begin{theorem*}[See Theorem \ref{Thm.Hausdorfness.Charac.InTheText} below]
 The following are equivalent for a u.l.f.\ metric space $X$:
     \begin{enumerate}
         \item\label{Prop.HausdorffSSS.Item1} $X$ does not contain a coarse copy  of $\cM_{3/2}$.
         \item\label{Prop.HausdorffSSS.Item2} $\calOrb(\partial X)$ is Hausdorff.
     \end{enumerate}\label{Thm.Hausdorfness.Charac}
\end{theorem*}

 The question of when $\calOrb(\partial X)$ is $T_0$ appears to be
considerably subtler. By Theorem~\ref{CorClosedQuasiOrbitsCoarseDisjointUnionSing.Inthetext},
every u.l.f.\ space that does not contain a coarse copy of $\cM_2$ has this property. However,  $\calOrb(\partial \mathcal M_2)$ is still $T_0$
(see Proposition~\ref{PropSimpleSpaceT_0}).  A natural conjecture would be that $\calOrb(\partial X)$ is $T_0$
whenever $X$ has asymptotic dimension zero. The next theorem shows
that this is false. It follows again from
Theorems~\ref{ThmCharacterizationOrbClosed.IntheText} and
\ref{ThmCharacterizationqOrbClosed} and provides a sufficient
geometric condition ensuring that $\calOrb(\partial X)$ fails to be
$T_0$.

\begin{theorem*}[See Theorem \ref{CorAsymDomQuaseLargerOrb.Inthetext} below] 
Let $X$ be a u.l.f.\ metric space.  If $\cM$ coarsely embeds into $X$,
then $\calOrb(\partial X)$ is not $T_0$. In particular, if the
asymptotic dimension of $X$ is at least $1$, then $\calOrb(\partial X)$
is not $T_0$.\label{CorAsymDomQuaseLargerOrb}
\end{theorem*}

Surprisingly, while $\calOrb(\partial X)$ is typically far from being $T_0$, we show
that the following ``localized'' version of Urysohn's lemma always
holds.

\begin{theorem*}[See Theorem \ref{ThmLocalUrysohnForOrb.InTheText} below]
Let $X$ be a u.l.f.\ metric space and suppose that
$E,F\subseteq \calOrb(\partial X)$ are closed subsets that can be
separated by open sets. Then there exists a continuous function
$f\colon \calOrb(\partial X)\to [0,1]$ such that $f|_E=0$ and $f|_F=1$.\label{ThmLocalUrysohnForOrb}
\end{theorem*}

The proof of Theorem \ref{ThmLocalUrysohnForOrb.InTheText} heavily relies on the analysis of  Higson functions and on a characterization of when elements in $\partial X$ can be   distinguished by Higson functions (see Theorems \ref{ThmHigsonRel} and \ref{ThmHigsonRel.Cont}).

\subsection{Uniform Roe algebras and GNS representations} We now describe the interactions between the dynamics of $\partial X$ and the uniform Roe algebra of $X$. We start introducing some necessary notation.

Given a metric space $X$, $\ell_2(X)$ denotes   the Hilbert space of complex-valued families indexed by $X$ which are square-summable and we denote the canonical orthonormal basis of $\ell_2(X)$ by $(\delta_x)_{x\in X}$. The space of bounded operators on $\ell_2(X)$ is denoted by $\cB(\ell_2(X))$ and $\ell_\infty(X)$ is the \cstar-subalgebra of $\cB(\ell_2(X))$ of diagonal operators. Given a partial translation $f$ of $X$, we can define an operator $v_f$ on $\ell_2(X)$ by letting
\[v_f\delta_x=\left\{\begin{array}{ll}
 \delta_{f(x)} ,   &\ x\in \dom(f)  \\
   0,  & x\not\in \dom(f),
\end{array}\right. \]
so, each $v_f$ is a partial isometry of $\ell_2(X)$.

\begin{definition}
    Let $X$ be a u.l.f.\  metric space. The \emph{uniform Roe algebra of $X$}, denoted by $\cstu(X)$, is the \cstar-subalgebra of $\cB(\ell_2(X))$ generated by  $\{v_f\in \cB(\ell_2(X))\colon f\in \PT(X)\}$.
\end{definition}

Let $E\colon \cB(\ell_2(X))\to \ell_\infty(X)$ denote the canonical conditional expectation, i.e., given $a\in \cB(\ell_2(X))$,  $E(a)$ is simply the diagonal operator whose diagonal coordinates coincide with the ones of $a$. As $X$ is discrete, $\ell_\infty(X)$ is canonically identified with $C(\beta X)$, the \cstar-algebra of continuous functions on $\beta X$.   Given $\xi\in \beta X$, this identification allows us to  define a state $\varphi_\xi$ on $\cstu(X)$ by letting
\[\varphi_\xi(a)=E(a)(\xi)\]
for all $a\in \cstu(X)$.

\begin{definition}\label{DefiGNSRep}
    Let $X$ be a u.l.f.\   metric space and $\xi\in \beta X$. We denote the GNS representation given by the state $\varphi_\xi$ defined above by $\pi_\xi$. We refer to Section \ref{SectionGNSRep} for the precise definition of $\pi_\xi$.
\end{definition}

 As we show in Proposition \ref{PropGNSIrreducible}, the representations $\pi_\xi$ are always irreducible (Definition \ref{DefiIrreducibleRepPrimitive}). Hence, all $\ker(\pi_\xi)$ are primitive ideals.  The following
 theorem characterizes orbit and quasi-orbit equivalence in terms of the GNS representations above.

\begin{theorem*}[See Theorems \ref{ThmOrbitEquivIFFGNSEquiv.Inthetext}  and \ref{ThmOrbitEquivIFFGNSEquiv.Inthetext2}  below]
  Let $X$ be a u.l.f.\ metric space and   $\xi,\eta\in \beta X$. Let $\pi_\xi,\pi_\eta$ be the GNS representations given in Definition \ref{DefiGNSRep}. \begin{enumerate}
      \item\label{ThmOrbitEquivIFFGNSEquiv.Item1}   $\Orb(\xi )=\Orb(\eta )$ if and only if    $\pi _\xi $ and $\pi _\eta $ are unitarily equivalent.
      \item\label{ThmOrbitEquivIFFGNSEquiv.Item2} $\qOrb(\xi )=\qOrb(\eta )$ if and only if   $\ker(\pi _\xi)=\ker(\pi _\eta)$.
  \end{enumerate}\label{ThmOrbitEquivIFFGNSEquiv}
  \end{theorem*}

 For our next theorem, we need the   \emph{space of quasi-orbits of $\partial X$}, i.e., the   topological quotient space
    \[\qcalOrb(\partial X)=\partial X/\mathrm{q}\text{-}\Orb.\]
 As being quasi-orbit equivalent is weaker than being orbit equivalent, there is a canonical quotient map \[ \calOrb(\partial X)\mapsto \qcalOrb(\partial X)\]
 which, moreover, is an  open  map (see Corollary \ref{CorQuotientOpen}).

We prove that both $\calOrb(\beta X)$ and $\qcalOrb(\beta X) $ can be canonically seen as topological subspaces of certain analytical objects related to $\cstu(X)$. Precisely, in the next statement, $\mathrm{Prim}(\cstu(X))$  denotes the space of primitive ideals of $\cstu(X)$
 endowed with the hull-kernel topology (or Jacobson topology) and $\mathrm{Irr}(\cstu(X))$ denotes the space of the irreducible representations of $\cstu(X)$ (see Definition \ref{DefiIrreducibleRepPrimitive} and Section \ref{SubsectionTopOrbSpacePrimitive} for definitions).

\begin{theorem*}[See Theorem \ref{ThmMapFromOrbToPrimIsNice.InTheText} below]
    Let $X$ be a u.l.f.\ metric space. Both the maps \begin{equation*}  \Orb(\xi)\in \calOrb(\beta X) \mapsto [\pi_\xi]\in \mathrm{Irr}(\cstu(X))
    \end{equation*}
    and \begin{equation}\label{EqMapOrbToPrim.2Intro} \qOrb(\xi)\in \qcalOrb(\beta X) \mapsto \ker(\pi_\xi)\in \mathrm{Prim}(\cstu(X))
    \end{equation}
    are homeomorphisms onto their images.\label{ThmMapFromOrbToPrimIsNice}
\end{theorem*}

We point out that the assignments $\xi\in \beta X\mapsto \ker(\pi_\xi)\in \mathrm{Prim}(\cstu(X))$ are related to other maps studied in the literature (e.g., \cite{ChenWang2004JFA,ChenWang2005}) connecting  open orbit-invariant subsets of $\beta X$ with the lattice of ideals of $\cstu(X)$. We discuss this in detail in Sections \ref{SubSubSectionIdeals} and   \ref{SubsectionTopOrbSpacePrimitive} below. For now, we simply mention that  our study has a fundamental difference since the map in \eqref{EqMapOrbToPrim.2Intro} has as image primitive ideals:  this allows us to deal with topological properties of these maps and not only lattice properties as done in \cite{ChenWang2004JFA,ChenWang2005} (see Remark \ref{RemarkTopLattice}).

\subsection{Prime and primitive ideals of uniform Roe algebras}
Primitive ideals are always prime,\footnote{Recall that an ideal $I$ of a $\mathrm C^*$-algebra $A$ is \emph{prime} if
    for all ideals $I_0,I_1\subseteq A$ we have that $I_0 I_1\subseteq I$ implies
 that    either $I_0\subseteq I$ or $I_1\subseteq I$.} and the converse holds for separable
$\mathrm C^*$-algebras (\cite[Corollary~1]{Dixmier1960BullSMF}).  For nonseparable
$\mathrm C^*$-algebras, however, prime ideals need not be primitive
(\cite[Theorem on p.~360]{Weaver2003JFA}). Since uniform Roe algebras always
contain $\ell_\infty(X)$, they are never separable (unless $X$ is finite).
On the other hand, they are generated by $\ell_\infty(X)$ together with a
countable family of partial isometries $v_f$ coming from $\PT(X)$. This
intermediate position between the separable and nonseparable worlds makes
it plausible that prime ideals of $\cstu(X)$ might still be primitive
under suitable geometric assumptions on $X$. Our paper culminates with applications of the theory built herein to this problem. We briefly describe them now.

In an arbitrary topological space, a nonempty closed subset is called \emph{irreducible} if
it cannot be written as a nontrivial union of two closed subsets (see Definition
\ref{Definition.Irreducible.in.Orb}). In case the u.l.f.\ metric space $X$ has Yu's
property A,\footnote{In terms of uniform Roe algebras, a u.l.f.\ metric space has
\emph{property A} if and only if $\cstu(X)$ is a nuclear $\mathrm C^*$-algebra (see
\cite[Theorem 5.5.7]{BrownOzawa}.} we show that irreducible subsets of
$\calOrb(\partial X)$ are in correspondence with the prime ideals of $\cstu(X)$
(see  Proposition \ref{Prop.Irrd.Gives.Prime}). Moreover, still in the presence of
property A, we show that if irreducible subsets of the orbit space of $\partial X$
are the closure of singletons, then prime ideals in $\cstu(X)$ must be primitive
(see Theorem \ref{Thm.IrredCLosureofOrbImplisPrimeIsPrimitive}).  For Hausdorff
spaces, it is immediate that irreducible subsets are the closure of
singletons, i.e., they are simply singletons.  However, for non-Hausdorff spaces, this is far from the case in general.  The next theorem shows that, in our orbit spaces, $T_1$-ness is already enough.

\begin{theorem*}[See Theorem \ref{Cor.T1.Prime=Primitive.InTheText} below]
    Let $X$ be a u.l.f.\ metric space for which $\calOrb(\partial X)$ is $T_1$. Then every  irreducible subset of $\calOrb(\partial X)$ is the closure of a singleton.  In particular,  the prime and primitive ideals of $\cstu(X)$ coincide.
\end{theorem*}

The theorem above will be obtained as a corollary of  Theorems \ref{ThmCharacterizationOrbClosed.IntheText} and \ref{CorClosedQuasiOrbitsCoarseDisjointUnionSing.Inthetext} together with a general result about  inverse semigroups acting on compact Hausdorff spaces (see Theorem \ref{Thm.Irr.General.Inv.Semi.Group}).

The next theorem provides examples of u.l.f.\ metric spaces whose orbit spaces are not even $T_1$ but still their irreducible subsets are the closure of singletons.

\begin{theorem*}
    [See Theorem \ref{Thm.M_k.Irreducible.Inthetext} below]
Let $X$ be a u.l.f.\ metric space which coarsely embeds into $\bigsqcup_{i=1}^k\cM_{i}\subseteq \cM$
   for some
$k\in\N$. Then every  irreducible subset of $\calOrb(\partial X)$ is the closure of a singleton. In particular, the prime and primitive ideals of $\cstu(X)$ coincide.
\end{theorem*}

\begin{remark}
The phenomenon appearing in Theorem~\ref{ThmMapFromOrbToPrimIsNice.InTheText}
parallels the general philosophy in the theory of $\mathrm C^*$-algebras arising
from dynamical systems, where primitive ideals are often described in
terms of quasi-orbits of the underlying action; see for instance
\cite{Renault1991,Clark2005,SimsWilliams2016}. In many of these works,
second countability plays an essential role, typically through
arguments going back to Green \cite{Green1978}, which imply that every
irreducible closed invariant set is the closure of a singleton. This is
precisely the step that allows one to pass from quasi-orbit data to
primitive ideals.

In our setting, however, the action takes place on the large compact
space $\beta X$, which is not second countable as our metric space $X$ of interest is infinite. From the
groupoid viewpoint, the inverse semigroup $\PT(X)$ acting on $\beta X$
gives rise to the semi-direct product groupoid
\(
G(X)=\PT(X)\ltimes \beta X,
\)
whose unit space is $\beta X$ and whose reduced $\mathrm C^*$-algebra coincides
with the uniform Roe algebra $\mathrm C_u^*(X)$. Since the arguments based on
Green's lemma are not directly available in this non-second countable
setting, we do not know in general whether every irreducible closed invariant subset of $\calOrb(\partial X)$ is the closure of a singleton. This is
closely related to our open question of whether every prime ideal of
$\mathrm C_u^*(X)$ is primitive. To overcome these technical hurdles, we develop our arguments in this paper by working directly with the dynamical properties of the action
\(
\PT(X)\curvearrowright \beta X.
\)
\end{remark}

\subsection{Paper organization}
The paper was organized so that uniform Roe algebras only appear in Section \ref{SectionGNSRep}. Sections \ref{SectionBasicProp} and  \ref{SectionPseudo} are purely about dynamics and  make no usage of   $\mathrm C^*$-algebras. Section \ref{SectionHigson} contains our results about Higson functions.

At last, as this paper should be of interest to researches interested in different fields such as dynamics, large scale geometry, and operator algebras, we chose to present it so as
not to exclude  any of these groups. We apologize to the experts in all these areas for any possible pedantry.

  \section{The orbits and quasi-orbits of $(\beta X,\PT(X))$}\label{SectionBasicProp}

We start setting some notation which will be used throughout. Given a u.l.f.\ metric space $X$ and $F\subseteq \beta X$, the \emph{orbit saturation of $F$} is
\[\Orb(F)=\bigcup_{\xi\in F}\Orb(\xi).\]
Given $\xi\in \beta X$, we define
\[ \PT(X)_\xi=\{f\in \PT(X)\colon \dom(f)\in \xi\}.\]
So, $\PT(X)_\xi$ are precisely the partial translations of $X$ for which writing $\overline f(\xi)$ makes sense.

\subsection{Basic properties}
We start with some simple observations: (1) we provide characterization of orbit equivalence, (2) a characterization of being in the closure of an orbit, and (3) we show that the orbits of elements in $\beta X$ are always countable. Besides  serving  as  a warm up, the fact that orbits are countable will also be used later in Section \ref{SubsectionRestrictionGNSHigFunc}. These observations are  consequences of the fact that our metric spaces are u.l.f.. Such spaces can always be partitioned into finitely many subsets which are arbitrarily well separated. Precisely,   the following is a well-known property of u.l.f.\ metric spaces: if $(X,d)$ is u.l.f., then, for all $r>0$, there is a partition
 \[X=X_1\sqcup\ldots \sqcup X_k\]
 such that each $X_i$ is \emph{$r$-separated}, i.e., $d(x,y)>r$ for all $i\in \{1,\ldots, k\}$ and all distinct $x,y\in X_i$. As a consequence, if $E\subseteq X\times X$ is such that $\sup_{(x,y)\in E}d(x,y)<\infty$, then there is a partition
 \[E=E_1\sqcup \ldots \sqcup E_m\]
such that each $E_i$ is the graph of a partial translation of $X$.

 The next proposition provides a concrete description of orbit equivalence in terms of coarse neighborhoods.

\begin{proposition}\label{PropOrbitsAreCountable}
    Let $X$ be a u.l.f.\ metric space and  $\xi,\eta\in \beta X$. The following holds.
    \begin{enumerate}
    \item\label{PropOrbitsAreCountable.Item1}   $\xi\in \Orb(\eta)$  if and only if there is $r>0$ such that $N_r(A)\in \eta$ for all $A\in \xi$.
\item\label{PropOrbitsAreCountable.Item1.5} $\eta\in \COrb(\xi)$ if and only if for all   $ A\in \eta$ there is    $r>0$ such that  $ N_r(A)\in \xi$.
\item\label{PropOrbitsAreCountable.Item2} $\Orb(\xi)$ is countable.
    \end{enumerate}
\end{proposition}

\begin{remark}\label{Remark.parallel.rel}
 Notice that if $\xi,\eta\in  \beta X$ are so that there is $r>0$   such that $N_r(A)\in \eta$ for all $A\in \xi$, then $N_{r}(B)\in \xi$ for all $B\in \eta$. Indeed, this follows from the fact that $\xi$ and $\eta$ are ultrafilters. So, the latter condition in Proposition \ref{PropOrbitsAreCountable}\eqref{PropOrbitsAreCountable.Item1}  defines an equivalence relation on $\beta X$ (which is obviously necessary for it to be equal to the orbit equivalence relation). While we will not explicitly rely on this characterization, this is relevant  since it shows that   orbit equivalence is the same as   \emph{parallel equivalence}, defined in \cite{Protasov2003MatStud} and further studied in  \cite{Protasov2005TopAndAppl,Protasov2008TopAndAppli,Protasov2014TopAppl}.
\end{remark}

\begin{proof}[Proof of Proposition \ref{PropOrbitsAreCountable}]
\eqref{PropOrbitsAreCountable.Item1} The forward implication is immediate. For the backward, let $r>0$ be as above. As $X$ is u.l.f., the discussion preceding the proposition gives us a partition
 \[X=X_1\sqcup\ldots \sqcup X_k\]
 such that each $X_i$ is  $2r$-separated. As $\xi$ is an ultrafilter, there is $i\in \{1,\ldots, k\}$ such that
$X_i\in \xi$.  As $X$ is u.l.f., there exists $m \in \mathbb N$ such that $|N_r(x)| \leq m$ for all $x \in X$. So, there are $B_1,\ldots, B_m\subseteq X$ such that
\begin{itemize}
      \item $N_r(X_i)=B_1\cup\ldots\cup B_m$ and
      \item $|B_j\cap N_r(x)|\leq 1$ for all $x\in X_i$ and all $j\in \{1,\ldots, m\}$.
\end{itemize}
As $X_i\in \xi$, the   hypothesis of the item give that  $N_r(X_i)\in \eta$. Hence, as $\eta$ is an ultrafilter, there is   $j\in \{1,\ldots, m\}$ with $B_j\in \eta$.

Let $f\colon X_i\to B_j$ be the bijection such that $f(x)\in N_r(x)$ for all $x\in X_i$. So, $f $ is a partial translation. We are left to observe that $\overline f(\xi)=\eta$. For that it is enough to show that if   $B\in \eta$ then there is $A\in \xi$ such that $f(A)\subseteq B$. Fix such $B$; without loss of generality, assume $B\subseteq B_j$. By Remark \ref{Remark.parallel.rel}, $N_r(B)\in \xi$. So, $X_i\cap N_r(B)\in \xi$. Clearly, $f(X_i\cap N_r(B))=B$.

\eqref{PropOrbitsAreCountable.Item1.5}    Suppose $\eta\in\COrb(\xi)$ and let $A\in \eta$. Then, $\overline A\cap \Orb(\xi)\neq \emptyset$ and we can pick $f\in \PT(X)_\xi$ such that $\overline f(\xi)\in \overline A$. So, there is $B\in \xi$ such that $f(B)\subseteq A$. Then, letting
    \[r=\sup_{x\in \dom(f)}d(x,f(x)),\]
    we have $B\subseteq N_r(A)$, which gives $N_r(A)\in \xi$.

The backwards implication is an easy consequence of the u.l.f.-ness of $X$. Indeed, let $A\in \eta$ and let us show $\overline A\cap \Orb(\xi)\neq\emptyset$. By the hypothesis, there is $r>0$ such that $N_r(A)\in \xi$. Let $f\colon N_r(A)\to A$ be a map such that
\[\sup_{x\in N_r(A)}d(x,f(x))\leq r.\]
As $X$ is u.l.f., there is a partition
\[N_r(A)=B_1\sqcup\ldots\sqcup B_k\]
such that $f\restriction B_i$ is injective for all $i\in\{1,\ldots, k\}$. So, each $f\restriction B_i$ is a partial translation of $X$. As $\xi$ is an ultrafilter and $N_r(A)\in \xi$, there is $i\in \{1,\ldots, k\}$ such that $B_i\in \xi$. So, if $g=f\restriction B_i$, we have $\overline g(\xi)\in \overline  A$ and we are done.

\eqref{PropOrbitsAreCountable.Item2} Fix $\xi\in \beta X$. For each $n\in \N$, let $E(n)=\{(x,y)\in X\times X\colon d(x,y)\leq n\}$ and let
\[E(n)=E_1\sqcup \ldots \sqcup E_{m(n)}\]
be a partition such that  each $E_i$ is the graph of a partial translation $f^n_i$ of $X$. Let us notice that
\[\Orb(\xi)=\left\{\overline{f_i^n}(\xi)\colon n\in\N, \ i\in \{1,\ldots, m(n)\}, \ f^n_i\in \PT(X)_\xi\right\}.\]

Let $f$ be a partial translation of $X$ with $\dom(f)\in \xi$ and let us show that $\overline f(\xi)=\overline{f^n_i}(\xi)$ for some $n\in\N$ and $i\in \{1,\ldots, m(n)\}$. For that, pick $n\in \N$ large enough so that $d(x,f(x))\leq n$ for all $x\in \dom(f)$ and let \[M=\left\{i\in \{1,\ldots, m(n)\}\colon \dom(f^n_i)\in \xi\right\}.\] Since the graphs of $(f^n_{i})_{i\in \{1,\ldots, m(n)\}}$ form a partition of  $E$, it is clear that $M$ is nonempty. Let \[A=\dom(f)\cap\bigcap_{i\in M}\dom(f^n_i),\]
so, $A\in \xi$. For each $i\in M$, let
\[A_i=\{x\in A\colon f(x)=f^n_i(x)\}.\]
So, $A=\bigsqcup_{i\in M}A_i$. As $\xi$ is an ultrafilter, there must be $i\in M$ such  that $A_i\in \xi$. Since $f\restriction A_i=f^n_i\restriction A_i$ and $A_i\in \xi$, it follows that  $\overline f(\xi)=\overline{f^n_i}(\xi)$.
\end{proof}

It will be important for us to look at the orbit saturation of open subsets of $\beta X$. Moreover, the quasi-orbit saturation will also play an important role: given a u.l.f.\ metric space $X$ and $F\subseteq \beta X$, the \emph{quasi-orbit saturation of $F$} is
\[\qOrb(F)=\bigcup_{\xi\in F}\qOrb(\xi).\]
\begin{proposition}\label{PropOrbUOpenAndEqualqOrbU}
    Let $X$ be a u.l.f.\ metric space and $U\subseteq \beta X$ be open. Then $\Orb(U)$ is open and
    \[\Orb(U)=\qOrb(U)= \{\xi\in \beta X\colon \exists \eta \in U,\ \eta\in \COrb(\xi)\}.\]
\end{proposition}

\begin{proof}
The fact that $\Orb(U)$ is open is almost immediate and, moreover, it will also be a subproduct of the proof of the equalities above. So, we start proving these subset equalitites. Since the inclusions ``$\subseteq$'' are immediate, it is enough to notice that
 \[\{\xi\in \beta X\colon \exists \eta \in U,\ \eta\in \COrb(\xi)\}\subseteq \Orb(U).\]
 Fix $\xi$ in the set on the left-hand side above and $\eta\in U$ such that $\eta\in \COrb(\xi)$. As $U$ is open, there is $A\subseteq X$ such that $\eta\in \overline A\subseteq U$. Hence, as $\eta\in \COrb(\xi)$, there is $f\in \mathrm{PT}(X)_\xi$ such that  $\mathrm{im}(f)\subseteq A$. In particular,   \[\overline{\mathrm{im}(f)}\subseteq \overline A\subseteq U.\]
Therefore, since $\overline{\dom(f)}\subseteq \Orb(\overline{\mathrm{im}(f)})$, $\overline{\dom(f)}$ is an open subset containing $\xi$ which is itself contained in $\Orb(U)$. So, $\Orb(U)$ is open and  $\xi\in \Orb(U)$ as desired.
\end{proof}

\begin{corollary}\label{CorQuotientOpen}
    Let $X$ be a u.l.f.\ metric space. Both quotient maps
    \[\beta X\to \calOrb(\beta X)\ \text{ and }\ \beta X\to \qcalOrb(\beta X)\]
    are open. In particular, the quotient map
    \[\calOrb(\beta X)\mapsto \qcalOrb(\beta  X)\]
    is also open.
\end{corollary}

\begin{proof}
In order to notice that the quotient map $\beta X\to \calOrb(\beta X)$ (resp., $\beta X\to \qcalOrb(\beta X)$) is open, it is equivalent to show that $\Orb(U)$ (resp., $\qOrb(U)$) is open  for all open $U\subseteq \beta X$. This is the case from Proposition  \ref{PropOrbUOpenAndEqualqOrbU}.
\end{proof}

The next proposition shows that, while  orbit closures do not form a partition of $\beta X$,   we still have that any given  $\COrb(\xi)$  and $\COrb(\eta)$ are either disjoint or comparable with respect to inclusion. We point out that   this was already known: it was proven in \cite[Theorem 3.4]{Protasov2014TopAppl} using \cite[Theorem 3.40]{HindmanStrauss2012Book}. However, since this will be important for some of our main results (e.g., Theorem \ref{Cor.T1.Prime=Primitive.InTheText}),  we chose to present a self-contained proof here for the readers convenience.

\begin{proposition}\label{PropClosureOrbDisjointOrComparable}
     Let $X$ be a u.l.f.\ metric space and $\xi,\eta\in \beta X$. If $\COrb(\xi)\cap \COrb(\eta)\neq \emptyset$, then either $\xi\in \COrb(\eta)$ or $\eta\in \COrb(\xi)$. Equivalently, either $\COrb(\xi)\subseteq \COrb(\eta)$ or $\COrb(\eta)\subseteq  \COrb(\xi)$.
\end{proposition}

We will need the following  lemma for the proof of Proposition \ref{PropClosureOrbDisjointOrComparable}.

\begin{lemma}\label{LemmaSeparatingSet}
    Let $X$ be a u.l.f.\ metric space and $\xi,\eta\in \beta X$. Suppose  there is a partition $X=X_\xi\sqcup X_\eta$ such that   $f^{-1}(X_\eta)\not\in \xi$ and $f^{-1}(X_\xi)\not\in \eta$ for all partial translations $f$ of $X$. Then, $\COrb(\xi)\cap \COrb(\eta)=\emptyset$.
\end{lemma}

\begin{proof}
    Suppose towards a contradiction that there is $\zeta\in \COrb(\xi)\cap \COrb(\eta)$. Since $\zeta$ is an ultrafilter, either $X_\xi\in \zeta$ or $X_\eta\in \zeta$. Suppose the former happens. Then, as $\zeta\in \COrb(\eta)$, there is a partial translation $f$ of $X$ with $\dom(f)\in \eta$ and such that $\mathrm{im}(f)\subseteq X_\xi$. But then $f^{-1}(X_\xi)=\dom(f)\in \eta$; contradiction. In case $X_\eta\in \zeta$, the contradiction is obtained analogously.
\end{proof}

\begin{proof}[Proof of Proposition  \ref{PropClosureOrbDisjointOrComparable}]
    Suppose $\xi\not\in \COrb(\eta)$ and $\eta\not\in \COrb(\xi)$ and let us show $\COrb(\xi)\cap \COrb(\eta)= \emptyset$.  The conditions on $\xi$ and $\eta$ give us $A\in \xi$ and $B\in \eta$ such that
    \begin{equation}\label{EqPropFechoOrbs}
        X\setminus N_n(A)\in \eta \ \text{ and }\ X\setminus  N_n(B) \in \xi
    \end{equation}
    for all $n\in\N$ (Proposition \ref{PropOrbitsAreCountable}\eqref{PropOrbitsAreCountable.Item1.5}).

    \begin{claim}There are decreasing sequences $(A_n)_n$ and $(B_n)_n$ of subsets of $X$ such that
    \begin{enumerate}
        \item $A_n\in \xi$ and $B_n\in \eta$ for all $n\in\N$, and
        \item $N_n(A_n)\cap N_m(B_m)=\emptyset$ for all $n,m\in\N$.
    \end{enumerate}
\end{claim}

\begin{proof}
We proceed by induction. Let $A_1=A\setminus N_2(B)$ and $B_1=B\setminus N_2(A)$. Clearly, $N_1(A_1)$ and $N_1(B_1)$ are disjoint and, since $A\in \xi$ and $B\in \eta$, \eqref{EqPropFechoOrbs} implies that $A_1\in \xi$ and $B_1\in \eta$. If $(A_i)_{i=1}^n$ and $(B_i)_{i=1}^n$ have been chosen satisfying the conclusion of the lemma, we then let  $A_{n+1}=A_n\setminus N_{2n+2}(B)$   and $B_{n+1}=B_n\setminus N_{2n+2}(A)$.  As $A_n\in \xi$ and $ B_n\in \eta$, it follows again by   \eqref{EqPropFechoOrbs}  that $A_{n+1}\in \xi$ and $B_{n+1}\in \eta$. It is also clear that $N_{n+1}(A_{n+1})\cap N_i(B_i)=\emptyset$ and $N_i(A_{i})\cap N_{n+1}(B_{n+1})=\emptyset$ for all $i\leq n$. \end{proof}

Let now $C=\bigcup_{n\in\N}N_n(A_n)$. Let us notice that the partition $X=C\sqcup (X\setminus C)$ satisfies the condition in Lemma  \ref{LemmaSeparatingSet}. For that, let $f$ be a partition of $X$. So,   \[r=\sup_{x\in \dom(f)}d(x,f(x))<\infty.\] Pick $n>r$. Since $d(A_n,X\setminus C)\geq n$, the preimage $f^{-1}(X\setminus C)$ is disjoint from $A_n$. As $A_n\in \xi$,  $f^{-1}(X\setminus C)\not\in \xi$. On the other hand, since $\bigcup_{n\in\N}N_n(B_n)\subseteq X\setminus C$, an analogous argument shows that $f^{-1}(C)\not\in \xi$. As the partial translation $f$ was arbitrary, Lemma \ref{LemmaSeparatingSet} implies that $\COrb(\xi)\cap \COrb(\eta)= \emptyset$.
\end{proof}

\subsection{Closed orbits}
We now  characterize when the orbits of $(\beta X,\PT(X))$ are closed (Theorem \ref{ThmCharacterizationOrbClosed.IntheText}). Among other things, this will be used  to characterize the coarse containment of     $\{(j^2,n^2)\colon n\geq j\}$ in terms of dynamics (Theorem \ref{CorClosedQuasiOrbitsCoarseDisjointUnionSing.Inthetext}).

\begin{theorem} 
    Let $X$ be a u.l.f.\ metric space and $\xi\in \beta X$. The following are equivalent.
    \begin{enumerate}
\item\label{Item1ThmCharacterizationOrbClosed} $\mathrm{Orb}(\xi)$ is closed.
\item\label{Item2ThmCharacterizationOrbClosed}   $\mathrm{Orb}(\xi)$ is finite.
\item\label{Item3ThmCharacterizationOrbClosed}
        There is  $r>0$ such that
    \[\lim_{x,\xi}d\left(N_r(x),X\setminus N_r(x)\right)=\infty.\]\end{enumerate}\label{ThmCharacterizationOrbClosed.IntheText}
\end{theorem}

\begin{proof}
It is straightforward to check the result holds if $\xi\in X$. So, we assume $\xi\in \partial X$ throughout.

\eqref{Item3ThmCharacterizationOrbClosed}$\Rightarrow$\eqref{Item2ThmCharacterizationOrbClosed}:
Suppose first that there is $r>0$ as in the statement and let us show that $\mathrm{Orb}(\xi)$ is finite. Firstly, by our choice of $r$, we can fix   $A\in \xi$ such that
\[d(N_r(x),X\setminus N_r(x))>2r\ \text{ for all }\ x\in A.\]
So, for all $x,y\in A$, either $N_r(x)=N_r(y)$ or $N_r(x)\cap N_r(y)=\emptyset$. As $X$ is u.l.f.,
 \[N=\sup_{x\in A} |N_r(x)|<\infty.\]
 Therefore, since $\xi$ is an ultrafilter, replacing $A$ by a subset of it still in $\xi$, we can assume furthermore that
$(N_r(x))_{x\in A}$ are disjoint.  By the definition of $N$, for each $x\in A$,  we can enumerate  \[N_r(x) =\{y^x_1,\ldots, y^x_{j(x)}\},\] where $j(x)\leq N$ and $y^x_1=x$.

\begin{claim}
$    \mathrm{Orb}(\xi)$ has at most $N$ elements.
\end{claim}

\begin{proof}
    For each $j\leq N$, let
    \[A_j=\bigcup_{x\in A}\{y^x_j\},\]
    where $\{y^x_j\}$ is interpreted as the empty set if $j>j(x)$. So, $A_1=A$ and
    \begin{equation}\label{Eqfinb0}\bigsqcup_{j\leq N}A_j=\bigsqcup_{x\in A}N_r(x).\end{equation} For each $j\leq N$, consider  the partial translation    $g_j\colon  A_j\to \mathrm{im}(g_j) $ such that $g_j(y^x_j)=y^x_1$ for all $x\in X$ so that $y^x_j$ is defined. We shall now notice that
    \[\mathrm{Orb}(\xi)=\left\{\overline{g_j^{-1}}(\xi)\colon j\leq N\right\}.\]

Let $f$ be a partial translation on $X$ with $\mathrm{dom}(f)\in \xi$. As $A\in \xi$, we can assume $\dom(f)\subseteq A$. Let
\[s=\sup_{x\in \dom(f)}d(x,f(x)).\] Since
\[\lim_{x,\xi}d(N_r(x),X\setminus N_r(x))=\infty,\] we can shrink the domain of $f$ if needed and assume that
\[d(N_r(x),X\setminus N_r(x))>s\]
for all  $x\in \dom(f)$. By the definition of $s$, this implies that
\begin{equation}\label{Eqfinb}
f(x)\in N_r(x)\ \text{ for all }\ x\in \dom(f).
\end{equation}
By \eqref{Eqfinb0} and \eqref{Eqfinb}, we have a partition
\[\dom(f)=f^{-1}(A_1)\sqcup\ldots \sqcup f^{-1}(A_N).\]
As $\xi$ is an ultrafilter, there must be $j\leq N$ such that $f^{-1}(A_j)\in \xi$. But then, as $f\restriction f^{-1}(A_j)=g_j^{-1} \restriction f^{-1}(A_j)$, we conclude that $\overline{f}(\xi)=\overline{g_j^{-1}}(\xi)$.
\end{proof}

The previous claim gives us the implication \eqref{Item3ThmCharacterizationOrbClosed}$\Rightarrow$\eqref{Item2ThmCharacterizationOrbClosed}. The implication \eqref{Item2ThmCharacterizationOrbClosed}$\Rightarrow$\eqref{Item1ThmCharacterizationOrbClosed} is immediate.

\eqref{Item1ThmCharacterizationOrbClosed}$\Rightarrow$\eqref{Item3ThmCharacterizationOrbClosed}: Suppose now that there is no $r>0$ as in the statement of the theorem and let us show that $\mathrm{Orb}(\xi)$ is not closed. Since no such $r$ exists and $\xi$ is nonprincipal,   there are a decreasing sequence $(\N_j)_{j\in\N}$ of infinite subsets of $\N$ with $\N_1=\N$ and $\bigcap_j\N_j=\emptyset$, and sequences  $((x_n(j))_{n\in \N_j})_{j\in\N}$ in $X$ such that, letting for each $n\in\N$,
\[F_n=\{j\in \N\colon n\in \N_j\},\]
we have
\begin{enumerate}
      \item\label{Item1.orbitnotclosed} $\{x_n(1)\colon n\in \N_j\}\in \xi$ for all $j\in\N$,
        \item\label{Item4.orbitnotclosed} $\sup_{n\in \N_j}d(x_n(1),x_n(j))<\infty$ for all $j\in\N$,
    \item\label{Item2.orbitnotclosed} $x_n(i)\neq x_{n}(j)$ for all $n\in\N$ and all distinct $i,j\in F_n$,
     \item\label{Item2.5orbitnotclosed} the subsets $(\{x_n(j)\colon j\in F_n\})_{n\in\N}$ are disjoint, and
     \item\label{Item3.orbitnotclosed} $d(\{x_n(j)\colon j\in F_n\},\{x_m(j)\colon j\in F_m\})\to \infty$ as $n,m\to \infty$ with $n\neq m$.
     \end{enumerate}
Notice that, since $\bigcap_j\N_j=\emptyset$, each $F_n$ is finite.

We  now introduce a system of coordinates to facilitate our visualization and writing. Let
\[X'=\{x_n(j)\colon j\in\N,  \ n\in \N_j\}\]
and  identify $X'$ with the subset of $\N\times \N$ given by the injection
\begin{align}\label{EqIdentificationTriangle}
 X'&\to \N\times \N\\
x_{n}(j)&\mapsto (j,n).\notag
\end{align}
Notice that the fact that this map is an injection follows from \eqref{Item2.orbitnotclosed} and \eqref{Item2.5orbitnotclosed} above.
This identification allows us to talk about vertical and horizontal sections of subsets of $X'$. Let $\pi\colon\N\times \N\to \N$ denote the projection onto the second coordinate. As $\{x_n(1)\colon n\in \N\}\in \xi$, the restriction of this projection to $\{x_n(1)\colon n\in \N\}$ (a
vertical section of $X'$) allows us to identity $\xi$ with an ultrafilter on $\N$.

Fix a nonprincipal ultrafilter $\zeta$ on $\N$ (one can just use $\xi$ once again). Inspired by the product of ultrafilters, we now define an ultrafilter $\eta$ on $X$ by letting
\begin{equation*}
    A\in \eta \ \Longleftrightarrow\ \Big\{j\in \N\colon \pi\big(A\cap \{x_{n}(j)\colon n\in \N_j\}\big)\in \xi \Big\}\in \zeta .
    \footnote{The fact $\eta$ is indeed an ultrafilter follows exactly as why the product of ultrafilters is an ultrafilter since $\eta$ can be seen as the product between $\xi$ and $\eta$.}
\end{equation*}

    \begin{claim}
        For each $A\in \eta$, we have
    \[\lim_{n,\xi}|A\cap\{x_{n}(j)\colon j\in F_n\}|=\infty.\]
    \end{claim}

    \begin{proof}
Fix $A\in \eta$. We must show that
\begin{equation}\label{EqSubsetkmanyelements}\Big\{n\in \N\colon |A\cap\{x_{n}(j)\colon j\in F_n\}|>k\Big\}\in \xi
\end{equation}
for all $k\in\N$. Fix $k\in\N$. As $\{j\colon \pi(A\cap \{x_{n}(j)\colon n\in \N_j\})\in \xi \}$ is in the nonprincipal ultrafilter $\zeta$, this set must be infinite. In particular, it has at least $k$ many elements, say $j_1,\ldots, j_k$. As $\xi$ is a filter, we have
\[\bigcap_{i=1}^k\pi(A\cap \{x_{n}(j_i)\colon n\in \N_{j_i}\})\in \xi.
\]
Since the horizontal fibers of each $n$ in the subset above have at least $k$ elements, this subset is contained in the subset in \eqref{EqSubsetkmanyelements}. This proves the latter is indeed in $\xi$.
    \end{proof}

    The previous claim implies that $\eta$ cannot be in $\mathrm{Orb}(\xi)$.
Indeed, suppose $f$ is a partial translation of $X$ such that $\overline{f}(\xi)=\eta$. By \eqref{Item1.orbitnotclosed} and \eqref{Item3.orbitnotclosed} above, going to a subsequence if necessary, we can assume that \[f(x_n(1))\in \{x_n(j)\colon j\in F_n\}\] for all $n\in\N$. In particular, the horizontal sections of $f(\{x_n(1)\colon n\in\N\})$ are singletons. On the other hand, as $\overline{f}(\xi)=\eta$, $f(\{x_n(1)\colon n\in\N\})$ is in  $\eta$. However, since the horizontal sections of $f(\{x_n(1)\colon n\in\N\})$ are singletons, this set does not satisfy the conclusion of the previous claim; contradiction.

We are left to notice that $\eta$ is in the closure of $\mathrm{Orb}(\xi)$. For that, pick $A\in \eta$ and let us find a partial translation $f$ of $X$ such that $\mathrm{dom}(f)\in \xi$ and $\mathrm{im}(f)\subseteq A$. As $A\in \eta$, it follows from the definition of $\eta$ that there must be $j\in\N$ such that $\pi(A\cap \{x_{n}(j)\colon n\in \N_j\})\in \xi$. Let now  $f$ be the map from \[\{1\}\times \pi(A\cap \{x_{n}(j)\colon n\in \N_j\})\subseteq X'\to     A\cap \{x_{n}(j)\colon n\in \N_j\}\] given by  letting $f(1,m)=x_m(j)$ for all $m\in\N$ so  $(1,m)$ is in the defined domain of $f$ (under the identification \eqref{EqIdentificationTriangle}, $(1,m)=x_m(1)$). It is clear that  $\mathrm{dom}(f)\in \xi$ and $\mathrm{im}(f)\subseteq A$. The fact that $f$ is  a partial translation follows from \eqref{Item4.orbitnotclosed}.
\end{proof}

 We finish this subsection with an application of the ideas in Theorem \ref{ThmCharacterizationqOrbClosed}. In Section \ref{SubsectionIrreducibleSet}, this will be used to explain the difficulty in   solving an important problem about irreducible sets, see Problem \ref{Prob.Irred.FORB} and  Remark \ref{RemarkIrreducibleCORB}  below.

\begin{proposition}\label{prop:uncountably_many_disjoint_orbits}
Let $X$ be a u.l.f.\ metric space containing a coarse copy of  $\cM_2$. Then, there is $\xi\in \partial X$ and a family $(\xi_i)_{i\in I}$ in $  \COrb(\xi)$, with $|I|=2^{\aleph_0}$,  such that $ \COrb(\xi_i)\cap \COrb(\xi_j)=\emptyset$ for all distinct $i,j\in I$.
\end{proposition}

\begin{proof}
As $\cM_2$ is bijectively coarsely equivalent to $\{(m^2, n^2) \in \mathbb{N}^2 \colon n\geq m\}$, for simplicity of notation, assume \[X= \{(m^2, n^2) \in \mathbb{N}^2 \colon n\geq m\}\] and let $\xi$ be a nonprincipal ultrafilter on $X$ containing
$A=\{(1,n^2)\colon n\in\N\}$. Let \( (A_i)_{i \in I} \) be a family of infinite, pairwise almost disjoint subsets of $\N$ with $|I|=2^{\aleph_0}$.\footnote{Here, \emph{almost disjoint} means that $A_i\cap A_j$ is finite for all distinct $i,j\in I$. } Proceeding as in the proof of Theorem \ref{ThmCharacterizationOrbClosed.IntheText}, for each $i\in I$, we can find an ultrafilter $\xi_i$ on $X$ such that
\begin{equation}\label{Eq.IAmNotINGranada}
\{(m^2, n^2) \in X \colon m \in A_i,\ m > k \}\in \xi_i
\end{equation}
for all $k\in\N$. Then, each \( \xi_i \) is in the closure of the orbit of \( \xi \), since there are partial translations that map the vertical set \( A \) into these subsets. On the other hand, as any partial translation of $X$ must be the identity on $  \{(m^2, n^2) \in  X \colon  m>k\}$ for $k\in\N$ large enough, it follows from \eqref{Eq.IAmNotINGranada}  that $\Orb(\xi_i)=\{\xi_i\}$ for all $i\in I$. In particular, $\COrb(\xi_i)$ are pairwise disjoint as desired.
\end{proof}

\subsection{Closed quasi-orbits}
In this subsection, we turn our focus to quasi-orbits. Besides characterizing when quasi-orbits are closed  (Theorem  \ref{ThmCharacterizationqOrbClosed}), we also  characterize the coarse type of   u.l.f.\ metric spaces containing a coarse copy of   $\cM_2$ in terms of quasi-orbits  (Theorem \ref{CorClosedQuasiOrbitsCoarseDisjointUnionSing.Inthetext}).

Our next theorem characterizes when quasi-orbits are closed.

\begin{theorem}\label{ThmCharacterizationqOrbClosed}
    Let $X$ be a u.l.f.\  metric space and $\xi \in \beta X$. The following are equivalent
    \begin{enumerate}
        \item \label{PropCharacterizationqOrbClosedItem1} $\qOrb(\xi)$ is closed.
                \item\label{PropCharacterizationqOrbClosedItem2} $\qOrb(\xi)=\COrb(\xi)$.
        \item\label{PropCharacterizationqOrbClosedItem3}    $\overline{\Orb}(\xi)$ is minimal among closed invariant subsets.
        \item \label{PropCharacterizationqOrbClosedItem4} For all $A\in \xi$ there is $r>0$ such that $N_s(X\setminus N_r(A))\notin \xi$ for all $s>0$.
    \end{enumerate}
\end{theorem}

The equivalence between \eqref{PropCharacterizationqOrbClosedItem2} and \eqref{PropCharacterizationqOrbClosedItem3} above was known for groups (see \cite[Theorem 5.1(i)]{Protasov2008TopAndAppli}).

\begin{proof}[Proof of Theorem \ref{ThmCharacterizationqOrbClosed}]
Suppose  $\qOrb(\xi)$ is closed. As $\Orb(\xi)\subseteq \qOrb(\xi)\subseteq \COrb(\xi)$, we have $\overline{\Orb}(\xi)=\qOrb(\xi)$. So, \eqref{PropCharacterizationqOrbClosedItem1}$\Rightarrow$\eqref{PropCharacterizationqOrbClosedItem2}.

\eqref{PropCharacterizationqOrbClosedItem2}$\Rightarrow$\eqref{PropCharacterizationqOrbClosedItem4}:
Let $A\in \xi$ and suppose towards a contradiction that  for all $r>0$ there is $s>0$ such that $N_s(X\setminus N_r(A))\in \xi$. Hence, for all $r>0$, there is a partial translation $f_r$ of $X$ with $\dom(f_r)\in \xi$ and $\mathrm{im}(f_r)\subseteq X\setminus N_r(A)$. Let $\eta_r=f_r(\xi)$. Let $\zeta$ be a nonprincipal ultrafilter on $\N$ and define an ultrafilter
$\eta$   on $X$ such that, for all $B\subseteq X$,
\[B\in \eta\ \Leftrightarrow \ \{r\in \N\colon B\in\eta_r \}\in \zeta.\]

Let us show that $\eta\in \COrb(\xi)\setminus \qOrb(\xi)$. We start with $\eta\in \COrb(\xi)$. For that, fix   $B\in \eta$ and notice that $\overline B\cap \Orb(\xi)$ is nonempty. Indeed, by the definition of $\eta$,  there must be $r\in \N$ such that $B\in \eta_r$. So, $f_r$ is a partial translation which takes $\xi$ into $\overline B$. We are left to show that $\eta\not\in   \qOrb(\xi)$. As   $\eta\in \COrb(\xi)$, this is equivalent to showing that $\xi\not\in \COrb(\eta)$. For that,  as $\xi\in \overline A$, it is enough to observe that $\overline A\cap \Orb(\eta)=\emptyset$.  In other  words, we shall prove  that $f(B)\not\subseteq A$ for all partial translations $f$ of $X$ and all $B\in \eta$. Fix $B\in \eta$ and a partial translation $f$ of $X$. Let \[t=\sup_{x\in \dom(f)}d(x,f(x)).\] As $B\in \eta$ and $\eta $ is nonprincipal, there is $r>t$ such that $B\in \eta_r$. By our choice of $f_r$, we have that $\mathrm{im}(f_r)\subseteq X\setminus N_r(A)$. So, $B\cap  (X\setminus N_r(A))\in \eta_r$. Since $r>t$, \[f\big(B\cap ( X\setminus N_r(A))\big)\not\subseteq A.\] So, $f(B)\not\subseteq A$ and we are done.

\eqref{PropCharacterizationqOrbClosedItem4}$\Rightarrow$\eqref{PropCharacterizationqOrbClosedItem3}: As the closure of orbits are invariant, it is sufficient to show that given an   $\eta\in \COrb(\xi)$, we have $\xi\in \COrb(\eta)$. Fix such $\eta$. Suppose towards a contradiction that there is $A\in \xi$ such that $\overline A\cap \Orb(\eta)=\emptyset$. So, $ N_r(A)\notin\eta$ for all $r>0$. As $\eta$ is an ultrafilter, we have that   $ X\setminus N_r(A)\in\eta$ for all $r>0$. As $\eta\in \COrb(\xi)$, this implies that for all  $r>0$, there is $s>0$ such that $ N_s(X\setminus N_r(A))\in\xi$ (Proposition \ref{PropOrbitsAreCountable}\eqref{PropOrbitsAreCountable.Item1.5}). This $A$  contradicts \eqref{PropCharacterizationqOrbClosedItem4}.

\eqref{PropCharacterizationqOrbClosedItem3}$\Rightarrow$\eqref{PropCharacterizationqOrbClosedItem1}: Suppose  $\overline{\Orb}(\xi)$ is minimal among closed invariant subsets. As $\overline{\Orb}(\eta)$ is always closed and invariant,   $\overline{\Orb}(\xi)=\overline{\Orb}(\eta)$ for all $\eta\in \overline{\Orb}(\xi)$. In other words, every element in $\overline{\Orb}(\xi)$ is quasi-orbit equivalent to $\xi$. So, $\overline{\Orb}(\xi)=\qOrb(\xi)$.
\end{proof}

\subsection{Separation axioms for $\calOrb(\partial X)$}
We now study when $\calOrb(\partial X)$ is Hausdorff, $T_1$, and $T_0$. We start proving our characterization of $T_1$-ness presented in the introduction. This is a good moment to recall the reader of the concepts of coarse embeddability and equivalence.

\begin{definition}
    Let $(X,d_X)$ and $(Y,d_Y)$ be metric spaces and $f\colon X\to Y$ be a map.
    \begin{enumerate}
        \item The map $f$ is a \emph{coarse embedding} if given sequences $(x_n)_n$ and $(z_n)_n$ in $X$ we have that
        \[d_X(x_n,z_n)\to \infty \Leftrightarrow d_Y(f(x_n),f(z_n))\to \infty.\]
        \item A coarse embedding $f$ is a \emph{coarse equivalence} if $f$ is \emph{cobounded}, i.e., if
        \[\sup_{y\in Y}d(y,f(X))<\infty.\]
    \end{enumerate}
\end{definition}

Equivalently, a map $f$ is a coarse equivalence if and only if for all $r>0$ there is $s>0$ such that
\[d_X(x,z)\leq r\Rightarrow d_Y(f(x),f(z))\leq s\]
 and for all $s>0$ there is $r>0$ such that
\[d_X(x,z)\geq r\Rightarrow d_Y(f(x),f(z))\geq s.\]

\begin{theorem}
   The following are equivalent for a u.l.f.\ metric space $X$:
   \begin{enumerate}
       \item\label{Item1CorClosedQuasiOrbitsCoarseDisjointUnionSing} $X$ does not contain a coarse copy of $\cM_2$.
       \item\label{Item2CorClosedQuasiOrbitsCoarseDisjointUnionSing}  $\mathrm{Orb}(\xi)$ is closed for all $\xi \in \partial X$.
       \item\label{Item3CorClosedQuasiOrbitsCoarseDisjointUnionSing} $\qOrb(\xi)$  is closed for all $\xi \in \partial X$.
       \item\label{Item4CorClosedQuasiOrbitsCoarseDisjointUnionSing} The space $\calOrb(\partial X)$ is $T_1$.
   \end{enumerate}\label{CorClosedQuasiOrbitsCoarseDisjointUnionSing.Inthetext}
\end{theorem}

Before proving Theorem \ref{CorClosedQuasiOrbitsCoarseDisjointUnionSing.Inthetext}, let us isolate a lemma which guarantees the coarse embeddability of $\cM_2$ in a u.l.f.\ metric space $X$.

\begin{lemma}\label{Lemma.Coarse.Embedding.of.M2}
    Let $X$ be a u.l.f.\ metric space. Let $(\N_j)_{j=1}^\infty$ be a decreasing sequence of infinite subsets of $\N$ and for each $n\in\N_1$ let 
    \[F_n=\{j\in\N\colon n\in \N_j\}.\] Let $((x_n(j))_{n\in\N_j})_{j\in\N}$ be a family of distinct elements of  $X$ such that
    
\begin{enumerate}
    \item\label{Eq.Coarse.M2.130426.1} $\sup_{n\in \N_j}d(x_n(1),x_n(j))<\infty$
for all $j\in\N$, and
\item\label{Eq.Coarse.M2.130426.3} $\inf_{n\in\N_j}d(x_n(1), x_n(j))\to \infty$ as $j\to \infty$.
\end{enumerate}
Then $\cM_2$ coarsely embeds into $X$.
\end{lemma}

\begin{proof}
As the family $((x_n(j))_{n\in\N_j})_{j \in\N}$ consists of distinct elements,    the subsets
$\{x_n(j)\colon j\in F_n\}$, $n\in\N_1$, are all disjoint. Therefore, since $X$ is u.l.f., going to a subsequence if necessary, we can assume they are getting further away from each other. More precisely, there is an infinite $M\subseteq \N$ such that,  replacing each $\N_j$ by $\N_j\cap M$ if necessary,  we can assume furthermore that 
\[d(\{x_n(j)\colon j\in F_n\},\{x_m(j)\colon j\in F_m\})\to \infty\] as $n+m\to \infty$ with $n\neq m$ in $\N_1$. Also, restricting the $j$'s to an appropriate infinite subset of $\N$ if necessary, we can also assume that
\[\inf_{n\in\N_j}d(\{x_n(i)\colon i\in F_n, i<j\}, x_n(j))\to \infty\ \text{ as }\   j\to \infty.\]

    By diagonalizing $(\N_j)_j$ and relabeling elements, we can assume without loss of generality that $\N_j=\{n\in \N\colon n\geq j\}$ for all $j\in\N$. In particular, $F_n=\{1,\ldots,n\}$ for all $n\in\N$. It is now straightforward to check that the map 
    \[(j^2,n^2)\in \left\{(j^2,n^2)\in \N^2\colon n\geq j\right\}\mapsto x_n(j)\in X\]
    is a coarse embedding. Since the space in the left-hand side of the displayed equation above is coarsely equivalent to $\cM_2$, the result follows.
\end{proof}

 \begin{proof}[Proof of Theorem \ref{CorClosedQuasiOrbitsCoarseDisjointUnionSing.Inthetext}]
\eqref{Item1CorClosedQuasiOrbitsCoarseDisjointUnionSing}$\Rightarrow$\eqref{Item2CorClosedQuasiOrbitsCoarseDisjointUnionSing}: Suppose towards a contradiction that there is $\xi\in \partial X$ such that $\Orb(\xi)$ is not closed and let us show that $\cM_2$  coarsely embeds into $X$.  

As $\Orb(\xi)$ is not closed,   there is no $r>0$ as in  Theorem \ref{ThmCharacterizationOrbClosed.IntheText}\eqref{Item3ThmCharacterizationOrbClosed}.
We can then let      $(\N_j)_{j\in\N}$, $((x_n(j))_{n\in \N_j})_{j\in\N}$ and $(F_n)_n$ be as in the proof of the implication  \eqref{Item1ThmCharacterizationOrbClosed}$\Rightarrow$\eqref{Item3ThmCharacterizationOrbClosed} of Theorem \ref{ThmCharacterizationOrbClosed.IntheText}. Moreover, without loss of generality, we can also assume that, besides items \eqref{Item1.orbitnotclosed}--\eqref{Item3.orbitnotclosed} in the definition of these objects, they also satisfy
\begin{enumerate}\setcounter{enumi}{5}
    \item  $\inf_{n\in\N_j}d(x_n(1), x_n(j))\to \infty$ as $j\to \infty$.
\end{enumerate}
Indeed, this is simply a variation of property \eqref{Item2.orbitnotclosed} satisfied by these objects and an immediate consequence of the nonexistence of $r>0$ as in  Theorem \ref{ThmCharacterizationOrbClosed.IntheText}\eqref{Item3ThmCharacterizationOrbClosed}. By Lemma \ref{Lemma.Coarse.Embedding.of.M2}, $\cM_2$ coarsely embeds into $X$.

The implication \eqref{Item2CorClosedQuasiOrbitsCoarseDisjointUnionSing}$\Rightarrow$\eqref{Item3CorClosedQuasiOrbitsCoarseDisjointUnionSing}
    follows since $\Orb(\xi)\subseteq \qOrb(\xi)\subseteq \COrb(\xi)$ for all $\xi\in \beta X$.

For the implication \eqref{Item3CorClosedQuasiOrbitsCoarseDisjointUnionSing}$\Rightarrow$\eqref{Item1CorClosedQuasiOrbitsCoarseDisjointUnionSing}, suppose
  the space $X$ contains a coarse copy of $\{(m^2,n^2)\colon  n\geq m\}$.   For simplicity, assume $X$ actually equals  $ \{(m^2,n^2)\colon  n\geq m\}$  endowed with its canonical metric  and let $\xi$ be a nonprincipal ultrafilter on $X$ containing $A= \{(1,n^2)\colon n\in \N\}$. It is evident that
    \[N_{k^2}(X\setminus N_{k^2-1}(A))\in \xi \ \text{for all}\ k\in\N.\]
 So, Theorem \ref{ThmCharacterizationqOrbClosed} implies that $\qOrb(\xi)$ is not closed.

 The equivalence \eqref{Item2CorClosedQuasiOrbitsCoarseDisjointUnionSing}$\Leftrightarrow$\eqref{Item4CorClosedQuasiOrbitsCoarseDisjointUnionSing} is immediate.
\end{proof}

 We now characterize Hausdorfness of the orbit space of $\partial X$ in terms of the coarse embeddability of $\cM_{3/2}$. For this, the fact that a space which does not contain $\cM_{3/2}$ coarse must have asymptotic dimension zero will be needed. Now is a good moment to recall the reader of the definition of asymptotic dimension.

 \begin{definition}
\label{Definition.Asym.Dim.d}
 Let $X$ be a metric space and $n\in\{0\}\cup \N$. We say that $X$ has \emph{asymptotic dimension at most $n$} if for all $r>0$ there are families $\cU_0,\ldots,\cU_n$ of subsets of $X$ such that
 \begin{enumerate}
     \item $X=\bigcup_{i=0}^n\bigcup_{U\in \cU_i}U$,
     \item $\sup_{U\in \cU_i} \mathrm{diam}(U)<\infty$ for all $i\in\N$, and
     \item for all $i\in \{0,\ldots, n\}$ and all distinct $U,V\in \cU_i$, we have $d(U,V)>r$.
 \end{enumerate}
 If $X$ has asymptotic dimension at most $n$ but it does not have asymptotic dimension at most $n-1$, then $X$ has \emph{asymptotic dimension $n$}
 \end{definition}

\begin{definition}\label{Defi.Coarse.Disj.Union}
    Let $(X_n,d_n)_n$ be a sequence of u.l.f.\ metric spaces. The \emph{coarse disjoint union of $(X_n,d_n)_n$} is a metric space $X=\bigsqcup_nX_n$ endowed with a metric $d$ such that
    \begin{enumerate}
        \item $d_n\restriction X_n\times X_n=d_n$ for all $n\in\N$ and
        \item $d(X_n,X_m)\to \infty$ as $n+m\to \infty $ with $n\neq m$.
    \end{enumerate}
\end{definition}

Technically speaking, the article ``the'' in the definition of coarse disjoint unions is not correct since there are many metrics satisfying the above. However, any such metric gives rise to bijectively coarsely equivalent spaces. As we only care about the coarse type of our spaces, this justify our choice of article.

\begin{theorem} 
 The following are equivalent for a u.l.f.\ metric space $X$:
     \begin{enumerate}
         \item\label{Prop.HausdorffSSS.Item1} $X$ does not contain a coarse copy  of $\cM_{3/2}$.
         \item\label{Prop.HausdorffSSS.Item2} $\calOrb(\partial X)$ is Hausdorff.
     \end{enumerate} \label{Thm.Hausdorfness.Charac.InTheText}
\end{theorem}

\begin{proof}
\eqref{Prop.HausdorffSSS.Item1}$\Rightarrow$\eqref{Prop.HausdorffSSS.Item2}
As mentioned in the introduction, $\cM_{3/2}$ is coarsely equivalent to
\begin{equation}\label{Eq.M32.prova}
\{(1,k^2),(n,k^2)\in \N^2\colon n\in\N,\ k\in \N_n\},\end{equation}
where $\N=\bigsqcup_n\N_n$ is a partition of $\N$ into infinite subsets and $\N^2$ is considered with its standard metric. It is therefore straightforward that $X$ contains a coarse copy of $\cM_{3/2}$ if and only if  for all $r>0$ there are sequences $(x_n)_n$ and $(y_n)_n$ of distinct elements of $X$ such that $\sup_{n}d(x_n,y_n)<\infty$ and $d(x_n,y_n)>r$ for all $n\in\N$. This characterization may be even more clear if one consideres the other description of $\cM_{3/2}$ given in the introduction: $\cM_{3/2}$ is the coarse disjoint union  over $n$ of $X_n=\cM_1\sqcup\cM_1$ where the Hausdorff distance in each $X_n$ between the disjoint copies of $\cM_1$ is finite but goes to infinite as $n\to \infty$.

If $X$ does not contain a coarse copy of $\cM_{3/2}$, fix then $r>0$ for which there are no sequences as above. Moreover, the lack of coarse embeddability of  $\cM_{3/2}$ into  $X$ implies that $X$ must have asymptotic dimension zero. So, there are subsets $(U_n)_n$ of $X$ such that $d(U_n,U_m)>r$ for all distinct $n,m\in\N$ and $\sup_n\mathrm{diam}(U_n)<\infty$. By our choice of $r$, we must have that $d(U_n,U_m)\to \infty$ as $n+m\to \infty$ with $n\neq m$. As $X$ is u.l.f., there is $k\in\N$ such that $|U_n|\leq k$ for all $n\in\N$. Then, letting $i(n)=|U_n|$, we have that $X$ is bijectively coarsely equivalent to the   coarse disjoint union
\[Y=\bigsqcup_n\{1,\ldots, i(n)\}.\]

Let  $\xi,\zeta\in \partial Y$ be in different orbits. It is straightforward that there are  disjoint $A,B\subseteq \N$ such that $\bigsqcup_{n\in A}\{1,\ldots, i(n)\}\in\xi$ and $\bigsqcup_{n\in B}\{1,\ldots, i(n)\}\in \zeta$.  The orbits of $\overline A$ and $\overline B$ are clearly disjoint.

\eqref{Prop.HausdorffSSS.Item2}$\Rightarrow$\eqref{Prop.HausdorffSSS.Item1} Suppose $X$ contains a coarse copy of $\cM_{3/2}$. For simplicity of notation, assume $\cM_{3/2}$ is the set in \eqref{Eq.M32.prova} and it is contains in $X$. For each $n\in\N$, let
\begin{equation}
X_n=\{(1,k^2),(n,k^2)\in \N^2\colon \ k\in \N_n\},\end{equation}
Let $\xi$ be a nonprincipal ultrafilter on $\N$ and, for each  $p\in\N$, let $\zeta_n$ be a nonprincipal ultrafilter on $\N_n$. Define nonprincipal ultrafilters $\eta_1$ and $\eta_2$ on $X$ by letting
\begin{align*}
A\in \eta_1\ \Leftrightarrow\ \{n\in\N\colon \{k\in\N\colon (1,k^2)\in A\}\in \zeta_n\}\in\xi
\end{align*}
and
\begin{align*}
A\in \eta_2\ \Leftrightarrow\ \{n\in\N\colon \{k\in\N\colon (n,k^2)\in A\}\in \zeta_n\}\in\xi.
\end{align*}
It is immediate that $\eta_1$ and $\eta_2$ are not in the same orbit and that for all $A_1\in \eta_1$ and $A_2\in \eta_2$, the orbits of the subsets $\overline{A_1}$ and $\overline{A_2}$ intersect. In other words, $\Orb(\eta_1)$ and $\Orb(\eta_2)$ are distinct elements of $\calOrb(\partial X)$ which cannot be separated by open subsets. So, $\calOrb(\partial X)$ is not Hausdorff.
\end{proof}

\subsubsection{When is $\calOrb(\partial X)$ a $T_0$ space?}
While we have characterizations of Hausdorffness and $T_1$-ness, this is not the case for $T_0$-ness. We now present our findings about the $T_0$ axiom for $\calOrb(\partial X)$. The following is completely immediate from the definitions.

\begin{proposition}\label{CorT0IFFOrb=qOrb}
Let $X$ be a u.l.f.\ metric space. Then $\calOrb(\partial X)$ is $T_0$ if and
only if $\Orb=\qOrb$ on $\partial X$.\qed
\end{proposition}

\begin{remark}
    Proposition \ref{CorT0IFFOrb=qOrb} is actually valid in a much broader sense. Precisely, consider  a dynamical system $(M,\mathcal S)$ where $M$ is a  locally compact Hausdorff space and $\mathcal S$ is   an inverse
semigroup $\mathcal S$ acting on $M$ by partial homeomorphisms
$s\colon \dom(s)\to \ran(s)$ between open subsets of $M$. Then the quotient space $M/\qOrb$ is always $T_0$ --- in fact, it is the $T_0$-ization (Kolmogorov quotient) of the quotient space $M/\Orb$. In particular, $M/\Orb$ is $T_0$ if and only if $\Orb=\qOrb$.
\end{remark}

 While Theorem \ref{CorClosedQuasiOrbitsCoarseDisjointUnionSing.Inthetext} gives that $\calOrb(\partial\cM_2)$ is not $T_1$, the next proposition shows that this space is at least $T_0$.

\begin{proposition}\label{PropSimpleSpaceT_0}
The orbit space $\calOrb(\partial \cM_2)$ is $T_0$. \end{proposition}

\begin{proof}
    Let $\xi,\eta\in \partial X$ and suppose $\Orb(\xi)\neq \Orb(\eta)$. If either $\Orb(\xi)$ or $\Orb(\eta)$ are closed, then its complement is an open neighborhood of one of this elements which does not contain the other, as desired. Suppose then that neither  $\Orb(\xi)$ nor $\Orb(\eta)$ are closed. For simplicity, let us consider the representation of $\cM_2$ given by \eqref{Rq.Definition.Mk.concreta}, so,
    \[\cM_2=\{(m^2,n^2)\in \N^2\colon m<n\}.\]
    By Theorem \ref{ThmCharacterizationOrbClosed.IntheText}, it is immediate that there is $k\in\N$ such that both $\xi $ and $\eta$ contain
    \[\{(m^2,n^2)\in \cM_2\colon  m\leq k\}.\]
    Moreover, as $\xi$ and $\eta$ are not in the same orbit, there must be $A\subseteq \N$ such that
\[ B=   \{(m^2,n^2)\in \cM_2\colon \ m\leq k,\ n\in A\}\in \xi\] and
\[   C=  \{(m^2,n^2)\in \cM_2\colon   m\leq k,\ n\notin A\}\in \eta.\]
Then $\overline B$ and $\overline C$ are disjoint open subsets containing $\xi $ and $\eta$, respectively.
\end{proof}

In the opposite direction, we now proceed to obtain  Theorem~\ref{CorAsymDomQuaseLargerOrb.Inthetext}, which provides a fairly weak condition which ensures  that $\calOrb(\partial X)$ is not   $T_0$. We start with a couple preparatory propositions.

\begin{proposition}\label{PropCoarseEmbOrbSpaceHomeo}
    Let $X$ and $Y$  be  u.l.f.\ metric spaces.
    \begin{enumerate}
        \item\label{PropCoarseEmbOrbSpaceHomeo.1} If $X\subseteq Y$, then this inclusion  induces a homeomorphic embedding of $\calOrb(\partial X) $ onto an open subset of  $\calOrb(\partial Y)$.
        \item\label{PropCoarseEmbOrbSpaceHomeo.2} If $X\subseteq Y$ and $X$ is cobounded in $Y$, the induced homeomorphism of the previous item is surjective.
        \item \label{PropCoarseEmbOrbSpaceHomeo.3} Any coarse embedding   $f\colon X\to Y$    induces a homeomorphic embedding of $\calOrb(\partial X) $ onto an open subset of  $\calOrb(\partial Y)$.
    \end{enumerate}

\end{proposition}

\begin{proof}
\eqref{PropCoarseEmbOrbSpaceHomeo.1} Suppose $X\subseteq Y$ and let $i\colon X\to Y$ be the inclusion map. Let then $\overline i\colon \partial X\to \partial Y$ be the restriction to $\partial X$ of the Stone--\v{C}ech extension $\beta X\to \beta Y$ of $i$. So, $\overline i$ is a homeomorphism between $\partial X$ and the open subset  $\overline X\subseteq \partial Y$.
 This map induces a map $j\colon \calOrb(\partial X)\to \calOrb(\partial Y)$ which makes the diagram in Figure \ref{Diagram.1} commute, where the vertical arrows in this diagram as the canonical quotient maps. As this quotient maps are open maps (Corollary \ref{CorQuotientOpen}), the result follows.
 \begin{figure}[h]
 \[ \xymatrix{
  \partial X\ar[r]^{\overline i}\ar[d] & \partial Y\ar[d]\\
  \calOrb(\partial X) \ar[r]_j       & \calOrb(\partial Y)
 } \]
\caption{}\label{Diagram.1}
\end{figure}

\eqref{PropCoarseEmbOrbSpaceHomeo.2} Suppose further more that $X$ is a cobounded subset of $Y$. Fix $r>0$ such that $Y=N_r(Y)$. As $Y$ is u.l.f., there is $N\in\N$ and a partition $Y=\bigsqcup_nY_n$ such that each $Y_n$ has diameter at most $r$ and cardinality at most $N$. Therefore, every $\xi\in \partial Y$ is in the orbit of some $\zeta\in \partial X$. This shows that the map defined in the previous time is surjective.

\eqref{PropCoarseEmbOrbSpaceHomeo.3} If $f\colon X\to Y$ is a coarse embedding, then there is $r>0$ such that $d_X(x,x')>r$ implies $d_Y(f(x),f(x'))>0$ for all $x,x'\in X$, where $d_X$ and $d_Y$ are the distances in $X$ and $Y$ respectively. So, there is $Z\subseteq X$ cobounded in $X$ such that $f\colon Z\to Y$ is injective.  The result then follows from the previous two items.
\end{proof}

\begin{proposition}\label{PropIfContainsSubspaceWithNoClosedOrbitsQuasiOrbWeakerOrb}
    Let $Y$ and $X$ be  u.l.f.\ metric space and suppose that, for all  $\xi\in \partial Y$,  $\Orb(\xi)$ is not  closed in $\calOrb(\partial Y)$. If $Y$ coarsely embeds into $X$, then $\calOrb(\partial X)$ is not $T_0$.
    \end{proposition}

\begin{proof}
Let us first notice the result holds if $X=Y$. Using  Zorn's lemma, we can find a nonempty $F\subseteq \partial X$ which is minimal under inclusion with respect to being closed and invariant. By the minimality of $F$,   $F=\COrb(\xi)$ for all $\xi\in F$. Hence, by Theorem \ref{ThmCharacterizationqOrbClosed}, $F=\qOrb(\xi)$. As $\Orb(\xi)$ is not closed, $\qOrb(\xi)$ is larger than $\Orb(\xi)$ and the result follows from Proposition \ref{CorT0IFFOrb=qOrb}.

The general case  follows from  Proposition \ref{PropCoarseEmbOrbSpaceHomeo} since $\calOrb(\partial X)$ is homeomorphic to an open subset of $\calOrb(\partial Y)$.
\end{proof}

The following proposition deals with the space $\cM$ defined in \eqref{Rq.Definition.M}.

\begin{proposition}\label{PropMnotT0}
Let $\cM$ be the metric space defined above. Then none of the orbits $\Orb(\xi)$, for $\xi\in \partial \cM$, is closed.
\end{proposition}

\begin{proof}
    The proof is  an application of Theorem \ref{ThmCharacterizationOrbClosed.IntheText}. For that, fix $r\in \N$. Then, if $\bar a=(a_i)_i,\bar b=(b_i)_i\in \cM$ are such that $a_i=b_i$ for all $i\in \N\setminus \{r\}$ and $a_r\neq b_r$, we have that
    $d(\bar a,\bar b)=3^r$, where $d$ is the metric of $\cM$ defined in \eqref{Eq.Definiion. Metric.M}. This shows that
    \[\bar a\in \{x\in \cM\colon d(N_r(x),X\setminus  N_r(x))\leq 3^r\}\]
    for any such $\bar a\in \cM$. Therefore, the set on the right-hand side above contains all tuples in $\cM$  with length larger than $r$  and, in particular, it is cofinite. Hence, if $\xi\in \partial \cM$, $\xi$ contains this set for all $r\in\N$. By  Theorem \ref{ThmCharacterizationOrbClosed.IntheText}, $\Orb(\xi)$ is not closed.
\end{proof}

 \begin{theorem}
Let $X$ be a u.l.f.\ metric space.  If $\cM$ coarsely embeds into $X$,
then $\calOrb(\partial X)$ is not $T_0$. In particular, if the
asymptotic dimension of $X$ is at least $1$, then $\calOrb(\partial X)$
is not $T_0$.\label{CorAsymDomQuaseLargerOrb.Inthetext}
\end{theorem}

\begin{proof}
Let $X$ be a u.l.f.\ metric space containing a coarse copy of $\cM$ as in Proposition \ref{PropMnotT0}. It then follows from this proposition and Proposition \ref{PropIfContainsSubspaceWithNoClosedOrbitsQuasiOrbWeakerOrb} that $\calOrb(\partial X)$ is not $T_0$.
\end{proof}

Unfortunately, we could not solve   the problem of completely characterizing  $T_0$-ness of $\calOrb(\partial X)$. We leave this task for a smarter reader:

\begin{problem}\label{ProbCharacT0}
Is there a u.l.f.\ metric space $Y$ such that, given a u.l.f.\ metric space $X$, $\calOrb(\partial X)$ is $T_0$ if and only if $Y$ does not coarsely embeds into $X$? Notice that, by Theorem~\ref{CorAsymDomQuaseLargerOrb.Inthetext}, any such $Y$ necessarily has asymptotic dimension zero. Can it be the case that $Y=\cM$ works?
\end{problem}

\section{The Higson corona   and applications}\label{SectionHigson}

In this section, we study the Higson equivalence relation on $\beta X$. This will be crucial for the proof of Theorem \ref{ThmLocalUrysohnForOrb.InTheText}. We start recalling the definition of Higson functions.

\begin{definition}\label{DefiHigson}
    Let $X$ be a u.l.f.\ metric space and $h\colon X\to \C$ be  bounded. We call $h$ a \emph{Higson function} if for all $\eps,r>0$ there is a finite $A\subseteq X$ such that
    \[\forall x,y\in X\setminus A,\ d(x,y)\leq r\ \text{ implies }\ |h(x)-h(y)|\leq \eps.\]
    The \cstar-algebra of all Higson functions is denoted by $C_h(X)$.
\end{definition}

 The next proposition gives a characterization of when a bounded function $X\to \C$ is Higson in terms of dynamics.

  \begin{proposition}\label{PropHigsonFunctionsAreConstantOnOrbits}
      Let $X$ be a u.l.f.\ metric space and $h\colon X\to \C$ be a bounded function. Then, under the canonical identification $\ell_\infty(X)\cong C(\beta X)$,   $h$ is a Higson function if and only if it is constant on the orbits of $\partial X$.
  \end{proposition}

\begin{proof}
Suppose first that $h$ is a Higson function and let  $\xi\in \partial X$ and $f\in \PT(X)_\xi$. We shall notice that $h(\xi)= h(\overline f(\xi))$. Let $\eps>0$ and
    \[r=\sup_{x\in \dom(f)}d(x,f(x)).\] Then, as $h$ is Higson, there is a finite $F\subseteq X$ such that
    \[\forall x,y\in X\setminus A,\ d(x,y)\leq r\ \text{ implies }\ |h(x)-h(y)|\leq \eps.\]
Therefore, as $\xi$ is nonprincipal,
    \[\{x\in \dom(f)\colon |h(x)-h(f(x))|\leq \eps\}\in \xi.\]
Since
    \[ h(\xi)=\lim_{x,\xi}h(x)\ \text{ and }\  h(\overline f(\xi))=\lim_{x,\overline f(\xi)}h(x)=\lim_{x\in \dom(f),\xi}h(f(x)),\]
 this implies that
    \[\left|   h(\xi)-  h(\overline f(\xi))\right|\leq \eps.\]
As $\eps$ was arbitrary, $ h(\xi)= h(\overline f(\xi))$.

Assume now that $h$ is constant on the orbits of $\partial X$ and suppose towards a contradiction that $h$ is not a Higson function.
 Then there are $\eps>0$ and sequences $(x_n)_n,(y_n)_n\subseteq X$ of distinct elements with   $x_n\neq y_m$ for all   $n,m\in\N$ such that
\[\sup_{n\in\N}d(x_n,y_n)<\infty \ \text{ and } \ |h(x_n)-h(y_n)|\geq \eps\]
for all $n\in\N$. Let $f\colon \{x_n\colon n\in\N\}\to\{y_n\colon n\in\N\}$ be the partial translation such that $f(x_n)=y_n$ for all $n\in\N$. Then, if $\xi\in \partial X$ is a nonprincipal ultrafilter containing $\{x_n\colon n\in\N\}$, we must have $| h(\xi)- h(\overline f(\xi))|\geq \eps$. In particular, $ h(\xi)\neq  h(\overline f(\xi))$, so,  $h$ is not constant on the orbits of $\partial X$.
\end{proof}

  \subsection{Characterizations of the Higson equivalence}

  We define the Higson-orbit of an element  $\xi\in \partial  X$ as
 \[\HOrb(\xi)=\{\eta\in \partial X\colon h(\xi)=h(\eta)\ \text{ for all }
 h\in C_h(X)\}.\]
 This  defines the \emph{Higson  equivalence relation} on $\partial X$, denoted by $\HOrb$. By Proposition \ref{PropHigsonFunctionsAreConstantOnOrbits},
 \[\Orb_{|\partial X}\subseteq \qOrb_{|\partial X}\subseteq \HOrb,\]
where $\Orb_{|\partial X}$ and $\Orb_{|\partial X}$ denote the restrictions of $\Orb$ and $\qOrb$ to $\partial X$, respectively.

The following is the main result of this subsection.

\begin{theorem}\label{ThmHigsonRel}
    Let $X$ be a u.l.f.\ metric space and $\xi,\eta\in \partial X$. The following are equivalent:
    \begin{enumerate}
        \item \label{ItemThmHigsonRel1}$\HOrb(\xi)=\HOrb(\eta)$.
\item\label{ItemThmHigsonRel2}
There are no disjoint  open subsets of $\calOrb(\partial X)$ which separate $\Orb(\xi)$ and $\Orb(\eta)$.
\item\label{ItemThmHigsonRel2.5} $\xi \sim_{R}\eta$, where $R\subseteq \partial X\times \partial X$ is the topological closure of the orbit equivalence relation.
\item\label{ItemThmHigsonRel3} $\xi \sim_{R'}\eta$, where $R'\subseteq \partial X\times \partial X$ is the smallest closed equivalence relation which is orbit-invariant.
    \end{enumerate}
\end{theorem}

The   equivalence between \eqref{ItemThmHigsonRel1} and \eqref{ItemThmHigsonRel3} of Theorem \ref{ThmHigsonRel} was already obtained in
\cite[Proposition 1]{Protasov2005TopAndAppl} in the language for balleans.  In Section \ref{SubsectionRestrictionGNSHigFunc}, we add yet another characterization to the list above in terms of certain GNS representations of uniform Roe algebras.

  We need some preparatory results before proving Theorem \ref{ThmHigsonRel}.

  \begin{lemma}\label{LemmaExistenceOfHigsonFunctionSeparationSubsetsFromFarAwaySequences}
    Let $X$ be a u.l.f.\ metric space, $\xi\in \partial X$, and $B\subseteq X$. Suppose there is $A\in \xi$ such that
    \[d(x,B)\to \infty\ \text{ as }\ x\underset{x\in A}{\to} \infty .\]
    Then, there are an open subset $U\subseteq \beta X$ with $\xi \in U$ and a Higson function $ h\colon X\to [0,1]$ such that $  h\restriction U=1$ and $ h\restriction \overline B=0$.
\end{lemma}

\begin{proof}
    As $X$ is u.l.f., the  hypothesis allow us to  find a partition for $A$ into finite subsets, say $A=\bigsqcup_nA_n$, such that
\begin{enumerate}
\item $d(A_n,B)\geq n$ for all $n>1$, and
    \item $d(A_n,A_m)> 2n$ for all $n\geq 3$ and all $m\leq n-2$.
\end{enumerate}
As $\xi$ is an ultrafilter, either $\bigsqcup_nA_{2n}\in \xi$ or $\bigsqcup_nA_{2n-1}\in \xi$.  Without loss of generality, assume $A'=\bigsqcup_nA_{2n}\in \xi$.  For each $n\in\N$, let $h_n\colon X\to [0,1]$ be given by
\[h_n(x)=\max\left\{1-\frac{d(x,A_{2n})}{n},0\right\}\]
for all $x\in X$.
Define then $h\colon X\to [0,1]$ by letting
\[h(x)=\sum_{n\in\N}h_n(x)\]
for all $x\in X$. Since for each $x\in X$, the set $\{n\in\N\colon h_n(x)\neq 0\}$ has at most 1 element, $h$ is well defined. Moreover, it is straightforward   to check that $h$ is a Higson function. Clearly, $h\restriction B=0$ and   $h\restriction A'=1$. So, the extension of $h$ to $\beta X$ satisfies  $ h\restriction \overline B=0$ and $ h \restriction \overline A=1$.
\end{proof}

\begin{lemma}
\label{LemmaHigsonLattice}
 Let $X$ be a u.l.f.\ metric space and $f,g\colon X\to \R$ be Higson functions. Then $\max\{f,g\}$ is also a Higson function.
\end{lemma}
\begin{proof}
 Fix $\eps>0$. As both $f$ and $g$ are Higson, there are $R>0$ and $A\subseteq X$ finite such that for all $x,y\in X\setminus A$ with $d(x,y)<R$, we have that $|f(x)-f(y)|\leq \eps$ and $|g(x)-g(y)|\leq \eps$. But then, for all such $x$ and $y$, we have $|h(x)-h(y)|\leq 2\eps$, where $h=\max\{f,g\}$.
\end{proof}

\begin{theorem}\label{ThmHigsonRel.Closed.Sets}
    Let $X$ be a u.l.f.\ metric space and $ E, F\subseteq \partial X$ be closed. Suppose there are invariant open disjoint subsets $U,V\subseteq \partial X$ with $E\subseteq U$ and $F\subseteq V$. Then, there is  a Higson function $h\colon X\to [0,1]$ such that  $  h \restriction E=0$, and $ h\restriction F=1$.
\end{theorem}

\begin{proof}Throughout the proof, we make the following abuse of notation: if $A\subseteq X$, we will denote by $\overline A$ the intersection $\overline A \cap \partial X$. I.e., $\overline A$ denotes the nonprincipal ultrafilters which are in the closure of $A\subseteq \beta X$.

 By the compactness of $E$, we can choose   subsets $ B_1,\ldots, B_k\subseteq X$ such that
 \[E\subseteq \bigcup_{i=1}^k\overline{B_i}\subseteq U.\]
Let $B=\bigcup_{i=1}^kB_i$. By the properties of ultrafilters, it is immediate that $\overline B=\bigcup_{i=1}^n\overline{B_i}$.  So, \[E\subseteq \overline B\subseteq U.\]

\begin{claim}
    For each $\xi\in F$, there is $A\in \xi$ and a Higson function $h\colon X\to [0,1]$ such that $ h\restriction \overline A=1$ and $ h\restriction \overline B=0$.
\end{claim}

\begin{proof}
As $\xi\in F\subseteq V$, we can pick $A_0\subseteq X$ such that $\xi\in \overline{A_0}\subseteq V$.   By Lemma \ref{LemmaExistenceOfHigsonFunctionSeparationSubsetsFromFarAwaySequences}, it is enough to notice that
\begin{equation}\label{EqPerguntaRuy25Sep24} d(x,B)\to \infty\ \text{ as }\ x\underset{x\in A_0}{\to} \infty .
\end{equation}
In other words, let us show that     for all $x_0\in X$ and  $n>0$ there is $r>0$ such that for all  $x\in A_0$ with $x\not\in N_r(x_0)$,  we have that  $d(x,B)>n$. Suppose this is not the case. Then, there is $n\in\N$ and an infinite $A\subseteq A_0$ such that $d(x,B)\leq n$ for all $x\in A$. Replacing $A$ by a smaller infinite subset if necessary, we can assume there is a partial translation $f$ of $X$ with $\dom(f)=A$ and $\mathrm{im}(f)\subseteq B$. Then, if $\eta$ is a nonprincipal ultrafilter with $A\in \eta$, we have that $\eta\in \overline{A}\subseteq \overline{A_0}$ and $f(\eta)\in \overline{B}$. On the other hand, since   $\overline{A_0}\subseteq V$ and $V$ is invariant, $\overline f(\eta)\in V$. As $\overline{B}\subseteq U$, this contradicts the assumption on $U$ and $V$ of them being   disjoint.
\end{proof}

 By the compactness of $F$ together with the previous claim, we can choose   subsets $ A_1,\ldots, A_k\subseteq X$ and Higson functions $(h_i\colon  X\to [0,1])_{i=1}^n$ such that
 \[F\subseteq \bigcup_{i=1}^n\overline{A_i}\subseteq U,\  h_i\restriction \overline{A_i}=1,\ \text{ and } \ {h_i}\restriction \overline{B}=0\]
for all $i\in \{1,\ldots, n\}$. Let $A=\bigcup_{i=1}^nA_i$. Then $\overline A=\bigcup_{i=1}^n\overline{A_i}$. Hence, letting
\[h=\max\{h_1,\ldots, h_n\},\]
we have that $\overline  h\restriction \overline A=1$ and $ h\restriction \overline B=0$.    By Lemma \ref{LemmaHigsonLattice}, $h$ is a Higson function and we are done.
\end{proof}

\begin{proof}    [Proof of Theorem \ref{ThmHigsonRel}]
\eqref{ItemThmHigsonRel1}$\Rightarrow$\eqref{ItemThmHigsonRel2}: This  is a particular case of Theorem \ref{ThmHigsonRel.Closed.Sets}.

\eqref{ItemThmHigsonRel2}$\Rightarrow$\eqref{ItemThmHigsonRel2.5}: If   \eqref{ItemThmHigsonRel2} holds, then, we can find nets $(\xi_i)_{i\in I}$ and $(\eta_i)_{i\in I}$ in $\partial X$ such that
\[\xi=\lim_{i,I}\xi_i, \ \eta=\lim_{i,I}\eta_i, \ \text{ and } \eta_j\in \Orb(\xi_j) \ \forall j\in I.\]
Since $R$ is a closed relation containing the  orbit equivalence relation,  this implies that $\xi\sim_R \eta$.

\eqref{ItemThmHigsonRel2.5}$\Rightarrow$\eqref{ItemThmHigsonRel3}:
Clearly, $\sim_R\subseteq_{R'}$.

\eqref{ItemThmHigsonRel3}$\Rightarrow$\eqref{ItemThmHigsonRel1}: Let $\sim_H\subseteq \partial X\times \partial X$ be the Higson relation given in \eqref{ItemThmHigsonRel1}, i.e., $\xi\sim_H\eta$ if and only if $h(\xi)=h(\eta)$ for all Higson functons $h\colon X\to \C$. Then, it is completely straightforward that $\sim_H$ is an orbit-invariant  closed relation. Therefore, the minimality of $\sim_R$ implies $\sim_R\subseteq \sim_H$.
\end{proof}

\subsection{Topologically transitive actions}

Higson functions are also the key to show that the dynamical system  $(\partial X, \PT(X))$ is never topologically transitive. Recall:

\begin{definition}
Consider a dynamical system  $(M,\mathcal{S})$, where $M$ is a topological space and $\mathcal{S}$ an inverse semigroup acting on $M$. Then $(M,\mathcal S)$ is called \emph{topologically transitive} if for all open $U,V\subseteq M$ there is $s\in \mathcal S$ such that $s(U)\cap V\neq \emptyset$.
\end{definition}

As mentioned in the introduction, while  $(\beta X,\PT(X))$ is always topologically transitive --- indeed, this is immediate since   $(X,\PT(X))$ is transitive ---, this is no longer the case when we restrict our action to the corona $\partial X$.
\begin{theorem}
Let $X$ be an infinite  u.l.f.\ metric space.  The dynamical system $(\partial X,\PT(X))$ is not topologically transitive.\label{ThmMixAction.InTheText}
\end{theorem}

\begin{proof}
Let $hX$ be the \emph{Higson compactification of $X$}, i.e., $hX$ is the compact Hausdorff space given by Gelfand duality so that the unital abelian \cstar-algebra $C_h(X)$ is $^*$-isomorphic to $C(hX)$. Under these canonical identifications, we have
\begin{equation}\label{Inclusion}
  C(hX) = C_h(X)\subseteq \ell _\infty (X)=C(\beta X).
  \end{equation}
 Hence there is a continuous  surjective map
  $$
  \varphi \colon \beta X\to hX,
  $$
  so that the inclusion in \eqref{Inclusion} becomes
  $$
  b\in C(hX)\mapsto b\circ \varphi \in C(\beta X).
  $$

\begin{claim}
    The map $\varphi$ is constant on the orbits of $\beta X$.
\end{claim}
\begin{proof}
Arguing by contradiction,  suppose the contrary, and hence there exist $\xi\in \beta X$ and $f\in \PT(X)_\xi$ such that $\varphi(\xi)\neq \varphi(\overline f(\xi))$. As $hX$ is a compact Hausdorff space, there is  $b\in C(hX)$ such that
$b(\varphi (\overline f(\xi )))\neq b(\varphi (\xi ))$.  The function on $\beta X$ defined by $a=b\circ \varphi $ is then a Higson function and it satisfies
$a(\overline f(\xi ))\neq a(\xi )$. This contradicts the fact that Higson functions are constant on the orbits of $\beta X$ (Proposition \ref{PropHigsonFunctionsAreConstantOnOrbits}).
\end{proof}

As $X$ is infinite, the Higson corona $\nu X=hX\setminus X$ has more than one point (in fact, $\nu X$ has $2^{2^{\aleph_0}}$ elements, see~\cite[Theorem 3]{Keesling:1994cr} or \cite[Theorem 4.14]{BragaExel2023KMS}). Then, letting   $\alpha ,\beta \in \nu X$ be distinct and $U,V\subseteq hX$ be disjoint open subsets containing $\alpha$ and $\beta$, respectively,  it follows that
  $$
  U'=\varphi ^{-1}(U)\cap \partial X  \ \text{ and }\
  V'=\varphi ^{-1}(V)\cap \partial X
  $$
  are nonempty open invariant subsets of $\partial X$. This shows that the action of $\PT(X)$ on $\partial X$ is not topologically transitive.
\end{proof}

 \subsection{A localized version of Urysohn's lemma for $\calOrb(\partial X)$}
If $X$ is a u.l.f.\ metric space, its \emph{Higson corona} is  $\nu X=h X\setminus X$ --- the Higson compactification $h X$ was defined in the proof of Theorem \ref{ThmMixAction.InTheText}.  Let $\phi\colon \partial X\to  \nu X$ be the canonical quotient map so that the inclusion
\begin{equation}\label{EqInclusions11Oct}
    C(\nu X)\cong C_h(X)/c_0(X)\subseteq \ell_\infty(X)/c_0(X)\cong C(\partial X)
\end{equation}
is given by
\[b\in C(\nu X)\mapsto b\circ \phi \in C(\partial X).\]

\begin{proposition}\label{PropCnuXCOrbX}
    Let $X$ be a u.l.f.\ metric space,  $\phi\colon \partial X\to  \nu X$ be the canonical quotient map described above, and let $\theta\colon \calOrb(\partial X)\to \partial X$ be a section for the quotient map $\partial X\to \calOrb(\partial X)$. Then,
    \[b\in C(\nu X)\mapsto b\circ \phi\circ \theta\in C(\calOrb(\partial X))\]
    is an $^*$-isomorphism.
\end{proposition}

\begin{proof}
    Let $\alpha$ denote the map in the proposition. We start noticing $\alpha$ is well defined, i.e., that $\alpha(b)$ is indeed continuous on $\calOrb(\partial X)$ for all $b\in C(\nu X)$. Fix $b\in C(\nu X)\cong C_h(X)/c_0(X)$ and pick $h\in C_h(X)$ such that $b=h+c_0(X)$. Suppose $\xi\in \partial X$ and $(\xi_i)_{i\in I}$ is a net in $\partial X$ such that $\Orb(\xi)=\lim_{i,I}\Orb(\xi_i)$. Fix  $\eps>0$. As $h$ is continuous as a function on  $\partial  X$, there is an open $U\subseteq \partial X$ containing $\xi$ such that $|h(\xi)-h(\eta)|\leq \eps$ for all $\eta\in U$. As $(\Orb(\xi_i))_{i,I}$ converges to $\Orb(\xi)$ in $\calOrb(\partial X)$, there is $i_0\in I$ such that \[\Orb(U)\cap \Orb(\xi_i)\neq \emptyset\] for all $i>i_0$. Therefore, since Higson function are constant on orbits (Proposition \ref{PropHigsonFunctionsAreConstantOnOrbits}), this implies that \[|h(\xi)-h(\xi_i)|\leq \eps \ \text{  for all }\  i>i_0.\] Since $\alpha(b)(\xi)=h(\xi)$ and $\alpha(b)(\xi_i)=h(\xi_i)$, we conclude that  \[|\alpha(b)(\Orb(\xi))-\alpha(b)(\Orb(\xi_i))|\leq \eps
\]
for all $i>i_0$. By the arbitrariness of $\eps$ and $\xi$, $\alpha(b)$ is continuous as desired.

Injectivity of $\alpha$ is immediate.  Indeed, if $b,b'\in C(\nu X)$ are distinct, then there is $\xi\in \partial  X$ such that $b(\phi(\xi))\neq b'(\phi(\xi))$. So, \[\alpha(b)(\Orb(\xi))\neq \alpha(b')(\Orb(\xi)).\]

We are left to notice that $\alpha$ is surjective. This will again be a consequence of Proposition \ref{PropHigsonFunctionsAreConstantOnOrbits}. Indeed,   let $a\in C(\calOrb(\partial X))$ the define a map $b\in C(\partial X)$ as
\[b=a\circ q \text{ where }\ q\colon \partial X\to \calOrb(\partial X)\]
is the quotient map. Under the identifications in \eqref{EqInclusions11Oct}, there is $h\in \ell_\infty(X)$ such that $b=h+c_0(X)$ and we should show that $h\in C_h(X)$. By the definition of $b$, it is clear that $h$ must be constant on the orbits of $\partial X$. So it follows from Proposition \ref{PropHigsonFunctionsAreConstantOnOrbits} that $h$ is a Higson function and we are done.
\end{proof}

\begin{theorem}
Let $X$ be a u.l.f.\ metric space and suppose that
$E,F\subseteq \calOrb(\partial X)$ are closed subsets that can be
separated by open sets. Then there exists a continuous function
$f\colon \calOrb(\partial X)\to [0,1]$ such that $f|_E=0$ and $f|_F=1$.\label{ThmLocalUrysohnForOrb.InTheText}
\end{theorem}

\begin{proof}
     Let $E,F\subseteq \calOrb(\partial X)$ be closed invariant subsets. Let $q\colon \partial X\to \calOrb(\partial X)$ be the quotient map and define
     \[E'=q^{-1}(E)\ \text{ and }\ F'=q^{-1}(F).\]
     So, $E'$ and $F'$ are closed. By hypothesis, there are  disjoint  open subsets $U,V\subseteq \calOrb(\partial X)$ such that $E\subseteq U$ and $F\subseteq V$. Then, letting
          \[U'=q^{-1}(U)\ \text{ and }\ V'=q^{-1}(V),\]
          $U'$ and $V'$ are disjoint open invariant subsets containing $E'$ and $F'$, respectively. By Theorem \ref{ThmHigsonRel.Closed.Sets}, there is a Higson function $h\colon X\to[0,1]$ such that $h\restriction E'=0$ and $h\restriction F'=1$.
Therefore, if   $\phi\colon C(\nu X)\to C(\calOrb(\partial X))$ denotes the  $^*$-isomorphism in Proposition \ref{PropCnuXCOrbX}, then $\phi(h+c_0(X))$ is a function in $C(\calOrb(\partial X))$ with   the desired properties. \end{proof}

\section{Pseudo-orbits of $(\partial X,\PT(X))$}\label{SectionPseudo}

This section deals with yet another notion of equivalence relation on $\partial X$ determined by dynamics.
Results of this section were inspired by \cite[Chapter 4]{Akin1993GSM}  and   \cite[Theorem 5.3]{Brian2015FundM}.

\begin{definition}\label{DefinitionPseudoOrbit}
    Let $X$ be a u.l.f.\ metric space.
    \begin{enumerate}
        \item Let $\xi,\eta\in \partial X$, $\cU$ be an open cover of  $\partial X$,  $n\in\N$, and $(U_i)_{i=1}^n$ be subsets in $\cU$. We call $(U_i)_{i=1}^n$  an  \emph{$\cU$-chain from $\xi$ to $\eta$ of length $n$} if \[\xi\in U_1,\ \eta\in U_n,\ \text{ and }\ \Orb(U_i)\cap \Orb(U_{i+1})\neq \emptyset\] for all $i\in \{1,\ldots, n-1\}$.
        \item Let $\xi, \eta\in \partial X$. We say that \emph{there are chains (of length $n$) from $\xi$ to $\eta$} if  for all open covers $\cU$ of $\partial X$ there is a $\cU$-chain (of length $n$) from $\xi$ to $\eta$.
        \item Given  $\xi\in \partial X$, the \emph{pseudo-orbit of $\xi$} is
        \[\pOrb(\xi)=\{\eta\in \partial X\colon  \text{there are chains from }\xi\text{ to }\eta\}.\]
    \end{enumerate}
    This defines the \emph{pseudo-orbit equivalence relation} on $\partial X$, denoted by $\pOrb$.
\end{definition}

\begin{remark}
  While we have defined pseudo-orbits only for elements $\xi$ in the corona  $\partial X$, one could replicate  Definition \ref{DefinitionPseudoOrbit} for any element in $\beta X$. However, this would not be interesting since, in $\beta X$, all pseudo-orbits would be the whole $\beta X$. As it will be clear below, this is far from the case when considering only the corona $\partial X$.
\end{remark}

\begin{remark}\label{RemarkConnectedComponentsTop} Recall the following basic fact in general topology: Given a topological space $(M,\tau)$ and $x,y\in M$, the elements $x$ and $y$ are in the same connected component of $M$ if and only if for all open covers $\cU$ of $M$ there is a finite sequence $(U_i)_{i=1}^n$ in $\cU$ such that $x\in U_1$, $y\in U_n$, and $U_i\cap U_{i+1}\neq \emptyset$ for all $i\in \{1,\ldots, n-1\}$. Therefore, $\xi,\eta\in \partial X$ are in the same pseudo-orbit if and only if $\Orb(\xi)$ and $\Orb(\eta)$ are in the same connected component of $\calOrb(\partial X)$.
\end{remark}

\begin{example}
    Given  a u.l.f.\ metric space $X$ and $\xi, \eta\in \partial X$, the existence of chains of length $1$ from $\xi$ to $\eta$ is equivalent to $\xi=\eta$.
\end{example}

\begin{example}\label{ExampleChainLengthNhogson}
    Let $X$ be a u.l.f.\ metric space and $\xi, \eta\in \partial X$. As a consequence of Theorem \ref{ThmHigsonRel},  $\HOrb(\xi)=\HOrb(\eta)$ if and only if there are  chains of length $2$ from $\xi$ to $\eta$.
\end{example}

By the previous example, we have the following chain of inclusions
\[\Orb_{|\partial X}\subseteq \qOrb_{|\partial X}\subseteq \HOrb\subseteq \pOrb\]
Example \ref{ExampleChainLengthNhogson} can   be generalized:

\begin{proposition}
Given  a u.l.f.\ metric space $X$ and $\xi, \eta\in \partial X$. Then, $\HOrb(\xi)=\HOrb(\eta)$  if and only if there is $n\in\N$ so that there are  chains of length $n$ from $\xi$ to $\eta$.
\end{proposition}

\begin{proof}
    The forward direction is immediate from Example \ref{ExampleChainLengthNhogson}. For the backwards direction, suppose there are chains of length $n$ from $\xi$ to $\eta$ and fix a Higson function $h\colon X\to \C$. Given $\eps>0$, let $\cU$ be an open cover of $X$ such that $|h(\zeta)-h(\zeta')|\leq \eps$ for all $U\in \cU$ and all $\zeta,\zeta'\in U$. Since Higson functions are constant on orbits, the existence of these chains from $\xi$ to $\eta$ implies that $|h(\xi)-h(\eta)|\leq n\eps$. As $\eps$ was arbitrary, we conclude $h(\xi)=h(\eta)$.
\end{proof}

The following proposition is immediate.

\begin{proposition}
    Let $X$ be a u.l.f.\ metric space. Then $\pOrb(\xi)$ is closed for all $\xi\in \partial X$. \qed
\end{proposition}

\begin{theorem}\label{ThmChainTransitiveSpaces}
    Let $X$ be a u.l.f.\ metric space. Then $\partial X$ has only one pseudo-orbit  if and only if for every partition $X=A\sqcup B$ into infinite subsets,   there is $r>0$ such that $N_r(A)\cap B$ is infinite.
\end{theorem}

\begin{proof}
For the forward direction, consider a partition    $X=A\sqcup B$ into infinite subsets and pick  $\xi,\eta\in \partial X$  such that $A\in \xi$ and $B\in \eta$. Consider the open cover  $\cU=\{\overline A,\overline{B}\}$ of $X$. As $\partial  X$ has only one pseudo-orbit,
there is an $\cU$-chain from $\xi$ to $\eta$. By the definition of $\cU$, $\Orb(\overline A)\cap \Orb(\overline B)\neq \emptyset$. So, there are nonprincipal ultrafilters $\zeta\in \overline A$,  $\zeta'\in \overline B$, and a partial translation $f$ of $X$ with $\dom(f)\in \zeta$ such that $\overline f(\zeta)=\zeta'$. Letting $r=\sup_{x\in \dom(f)}d(x,f(x))$, it follows that $N_r(A)\cap B $ is infinite.

Suppose now $\partial X$ has more than one pseudo-orbit. So, there are   $\xi,\eta \in \partial X$ and an open cover $\cU$ of $\partial X$ for which   there are no $\cU$-chains from $\xi$ to $\eta$.
 Let
\[\cU_0=\{U\in \cU\colon \exists \text{ a }\cU\text{-chain from }\xi\text{ to some } \zeta\in U\} \ \text{ and } \ \cU_1=\cU\setminus \cU_0.\]
Notice that if $U'\in\cU_0$, then there are $\cU$-chains from $\xi$ to any $\zeta'\in U'$. Therefore, letting
\[U=\bigcup \cU_0\ \text{ and } \ V=\bigcup \cU_1,\]
it follows that $U$ and $V$ are disjoint. So, $X=U\sqcup V$ is a partition of $X$ into open subsets so that $\xi\in U$ and $\eta\in V$. Since $V$ is both closed and open, it follows from compactness that there are $A_1,\ldots, A_n\subseteq X$ such that $V=\bigcup_{i=1}^n\overline{A_i}$. Letting  $A=\bigcup_{i=1}^nA_i$, we have $\overline A=\bigcup_{i=1}^n\overline{A_i}$. So,
\[U=\overline{X\setminus A} \ \text{ and }\ V=\overline A.\] As $\eta\in \overline A$ and $\eta$ is nonprincipal, $A$ is infinite. Also, as $\xi\not\in \overline A$ and $\xi$ is nonprincipal, $X\setminus A$ is infinite.

We are left to notice that $N_r(A)\setminus A$ is finite  for all $r>0$. Indeed, if this were not the case, then there would be a partial translation $f$ of $X$ with infinite domain and such that  $\dom(f)\subseteq X\setminus A$ and $\mathrm{im}(f)\subseteq A$. Suppose  $\eta\in \partial X$ is such that $\dom(f)\in \eta$. As $\eta\in \overline{X\setminus A}=U$, there is an $\cU$-chain from $\xi$ to $\eta$. Therefore, there must also be an $\cU$-chain from $\xi$ to $\overline f(\eta)$. However, since $\mathrm{im}(f)\subseteq A$, $\overline f(\eta)\in \overline A=V$. So, there cannot be an $\cU$-chain from $\xi$ to $f(\eta)$; contradiction.
\end{proof}

\begin{corollary}\label{CorPseudoOrbN}
Let $X$ be either  $\N^n$, for $n\in\N$, or $\Z^n$, for $n\in\N$ with $n>1$.  Then, $\partial X$ has only one pseudo-orbit.\footnote{When $\partial X$ has only one pseudo-orbit, $X$ is called \emph{chain transitive}, see \cite{Akin1993GSM}.} \qed
\end{corollary}

\begin{corollary}\label{CorPseudoOrbF}
    Let $n\in\N$ and $\mathbb F_n$ be the free group with $n$ generators. Then $\partial \mathbb F_n$ has more than one pseudo-orbit.  \qed
\end{corollary}

\section{GNS representations and the dynamical system $(\beta X,\PT(X))$}\label{SectionGNSRep}

 Let $X$ be a u.l.f.\ metric space and $\xi\in \beta X$. Throughout this section,  $\varphi_\xi$ and $\pi_\xi$ will be  as in Definition \ref{DefiGNSRep}. Let us define them precisely for the readers convenience:
 $\varphi_\xi$ is the state on $\cstu(X)$ given by
 \[\varphi_\xi(a)=E(a)(\xi)\]
for all $a\in \cstu(X)$; where $E\colon \cstu(X)\to \ell_\infty(X)$ denotes the canonical conditional expectation. In order to define the representation $\pi_\xi$, we define a pseudo-inner product $\langle \cdot,\cdot \rangle_\xi$ on $\cstu(X)$ by letting
\[
  \langle a,b\rangle_\xi  = \varphi _\xi (b^*a).
  \]
for all $a,b\in \cstu(X)$. Then, letting
\[N_\xi=\{a\in \cstu(X)\colon \langle a,a\rangle =0\},\]
$\cstu(X)/N_\xi$ becomes an inner product space and we  denote  by $H_\xi$ the Hilbert space obtained by completing $\cstu(X)/N_\xi$. By abuse of notation, we still denote the inner product on $H_\xi$ by $\langle \cdot,\cdot \rangle_\xi$ and we let $\|\cdot\|_\xi$ denote the norm on $H_\xi$ induced by this inner product.

  \begin{definition}
  Let $X$ be a u.l.f.\ metric space and $\xi\in \beta X$.   The \emph{GNS representation  $\pi _\xi $} is the $^*$-homomorphism $\pi_\xi\colon \cstu(X)\to \cB(\ell_2(H_\xi))$ given by
\[\pi_\xi(a)(b+N_\xi)=ab+N_\xi\]
for all $a,b\in \cstu(X)$.
  \end{definition}

We shall now identify a canonical orthonormal basis for $H_\xi$. We start noticing that, since $\cstu(X)$ is the closed linear span of  $\{v_f\colon f\in \PT(X)\}$, we  have
  \begin{equation*}
  H_\xi=\overline{\text{span}}\{v_f+N_\xi\colon f\in \mathrm{PT}(X)\}.
  \end{equation*}
Moreover, we only need to consider partial translations whose domain are in $\xi$ to describe $H_\xi$. Precisely, if   $f\in \PT(X)$,  then
  \begin{equation}
  \|v_f+N_\xi\|^2_\xi=
  \langle v_f+N_\xi,v_f+N_\xi\rangle_\xi  = \varphi _\xi (v_f^*v_f) = \varphi _\xi (1_{\dom(f)}),
\end{equation}
where for a given $A\subseteq X$ we let $1_A$ denote the projection in $\ell_\infty(A)\subseteq \cB(\ell_2(X))$ with $1$'s in the diagonal coordinates indexed by elements in $A$ and zero elsewhere.
So, by the definition of $\varphi_\xi$,
\begin{equation}
  \label{norm}\|v_f+N_\xi\|^2_\xi=\left\{\begin{array}{ll}
 1  ,  & \ \text{ if  } \ \dom(f)\in \xi,\\
0,      &\ \text{ if }\ \dom(f)\not\in \xi.
\end{array}\right.
\end{equation}
  Therefore, we have a better description for $H_\xi$:
  \begin{equation}
  \label{Gene}
  H_\xi=\overline{\text{span}}\{v_f+N_\xi\colon f\in \mathrm{PT}(X)_\xi\}.
  \end{equation}

\begin{proposition} \label{Basis}
Let $X$ be a u.l.f.\ metric space, $\xi\in \beta X$, and   $f, g\in \mathrm{PT}(X)_\xi$. Let
  \[
  A=\{x\in \dom(f)\cap \dom(g)\colon f(x)=g(x)\}.
  \]
  \begin{enumerate}
  \item\label{BasisItem1} If $A\notin \xi $, then $v_f+N_\xi\perp v_g+N_\xi$.
    \item\label{BasisItem2} If $A\in \xi $, then $v_f+N_\xi=v_g+N_\xi$.
  \end{enumerate}
\end{proposition}
\begin{proof}
By the definition of $\varphi_\xi$, we have
  \[
  \langle v_f+N_\xi, v_g+N_\xi\rangle_\xi  =
  \varphi_\xi (v_g^*v_f) =
  \varphi _\xi (v_{g^{-1}f})=1_A(\xi).
  \]
So,    $ \langle v_f+N_\xi, v_g+N_\xi\rangle_\xi$ equals $1$ if $A\in \xi$ and zero otherwise. This  proves  \eqref{BasisItem1}.
On the other hand, when $A\in \xi $,
we have
  \[
  \|v_f-v_g+N_\xi\|_\xi^2 =
  2 - 2\Re\langle v_f+N_\xi, v_g+N_\xi\rangle_\xi  = 0.
  \]
  So,  $v_f+N_\xi=v_g+N_\xi$ and \eqref{BasisItem2} is proved.
\end{proof}

In light of  \eqref{norm}, \eqref{Gene}, and Proposition \ref{Basis},  \[\{v_f+N_\xi\colon f\in \PT(X)_\xi\}\] is a maximal orthonormal subset of $H_\xi$. We shall now describe how to canonically identify partial translations given rise to the same elements of $H_\xi$.

\begin{proposition}\label{Propeetawelldef}
    Let $X$ be a u.l.f.\ metric space and $\xi\in \beta X$. Let  $\eta\in \Orb(\xi)$ and suppose $f,g\in \PT(X)_\xi$ are such that $\eta=\overline f (\xi)=\overline g(\xi)$. Then,
    $v_f+N_\xi=v_g+N_\xi$.
\end{proposition}

\begin{proof}
By \cite[Lemma 3.4]{SpakulaWillett2017}, the hypothesis of the proposition  implies that
\[\{x\in \dom(f)\cap \dom(g)\colon f(x)=g(x)\}\in \xi. \]
It then follows from  Proposition \ref{Basis} that $v_f+N_\xi=v_g+N_\xi$.
\end{proof}

\begin{definition}
    Let $X$ be a u.l.f.\ metric space and $\xi\in \beta X$. If $\eta\in \Orb(\xi)$, we define $e_\eta\in H_\xi$ by letting
    \[e_\eta=v_f+N_\xi,\]
    where $f\in \PT(X)_\xi$ is such that $\overline f(\xi)=\eta$. By Proposition \ref{Propeetawelldef}, this is well defined.
\end{definition}

\begin{proposition}\label{DescrRep}
Let $X$ be a u.l.f.\ metric space and $\xi\in \Orb(\xi)$.   The set
  \[
  \{e_\eta \in H_\xi\colon \eta \in \Orb(\xi )\}
  \]
  is an orthonormal basis for  $H_\xi$.  Moreover, if $f\in \PT(X)$ and  $\eta \in \Orb(\xi )$, we have:
  \begin{enumerate}
  \item \label{DescrRepItem1} \label{ActionOfPtrans} If $\dom(f)\in \eta$,   then  $\pi _\xi (v_f) e_\eta  = e_{\overline f(\eta )}$.
  \item  \label{DescrRepItem2} If $\dom(f)\not\in\eta $,  then  $\pi _\xi (v_f) e_\eta  = 0$.
\end{enumerate}
 Moreover, if $a\in \ell _\infty (X)$, then
\[\pi _\xi (a)e_\eta  =  a(\eta )e_\eta ,\]
 where here we are seeing $a$ as an element in $C(\beta X)$ under the canonical identification $\ell_\infty(X)\cong C(\beta X)$.\end{proposition}

\begin{proof}
  The fact that this set forms an orthonormal basis of $H_\xi$ follows from \eqref{norm} and Proposition \ref{Basis}. Let now  $f\in \PT(X)$,   $\eta \in \Orb(\xi )$,  and  pick $g$ in $\mathrm{PT}(X)_\xi$ such that $\eta =\overline g(\xi )$.

  Suppose $\dom(f)\in \eta$. Then
    there is some $A\in \xi $, such that $A\subseteq \dom (g)$ and $g(A) \subseteq  \dom(f)$.
Consider the composition $fg$ defined where possible. Then,   $A\subseteq \dom(fg)$ and we have that  $fg\in \mathrm{PT}_\xi$.
Therefore,
  \[
  \pi _\xi (v_f) e_\eta  =   \pi _\xi (v_f) (v_g+N_\xi) = v_{fg}+N_\xi= e_{\overline {fg}(\xi )} = e_{\overline f(\overline g(\xi ))} = e_{\overline f(\eta )}
  \]
and  \label{DescrRepItem1} follows.

Suppose now that   $\dom(f)\not\in \eta$. Then  $\dom(fg) \notin \xi $. Indeed, otherwise we would have that $g(\dom(fg))\in \overline g(\xi)=\eta$ and, as $g(\dom(fg))\subseteq \dom(f)$, this would imply that $\dom(f)\in \eta$. Finally, as  $\dom(fg)\notin \xi$,
\[
  \pi _\xi (v_f) e_\eta  =   \pi _\xi (v_f) (v_g+N_\xi) = v_{fg}+N_\xi=0.\]
So, we obtain \eqref{DescrRepItem2}.

For the last statement, assume first that $a$ is idempotent.  Then $a$ is necessarily  the characteristic function of some  subset
$A\subseteq X$, in which case $a$ coincides with some  $v_f$ and the result follows from the previous items.  The general case is then a
consequence of the fact that the idempotent elements of $\ell _\infty (X)$ span a dense subspace.
\end{proof}

\begin{remark}\label{RemarkGNSRepEquivOrb}
Proposition  \ref{DescrRep} implies that, up to equivalence of representations, we can assume   $\pi _\xi $ is  the representation of
  $\cstu(X)$ on $\ell ^2 (\Orb(\xi ))$ determined by
  \begin{enumerate}
      \item $\pi_\xi(v_f)e_\eta=e_{\overline f(\eta)}$ for all  $\eta\in \Orb(\xi)$ and all $f\in \PT(X)_\eta$, and
\item      $\pi_\xi(v_f)e_\eta= 0$ for all  $\eta\in \Orb(\xi)$ and all $f\in \PT(X)\setminus \PT(X)_\eta$.
  \end{enumerate}\end{remark}

For the next proposition, we recall some basic $\mathrm C^*$-algebras terminology.

\begin{definition}\label{DefiIrreducibleRepPrimitive}
Let $A$ be a $\mathrm C^*$-algebra and $\varphi\colon A\to \cB(H)$ be a representation of it on a Hilbert space $H$. We say that $\varphi$ is irreducible if there is no nontrivial proper subspace $H'\subseteq H$  such that $\varphi(a)[H']\subseteq H'$ for all $a\in A$. An ideal of $A$ is called \emph{primitive} if it is the kernel of an irreducible representation.
\end{definition}

\begin{proposition}\label{PropGNSIrreducible}
Let $X$ be a u.l.f.\ metric space and $\xi\in \beta X$. Then, the GNS representation $\pi_\xi$ is irreducible. In particular, the ideal $\ker(\pi_\xi)$ is primitive.
\end{proposition}

\begin{proof}
Throughout the proof, we assume $\pi_\xi$ is given as in Remark \ref{RemarkGNSRepEquivOrb}. Using Schur's lemma it suffices to prove that any bounded operator on $\ell ^2(\Orb(\xi ))$, which commutes with the
range of $\pi _\xi $, is necessarily a multiple of the identity.

Fix $b\in \cB(\ell_2(\Orb(\xi)))$. Let $\eta $ and $\zeta $ be distinct points in $\Orb(\xi )$ and choose $a\in \ell_\infty (X)\cong C(\beta X)$  such that $ a(\eta )=1$ and $ a(\zeta )=0$.
By Proposition \ref{DescrRep},
  $$
  \big \langle b e_\zeta ,e_\eta \big \rangle_\xi   =
  \big \langle b e_\zeta ,\pi _\xi (a)e_\eta \big \rangle_\xi  =
  \big \langle \pi _\xi (a^*)b e_\zeta ,e_\eta \big \rangle_\xi  =
  \big \langle b\pi _\xi (a^*)e_\zeta ,e_\eta \big \rangle_\xi  = 0,
  $$
  so it follows that $b$ is a diagonal operator. We claim that all of the diagonal entries of $b$ coincide with
$\big\langle be_\xi ,e_\xi \big \rangle_\xi $.  To see this, pick any $\zeta \in\Orb(\xi )$, so there exists   $f\in \PT(X)_\xi$ such that   $\overline f(\xi )=\zeta $.  By Proposition \ref{DescrRep}, we have  $\pi _\xi (v_f) e_\xi =e_\zeta $ and it follows that
  \begin{align*}
  \big\langle b e_\zeta ,e_\zeta \big \rangle_\xi & =
  \big\langle  b\pi _\xi (v_f)e_\xi ,\pi _\xi (v_f)e_\xi \big \rangle_\xi \\
  &=
  \big\langle \pi _\xi (v_f)be_\xi ,\pi _\xi (v_f)e_\xi \big \rangle_\xi \\
  &=
  \big\langle be_\xi ,\pi _\xi (v_f^*v_f)e_\xi \big \rangle_\xi \\
  &=
  \big\langle b e_\xi ,e_\xi \big \rangle_\xi .
  \end{align*}
  This shows that $b$ is a scalar multiple of the identity operator and hence that $\pi _\xi $ is irreducible.
\end{proof}

\subsection{Characterizing orbits and quasi-orbits}\label{SubsectionCharacterizingOrbits}

\begin{theorem}
  Let $X$ be a u.l.f.\ metric space and   $\xi,\eta\in \beta X$. Let $\pi_\xi,\pi_\eta$ be the GNS representations given in Definition \ref{DefiGNSRep}. \begin{enumerate}
      \item\label{ThmOrbitEquivIFFGNSEquiv.Item1}   $\Orb(\xi )=\Orb(\eta )$ if and only if    $\pi _\xi $ and $\pi _\eta $ are unitarily equivalent.
  \end{enumerate} \label{ThmOrbitEquivIFFGNSEquiv.Inthetext}
  \end{theorem}

\begin{proof}
If $\Orb(\xi)=\Orb(\eta)$, it follows from Remark \ref{RemarkGNSRepEquivOrb} that $\pi_\xi$ and $\pi_\eta$ are equivalent.

Suppose now $\pi_\xi$ and $\pi_\eta$ are equivalent and let us show that  $\Orb(\xi )= \Orb(\eta )$.

\begin{claim}
We have
\begin{align*}
    \Orb(\xi)=\big\{\zeta\in \beta X\colon &\exists f\in \PT(X)_\xi,\ \forall a\in \ell_\infty(X)\  \\
    &\pi_\xi(a)(v_f+N_\xi)=  a(\zeta)(v_f+N_\xi)\big\}.
    \end{align*}
\end{claim}

\begin{proof}
    If $\zeta\in \Orb(\xi)$, Proposition \ref{Basis} gives that
    \[\pi_\xi(a)e_\zeta= a(\zeta)e_\zeta\]
for all $a\in \ell_\infty(X)$. As $e_\zeta=v_f+N_\xi$ for some $f\in \PT(X)_\xi$,  $\zeta$ belongs to the subset on  the right-hand side above.

Suppose now $\zeta\not\in \Orb(\xi)$ and pick $f\in \PT(X)_\xi$. Since $\zeta$ is not in the orbit of $\xi$, $\overline f(\xi)\neq \zeta$. So, there is $A\subseteq X$ such that $A\in \zeta$ and $X\setminus A\in \overline f(\xi)$.   In particular,      the projection $1_A\in \ell_\infty(X)$ satisfies $1_A(\zeta)=1$. On the other hand, $X\setminus A\in \overline f(\xi)$, we have that ${f^{-1}(X\setminus A)}\in \xi$, so, ${f^{-1}(A)}\not\in \xi$. Therefore, as  $v^*_f1_Av_f=1_{f^{-1}(A)}$, we have
\[\|1_Av_f+N_\xi\|^2=\varphi_\xi(v_f^*1_Av_f)=E(v^*_f1_Av_f)(\xi)=0.\]
We conclude that $\pi_\xi(a)(v_f+N_\xi)=0$, so it cannot equal $  a(\zeta)(v_f+N_\xi)=v_f+N_\xi$.
\end{proof}

Since $H_\xi$ is the closed span of $\{v_f+N_\xi\colon f\in \PT(X)_\xi\}$ (see \eqref{Gene}), the previous claim together with  straightforward approximating arguments gives that
\begin{align*}   \Orb(\xi)=\big\{\zeta\in \beta X\colon &\exists w\in H_\xi\setminus\{0\},\ \forall a\in \ell_\infty(X)\  \\
    &\pi_\xi(a)w=a(\zeta)w\big\}.
    \end{align*}
Analogously, the same holds for any element in $\beta X$, so
\begin{align*}   \Orb(\eta)=\big\{\zeta\in \beta X\colon &\exists w\in H_\eta\setminus\{0\},\ \forall a\in \ell_\infty(X)\  \\
    &\pi_\xi(a)w=  a(\zeta)w\big\}.
    \end{align*}
Therefore, since by hypothesis there is a unitary $u\colon H_\xi\to H_\eta$ such that $\pi_\xi(a)=u^*\pi_\eta(a)u$ for all $a\in \cstu(X)$, this shows that $\Orb(\xi)=\Orb(\eta)$.
\end{proof}

Here we will characterize the kernel of the GNS representation $\pi _\xi $ in terms of the quasi-orbits of the dynamical system $(\beta X,\PT(X))$. For that, we recall that $E\colon \cstu(X)\to \ell_\infty(X)$ denotes the canonical conditional expectation. Similarly, given $\xi\in \beta X$, we let $E_\xi\colon \cB(\ell_2(\Orb(\xi)))\to \ell_\infty(\Orb(\xi))$ denote the canonical conditional expectation. \

\begin{lemma} \label{CovarCondexps}
Let $X$ be a u.l.f.\ metric space and $\xi\in\beta X$. Consider the representation $\pi_\xi$ as a representation on $\ell_2(\Orb(\xi))$ as in   Remark \ref{RemarkGNSRepEquivOrb}. Then, for every $a\in  \cstu(X)$, we have
  $$
  E_\xi (\pi _\xi (a)) = E(a)\restriction_{\Orb(\xi )},
  $$
  where $E(a)$ is being regarded as a  continuous function on $\beta X$ via canonical isomorphism $\ell_\infty (X)\cong C(\beta X)$.
\end{lemma}

\begin{proof}
  It   suffices to prove the result only for $a=v_f$, where $f$ is some partial translation of $X$.  In that case,
  given $\eta \in\Orb(\xi )$, we choose some partial translation $g$ in $\mathrm{PT}(X)_\xi$, such that $\eta =\overline g(\xi )$. So
  \begin{align}\label{Eq10Oct24}
  E_\xi (\pi _\xi (v_f))(\eta ) &=
  \langle \pi _\xi (v_f)e_\eta , e_\eta \rangle_\xi \\ & =
  \langle \pi _\xi (v_f)w_g, w_g\rangle_\xi  \notag \\ &=
  \langle w_{fg}, w_g\rangle_\xi \notag \\
  &=
  \varphi _\xi (v_g^*v_{fg}) \notag \\ &=
  \varphi _\xi (v_{g^{-1}fg})\notag \\ & =
  E(v_{g^{-1}fg})(\xi).\notag
  \end{align}
  Letting   $A$ be the fixed point set for $f$,  namely
  $$
  A=\{x\in \dom (f)\colon f(x)=x\},
  $$
  then the fixed point set for $g^{-1}fg$ is given by
  \begin{align*}
  B &=
  \{x\in \dom (g^{-1}fg)\colon g^{-1}fg(x)=x\}\\
  &=
  \{x\in \dom (g)\colon g(x)\in A\} \\
  &=
  g^{-1}(A).
  \end{align*}
  Therefore, $E(v_{g^{-1}fg}) =1_{g^{-1}(A)}$. As $A\in \overline g(\xi)$ if and only if $g^{-1}(A)\in \xi$,  \eqref{Eq10Oct24} implies
  \[ E_\xi (\pi _\xi (v_f))(\eta )=\left\{\begin{array}{ll}
1,       &\text{ if }\ A\in \overline g(\xi),  \\
     0,  & \text{ if }\ A\not\in \overline g(\xi).
  \end{array}\right.\]
  On the other hand, we have that $E(v_f)= 1_A$, so
  $$
  E(v_f)(\eta) = 1_A(\eta ) =
  \left\{\begin{array}{ll}
1,       &\text{ if }\ A\in \eta,  \\
     0,  & \text{ if }\ A\not\in \eta.
     \end{array}\right.
  $$
This shows that $ E_\xi (\pi _\xi (v_f))(\eta ) =   E(v_f)_\eta ,
  $ and we are done.
\end{proof}

\begin{lemma}\label{LemmaConvergingNetsOrbitEquiv}
    Let $X$ be a u.l.f.\ metric space, $\xi,\eta\in \beta X$, with $\eta\in \Orb(\xi)$,  and $(\xi_i)_{i\in I}$ be a net in $\beta X$ converging to $\xi$. Then, there is a net $(\eta_{j})_{j\in J}$ converging to $\eta$ such that each $\eta_j$ is in the orbit of some $\xi_i$.
\end{lemma}

\begin{proof}
    As $\eta\in \Orb(\xi)$, there is a partial translation $f\colon A\to B$ such that $A\in \xi$, $B\in \eta$ and $\overline f(\xi)=\eta$. Let now $C\in \eta$ and let us find $\eta'\in U_C$ which is in the orbit of some $\xi_i$. As $C\in \eta$, we can assume $C\subseteq B$. Then, as $\overline f(\xi)=\eta$, $f^{-1}(C)\in \xi$. Hence, as $\xi=\lim_{i,I}\xi_i$, we can pick $i\in I$ with $\xi_i\in U_{f^{-1}(C)}$. But then $\overline f(\xi_i)\in U_C$ as desired.
\end{proof}

\begin{corollary}
    \label{Kernels}
Let $X$ be a u.l.f.\ metric space and $\xi\in \beta X$. Then,
  $$
  \ker (\pi _\xi ) = \{a\in \cstu(X)
\colon E(a^*a)\restriction \Orb(\xi ) = 0\}.
  $$
In particular, if $F\subseteq  \beta X$, then
\[
 \bigcap_{\eta\in F} \ker(\pi _\eta )\subseteq \ker(\pi _\xi )
 \ \Longleftrightarrow\
   \xi\in  \COrb(F ).\]
\end{corollary}

\begin{proof}
Given $a$ in $\cstu(X)$, Proposition  \ref{CovarCondexps} gives
  \begin{align*}
  \pi _\xi (a)=0 &\Leftrightarrow
  \pi _\xi (a^*a)=0 \\
  & \Leftrightarrow
  E_\xi (\pi _\xi (a^*a))=0 \\
  &\Leftrightarrow
  E(a^*a)\restriction {\Orb(\xi )}=0
  \end{align*}
  and the first statement follows.

 Let now $F\subseteq \beta X$. Suppose $\xi\in \COrb(F)$ and let $a\in \cap_{\eta\in F} \ker(\pi _\eta )$. By the above, in order to show that $a\in \ker(\pi_\xi)$,  we must show  $E(a^*a)\restriction \Orb(\xi)=0$. Pick $\eta\in \Orb(\xi)$. As $\xi \in \COrb(F)$,  Lemma \ref{LemmaConvergingNetsOrbitEquiv} implies that $\eta$ is the limit of a net in $\Orb(F)$. As $a$ is in  $\cap_{\eta\in F} \ker(\pi _\eta )$, the above implies that    $E(a^*a)$ vanishes on $\Orb(F)$. By continuity, we conclude that $E(a^*a)(\eta)=0$.

 Suppose now $\xi\not\in \COrb(F)$. As $\beta X$ is a compact Hausdorff space, Urysohn's lemma gives us that there is a continuous function $a\colon \beta X\to [0,1]$ such that $a(\xi)=1$ and $a\restriction \COrb(F)=0$. Under the identification $\ell_\infty(X)\cong C(\beta X)$, we see $a$ as an element  in $\ell_\infty(X)$. Since $E(a)=a$, the above implies that   $a^{1/2}\in  \bigcap_{\eta\in F} \ker(\pi _\eta )$ and $a\not\in \ker(\pi_\xi)$.
\end{proof}

\begin{theorem}
  Let $X$ be a u.l.f.\ metric space and   $\xi,\eta\in \beta X$. Let $\pi_\xi,\pi_\eta$ be the GNS representations given in Definition \ref{DefiGNSRep}. \begin{enumerate}\setcounter{enumi}{1}      \item\label{ThmOrbitEquivIFFGNSEquiv.Item2} $\qOrb(\xi )=\qOrb(\eta )$ if and only if   $\ker(\pi _\xi)=\ker(\pi _\eta)$.
  \end{enumerate} \label{ThmOrbitEquivIFFGNSEquiv.Inthetext2}
  \end{theorem}

\begin{proof}
    This follows immediately from Corollary \ref{Kernels}.
\end{proof}

\subsection{Invariant closed subsets of $\beta X$ and the ideal structure of $\cstu(X)$}\label{SubSubSectionIdeals}  We shall now explore some consequences of Corollary \ref{Kernels}.  Precisely, let
\[\cF(\beta X)=\{F\subseteq \beta X\colon F\text{ is closed and invariant}\}\] and
\[\cI(\cstu(X))=\{I\subseteq \cstu(X)\colon I \text{ is an ideal of }\cstu(X)\}.\]
We shall notice   that the map
\[\xi\in \beta X\mapsto \ker(\pi_\xi)\in \cI(\cstu(X))\]
can be thought of as   the restriction of a well-studied canonical map relating the invariant closed subsets of $\beta X$ to the ideals of $\cstu(X)$ (see \eqref{Eq.IOrbRelation} below). Moreover, this map is a bijection between these objects if $X$ has property A (Proposition \ref{PropBijPropA}). We should point out however that, except for the connections with Corollary \ref{Kernels} and, perhaps, Proposition \ref{Prop.Irrd.Gives.Prime}, the results presented now   are known to experts and even appear in the literature in some form or another (see \cite{ChenWang2004JFA,ChenWang2005,WangZhang:Ghostly}). We chose however to present proofs here in details for a couple of reasons: (1) we believe it serve a   didactic purpose; (2) results in the literature usually appear in the language of groupoids and use of ideals of coarse structures, making necessary to introduce   a fair amount of terminology in order to properly use them here; and (3) these results will be  important in Section \ref{SubsectionIrreducibleSet} below.

  Given $F\in \cF(\beta X)$,   define
\begin{equation}\label{Eq.IFromF}
I_F=\{a\in \cstu(X)\colon E(a^*a)\restriction F=0\},\end{equation}
where $E\colon \cstu(X)\to \ell_\infty(X)$ denotes the canonical conditional expectation.

\begin{proposition}
    Let $X$ be a u.l.f.\ metric space and $F\in \cF(\beta X)$.  Then $I_F$ is an ideal of $\cstu(X)$.
\end{proposition}

\begin{proof}
  As $\cstu(X)$  is the norm closure of
\[\{v_fb\colon f\in \PT(X),\ b\in \ell_\infty(X)\},\]
it is enough to show that $v_fba,av_fb\in I_F$ for all   $a\in I_F$, $f\in \PT(X)$,   and $b\in \ell_\infty(X)$. Fix such $a$, $f$, and $b$, and let $\xi\in F$. Firstly, as
\[0=E(a^*a)(\xi)=\lim_{x,\xi}\langle a^*a\delta_x,\delta_x\rangle=\lim_{x,\xi}\|a\delta_x\|^2,\] we have that
\begin{align*}E((v_fba)^*v_fba)(\xi)&=\lim_{x,\xi}\langle(v_fba)^*v_fba\delta_x,\delta_x\rangle\\
&=\lim_{x,\xi}\| v_fba\delta_x\|^2\\
&\leq \|b\|^2\lim_{x,\xi}\|a\delta_x\|^2\\
&=0.
\end{align*}
So, $v_fba\in I_F$.

As $E$ is a conditional expectation, we have
\begin{align*}
    E((av_fb)^*av_fb)= E(b^*v_{f^-1}a^*av_fb)=b^*E(v_{f^{-1}}a^*av_f)b.
\end{align*}
The invariance of $F$ under $\PT(X)$ implies that
\begin{equation}\label{Eq.FI,invariant} E(v_{f^{-1}}a^*av_f)(\xi)=E(a^*a)(\overline f(\xi))=0\end{equation}
if $\dom(f)\in \xi$; otherwise the left-hand side clearly equals and the invariance of $F$ plays no role.
Hence,
\[E((av_fb)^*av_fb)(\xi)=b^*(\xi)E(v_{f^{-1}}a^*av_f)(\xi)b(\xi)=0\] and $av_fb\in I_F$.
\end{proof}

By Corollary \ref{Kernels}, the restriction of the assignment
\begin{equation}\label{Eq.Map.F.to.IF}
F\in \cF(\beta X)\mapsto I_F\in \cI(\cstu(X))\end{equation}
to the closed invariant subsets of the form $\COrb(\xi)$ is realized   by
\begin{equation}\label{Eq.IOrbRelation}
I_{\COrb(\xi)}=\ker(\pi_\xi)\end{equation}
for all $\xi\in \beta X$.

Besides the canonical map in \eqref{Eq.Map.F.to.IF} from $\cF(\beta X)$ to $\cI(\cstu(X))$, there are also canonical maps going in the opposite direction. We consider the following: given $I\in \cI(\cstu(X))$, let
\begin{equation}\label{Eq.FFromI}
F_I=\{\xi\in \beta X\colon E(a^*a)(\xi)=0\text{ for all }a\in I\}.
\end{equation}
So, $F_I$ is closed and invariant, i.e., it is an element in $\cF(\beta X)$. Indeed, closeness of $F_I$ is completely immediate and the invariance of it under $\PT(X)$ follows from \eqref{Eq.FI,invariant} above.

Versions of the   next proposition have appeared in the literature in the language of groupoids and using the machinery of ideals of coarse structures. We refer to \cite[Page 208 and Theorem 6.3]{ChenWang2004JFA} for further details.

\begin{proposition}\label{PropBijPropA}
    Let $X$ be a u.l.f.\ metric space with property A. The maps
\[F\in \cF(\beta X)\mapsto I_F\in \cI(\cstu(X))\]
and
\[  I\in \cI(\cstu(X))\mapsto F_I\in \cF(\beta X)\]
    are inverse of each other.
\end{proposition}

\begin{proof}
    Given $F\in \cF(\beta X)$ and $I\in \cI(\cstu(X))$ the inclusions
\[F\subseteq F_{I_F}\ \text{ and }\ I\subseteq I_{F_I}\]
    are completely straightforward. Let us prove the converse inclusions also hold.

    Let $\xi\in F_{I_F}$ and suppose towards a contradiction that $\xi\not\in F$. Then there is a positive $a\in C(\beta X)$ such that $a\restriction F=0$ and $a(\xi)\neq 0$. Hence, $a^{1/2}$ satisfies the same properties and, considering the identification $\ell_\infty(X)\cong C(\beta X)$, we have $a^{1/2}\in I_F$. Since $\xi \in F_{I_F}$, this implies that
    \[a(\xi)=E(a)(\xi)=0;\]
    contradiction.

    We are left to show that $I_{F_I}\subseteq I$.   As $X$ has property A, the ideals of $\cstu(X)$ satisfy the following property: they are the norm closure of their elements with   finite propagation  (see \cite[Theorem 4.4]{ChenWang2005Arch}); such ideals of $\cstu(X)$ are called \emph{geometric ideals}. Therefore, $I_{F_I}$ must satisfy this property and we only need to show that if an operator in  $ I_{F_I}$ has finite propagation, then it belongs to $I$. As $X$ is u.l.f., finite propagation operators are the linear combination of operators of the form $av_f$ for $f\in \PT(X)$ and $a\in \ell_\infty(X)$. Moreover, as $I_{F_I}$ is an ideal, we can also assume each of these $av_f$ belong to  $I_{F_I}$. Indeed, say $b\in I_{F_I}$ has finite propagation $r$. As $X$ is u.l.f., we can write 
    \[X=X_1\sqcup \ldots\sqcup X_k,\]
    where each $X_i$ above is such that $d(x,y)>2r$ for all distinct $x,y\in X_i$. Then, 
    \[b=\sum_{i,j=1}^k1_{X_i}b1_{X_j}\]
and, as $b$ has propagation $r$,  each $1_{X_i}b1_{X_j}$ is of the form $av_f$ for some $a\in \ell_\infty(X)$ and some $f\in \PT(X)$ with $\sup_{x\in \dom(f)}d(x,f(x))\leq r$.

 We have   reduced the problem to show that if $f\in \PT(X)$ and $a\in \ell_\infty(X)$ are such that $av_f\in I_{F_I}$, then $av_f\in I$. Fix such $a$ and $f$. Then \[b=(av_f)^*av_f\in  \ell_\infty(X).\] Consider the canonical identification $\ell_\infty(X)\cong C(\beta X)$ and let  \[J=I\cap \ell_\infty(X),\] so, $J$ is an ideal of   $C(\beta X)$. Ideals of $C(\beta X)$ are well understood and we have that
 \[J= \left\{c\in C(\beta X)\colon c(\xi)=0 \text{ for all }  \xi\in \bigcap_{h\in J}\ker(h)\right\}.\]

Suppose towards a contradiction that   $b\not\in J$. So, there is $\xi\in \bigcap_{h\in J}\ker(h)$ such that $b(\xi)\neq 0$. Let us notice that  $\xi\in F_I$. For that we use property A once again: proceeding as above, since the ideal $I$ is also geometric, in order to show that $E(c^*c)(\xi)=0$ for all $c\in I$, it is enough to show that $E((dv_{g})^*dv_{g})(\xi)=0$ for all $g\in \PT(X)$ and all $d\in \ell_\infty(X)$ with $dv_g\in I$. For such $g$ and $d$,
$(dv_{g})^*dv_{g}\in \ell_\infty(X)$ and therefore  $(dv_{g})^*dv_{g}\in J$. Hence, as $\xi\in \bigcap_{h\in J}\ker(h)$, we  have   \[E((dv_{g})^*dv_{g})(\xi)=(dv_{g})^*dv_{g}(\xi)=0
\]
as desired.

Finally, as $\xi\in F_I$ and $b\in I_{F_I}$, it follows that $b(\xi)=0$. This contradicts our choice of $b$. So, $b\in J$ and, in particular, $b\in I$. By the definition of $b$, we have    $a^*av_f=v_fb\in I$. Letting $a'$ be the element of $\ell_\infty(X)$ whose coordinates are the inverses of the ones of $a^*$ when possible and $0$ otherwise, we have that $av_f=a'v_fb\in I$ as desired.
\end{proof}

As of now, we have been focusing our study on    elements of  $\cF(\partial X)$ of the
form $\COrb(\xi)$   for some $\xi\in \partial X$. We now introduce another class of subsets of $\cF(\beta X)$ which will be fundamental for our analysis of the prime and primitive ideals of uniform Roe algebras in Section \ref{SubsectionIrreducibleSet}.

\begin{definition}\label{Definition.Irreducible.in.Orb}
    Let $X$ be a u.l.f.\ metric space. A nonemptry closed subset of $\calOrb(\beta X)$ is called \emph{irreducible} if it cannot be written as a union  of two nontrivial closed subsets.\footnote{This definition works for topological spaces in general.} As an abuse of notation, we will say that a closed (orbit) invariant  $F\subseteq \beta X$ is \emph{irreducible} if its image in $\calOrb(\beta X)$ under the quotient map is irreducible. In other words, if every time $F=F_0\cup F_1$ for some  closed invariant subsets $F_0,F_1\subseteq \partial X$, then   either $F=F_0$ or $F=F_1$.
\end{definition}

 \begin{example}\label{Example.COrb.Irr}
     Given a u.l.f.\ metric space $X$,   $\COrb(\xi)$ is irreducible for every $\xi\in \beta X$. Indeed, that $\COrb(\xi)$ is closed and invariant it is clear. Now suppose $\COrb(\xi)=F_0\cup F_1$ for closed invariant subsets $F_0$ and $F_1$. Then either $F_0$ or $F_1$ must contain $\xi$. Being closed and invariant, the set containing $\xi$ must be the whole $\COrb(\xi)$.
 \end{example}

The problem of whether every irreducible subset is of the form $\COrb(\xi)$ for some $\xi\in \beta X$ is open and it is the topic of Section \ref{SectionIrred} below (see Problem \ref{Prob.Irred.FORB}). The next proposition shows that at least something weaker happens; recall, as we saw in Proposition \ref{PropHigsonFunctionsAreConstantOnOrbits}, Higson functions are constant on orbits and, in particular, on orbit closures. 

  \begin{proposition}\label{PropIrreducibleOpenInv}
     Let $X$ be a u.l.f.\ metric space and $F\subseteq \partial  X$ be irreducible. Then, $\HOrb(\xi)=\HOrb(\eta)$ for all $\xi,\eta\in F$.
 \end{proposition}

\begin{proof}
    By Theorem \ref{ThmHigsonRel}, it is enough to notice that if $U$ and $V$ are open orbit invariant subsets of $\partial X$ containing $\xi$ and $\eta$, respectively, then $U\cap V\neq \emptyset$. Let $U$ and $V$ be such sets and suppose towards a contradiction that $U\cap V=\emptyset$. Then $F=(F\setminus U)\cup (F\setminus V)$, and $F\setminus U$ and $F\setminus V$ are nonempty, closed,  orbit invariant subsets of $F$ which are not the whole $F$.  This contradicts the irreducibility of $F$.
\end{proof}

  As we see now, in terms of ideals of $\cstu(X)$, irreducible subsets of $\beta X$ correspond to prime  ideals under the map \eqref{Eq.Map.F.to.IF}.

\begin{definition}
    Let $A$ be a $\mathrm C^*$-algebra and $I\subseteq A$ be an ideal. We say that $I$ is \emph{prime} if every time $I_0I_1\subseteq I$ for ideals $I_0,I_1\subseteq A$, we have that either $I_0\subseteq I$ or $I_1\subseteq I$.\footnote{Notice that, as $I_0$ and $I_1$ are ideals, $I_0I_1=I_0\cap I_1$.}
\end{definition}

\begin{proposition}\label{Prop.Irrd.Gives.Prime}
Let $X$ be a u.l.f.\ metric space with property A. Then $F\in \cF(\beta X)$ is irreducible if and only if $I_F$ is prime.
\end{proposition}

\begin{proof}
    This is a straightforward consequence of Proposition \ref{PropBijPropA}. Indeed, it is immediate from the formulas of the map $F\mapsto I_F$ and $I\mapsto F_I$ that \[F\subseteq F' \Rightarrow \ I_{F'}\subseteq I_F\  \text{ and }\ I\subseteq I' \Rightarrow \ F_{I'}\subseteq F_I.\]
    By Proposition \ref{PropBijPropA} this maps are inverses of each other. Hence,
    \[F\subseteq F' \Leftrightarrow \ I_{F'}\subseteq I_F.\]
 Therefore, as $I_{F_0\cup F_1}=I_{F_0}\cap I_{F_1}$ and $F_{I_0\cap I_1}=F_{I_0}\cup F_{ I_{1}}$ for all $F_0,F_1\in \cF(\beta X)$ and all $I_0,I_1\in \cI(\cstu(X))$, the result follows.
\end{proof}

We finish this section showing  that, without the presence of property A, while we do not know if Proposition \ref{Prop.Irrd.Gives.Prime} is true, we do know that the conclusion of Proposition \ref{PropIrreducibleOpenInv} holds for subsets $F_I$ where $I$ is a nontrivial primitive ideal. 

\begin{proposition}
    Let $X$ be a u.l.f.\ metric space and $I$ be a nontrivial
 primitive ideal of $\cstu(X)$. Then  $\HOrb(\xi)=\HOrb(\eta)$ for all $\xi,\eta\in F_I$.
\end{proposition}

\begin{proof}
    As $I$ is primitive, we may pick an irreducible representation $\pi\colon
\cstu(X)\to \cB(H)$ on some Hilbert space $H$ such that $I=\ker(\pi)$. Since $I$ is
nontrivial, it contains $\cK(\ell_2(X))$. So $\pi$ factors through the quotient map
  $$
  q : \cstu(X) \to  \cstu(X)/\cK(\ell_2(X))
  $$
  leading up to an irreducible representation $\pi'$ of $\cstu(X)/\cK(\ell_2(X))$ on
$H$, such that $\pi =\pi '\circ q$.

    Fix $\xi,\eta\in F_I$ and suppose towards a contradiction that there is a Higson function $h\colon X\to \C$ such that $ h(\xi)\neq   h(\eta)$. Recalling that Higson functions map to the center of $\cstu(X)/\cK(\ell_2(X))$ under
the quotient map (see \cite[Proposition
3.6]{BaudierBragaFarahVignatiWillett2023JFA}),   Schur's Lemma implies that
$\pi '(q(h))\in {\mathbb C}$.  So, upon replacing $h$ with $h+\lambda $, for a suitable constant $\lambda $,  we may assume that $\pi '(q(h))=0$. Consequently $h\in I$, whence for every $\zeta $ in $F_I$ we have that
  \[
  0 = E(h^*h)(\zeta ) = |h(\zeta )|^2,
  \]
  so
  $h(\xi )=0=h(\eta )$. \end{proof}

\subsection{Restricting  GNS representations to the Higson functions }\label{SubsectionRestrictionGNSHigFunc}
In this subsection, we revisit Theorem \ref{ThmHigsonRel} and  add yet another characterization of the Higson equivalence relation but now in terms of the GNS representations $\pi_\xi$ of $\cstu(X)$. We restate Theorem \ref{ThmHigsonRel}  below with the extra equivalence, item \eqref{ThmHigsonRel.Item5}, which is proved in this subsection.

\begin{theorem}\label{ThmHigsonRel.Cont}
    Let $X$ be a u.l.f.\ metric space and $\xi,\eta\in \partial X$. The following are equivalent:
    \begin{enumerate}
        \item  $\HOrb(\xi)=\HOrb(\eta)$.
\item
There are no disjoint  open subsets of $\calOrb(\partial X)$ which separate $\Orb(\xi)$ and $\Orb(\eta)$.
\item $\xi \sim_{R}\eta$, where $R\subseteq \partial X\times \partial X$ is the topological closure of the orbit equivalence relation.
\item  $\xi \sim_{R'}\eta$, where $R'\subseteq \partial X\times \partial X$ is the smallest closed equivalence relation which is orbit-invariant.
\item\label{ThmHigsonRel.Item5} $\pi_\xi\restriction C_h(X)$ and $\pi_\eta\restriction C_h(X)$ are unitarily equivalent.
\end{enumerate}
\end{theorem}

The addition of  \eqref{ThmHigsonRel.Item5} to the theorem above will be a consequence of the following proposition and together with previous results.

\begin{proposition}\label{PropCardOrbStableHigsonEquiv}
    Let $X$ be a u.l.f.\ metric space and $\xi,\eta\in \beta X$. If $h(\xi)=h(\eta)$ for all Higson functions $h\colon X\to \C$, then $|\Orb(\xi)|=|\Orb(\eta)|$.
\end{proposition}

\begin{proof}
Suppose  $|\Orb(\xi)|\neq|\Orb(\eta)|$ and let us find a Higson function separating $\xi$ and $\eta$. Without loss of generality, assume $|\Orb(\xi)|<|\Orb(\eta)|$. Since the orbits of elements in $\beta X$ by the partial action of the inverse semigroup of partial translation are at   countable (Proposition \ref{PropOrbitsAreCountable}), $\Orb(\xi)$  must be finite; set $n=|\Orb(\xi)|$. By Theorem \ref{ThmCharacterizationOrbClosed.IntheText}, there is $r>0$ such that
\begin{equation}\label{Eq.PropCardOrbStableHigsonEquiv}
\lim_{x,\xi}d(N_r(x),X\setminus N_r(x))=\infty
\end{equation}
Notice that
\[A=\{x\in X\colon |N_r(x)|=n\}\in \xi.\]
Indeed, if the $\{x\in X\colon |N_r(x)|>n\}$ is in $\xi$, then it is straightforward that $\Orb(\xi)$ has more than $n$ elements. On the other hand, \eqref{Eq.PropCardOrbStableHigsonEquiv} implies that if $f$ is a partial translation of $X$ with $\dom(f)\in \xi$ then, replacing $\dom(f)$ by a smaller subset if necessary, we can assume that   $f(x)\in N_r(x)$ for all $x\in \dom(f)$. Therefore, if $\{x\in X\colon |N_r(x)|<n\}$ is in $\xi$, then $\{x\in X\colon |N_r(x)|=m\}$  is in $\xi$ for some $m<n$ and the above implies that  $\Orb(\xi)$ has at most $m$ elements.

Notice that
\[B=X\setminus \bigcup_{x\in A}N_r(x)\in \eta.\]
Suppose otherwise that $X\setminus B\in \eta$. Then, as $A\in \xi$, \eqref{Eq.PropCardOrbStableHigsonEquiv} implies that
\[\lim_{x,\eta}d(N_r(x),X\setminus N_r(x))=\infty.\]
Then,  the arguments in the previous paragraph  give   that $\Orb(\eta)$ has at most $n$ elements. This contradicts the assumption that $|\Orb(\xi)|<|\Orb(\eta)|$.

 It is clear from  the definition of $B$, $A$, and \eqref{Eq.PropCardOrbStableHigsonEquiv} that
   $d(x,B)\to \infty$ as  $ x{\to} \infty$ in $A$.
 By
Lemma \ref{LemmaExistenceOfHigsonFunctionSeparationSubsetsFromFarAwaySequences}, replacing $A$ by a smaller subset in $\xi$ if necessary, we can find  a Higson function $h\colon X\to [0,1]$ such that $h\restriction \overline A=1$   and $h\restriction \overline B=0$. As $A\in \xi$ and $B\in \eta$,  $ h$ separates the ultrafilters $\xi$ and $\eta$.
\end{proof}

\begin{proof}
    [Proof of Theorem \ref{ThmHigsonRel.Cont},  \eqref{ItemThmHigsonRel1}$\Leftrightarrow$\eqref{ThmHigsonRel.Item5}] As $C_h(X)\subseteq \ell_\infty(X)$,  Proposition \ref{Basis}
shows that
\[\pi_\xi(h)=h(\xi)\mathrm{Id}_{H_\xi}\]
for all $h\in C_h(X)$. Therefore, since $\Orb(\xi)$ and $\Orb(\eta)$ have the same cardinality if $\HOrb(\xi)=\HOrb(\eta)$ (Proposition \ref{PropCardOrbStableHigsonEquiv}), this implies that $\pi_\xi\restriction C_h(X)$ and $\pi_\eta\restriction C_h(X)$ are unitarily equivalent if and only if $h(\xi)=h(\eta)$ for all Higson functions $h\colon X\to \C$. This finishes the proof.
\end{proof}

\subsection{The topologies of $\calOrb(\beta X)$ and $\qcalOrb(\beta X)$}\label{SubsectionTopOrbSpacePrimitive}

Given a u.l.f.\ metric space $X$ and $\xi\in \beta X$, Proposition \ref{PropGNSIrreducible} says that the GNS representation $\pi_\xi$ is irreducible. Therefore, each $\ker(\pi_\xi)$ is a primitive ideal of $\cstu(X)$. Our next results will relate the dynamics of $(\beta X,\PT(X))$ with the topological space of primitive ideals of $\cstu(X)$. Let
\[\mathrm{Prim}(\cstu(X))=\{I\subseteq \cstu(X)\colon I \text{ is an primitive ideal of } \cstu(X)\}\]
and let $\mathrm{Irr}(\cstu(X))$ denote the space of all irreducible representations of $\cstu(X)$ modulo unitary equivalence. As usual, we endow $\mathrm{Prim}(\cstu(X))$ with the hull-kernel topology (or Jacobson topology), i.e, the topology generated by sets of the form
\[\mathrm{Prim}(\cstu(X))\setminus \left\{ J\in \mathrm{Prim}(\cstu(X))\colon \bigcap_{J'\in R}J '\subseteq J\right\}
,\]
where $R$ runs over all subsets of $\mathrm{Prim}(\cstu(X))$. We then endow $\mathrm{Irr}(\cstu(X))$ with the smallest topology which making the surjection
\begin{equation*}\label{Eq.Quotient.Irr.Prim}
 [\pi]\in \mathrm{Irr}(\cstu(X))\mapsto \ker(\pi)\in \mathrm{Prim}(\cstu(X))
\end{equation*}
continuous.

 \begin{remark}\label{RemarkTopLattice}
     As explained in Section \ref{SubSubSectionIdeals} above, the maps \[\xi\in \beta X\mapsto \ker(\pi_\xi)\in \mathrm{Prim}(\cstu(X))\] can be interpreted as restrictions of maps between $\cF(\beta X)$ and $\cI(\cstu(X))$ (see \eqref{Eq.IOrbRelation}) and such maps have been extensively studied before (see \cite{ChenWang2004JFA,ChenWang2005}). Our approach here has however a fundamental difference: since the image of our maps lies in $\mathrm{Prim}(\cstu(X))$ (and not only in $\cI(\cstu(X))$), we deal with topological properties of these maps and not only lattice properties as done in \cite{ChenWang2004JFA,ChenWang2005}.
     \end{remark}

\begin{proposition}\label{PropMapFromBetaXToPrim}
    Let $X$ be a u.l.f.\ metric space. The map
    \[\psi\colon \xi\in \beta X\mapsto \ker(\pi_\xi)\in \mathrm{Prim}(\cstu(X))
    \]
    is  continuous.
    \end{proposition}

\begin{proof}
Fix $\xi\in \beta X$ and let us show  that $\psi$ is continuous at $\xi$. For that, let $ (\xi_i)_{i\in I}$ be a net in $\beta X$ converging to $\xi$ and let us show $(\ker(\pi_{\xi_i}))_{i\in I}$ converges to $\ker(\pi_\xi)$. If not, we can fix an open neighborhood  $W\subseteq \mathrm{Prim}(\cstu(X))$ of $\ker(\pi_\xi)$  such that, going to a subnet if necessary, $\ker(\pi_{\xi_i})\not\in W$ for all $i\in I$. By the definition of the topology of $\mathrm{Prim}(\cstu(X))$, we can assume
\[W=\mathrm{Prim}(\cstu(X))\setminus \left\{ J\in \mathrm{Prim}(\cstu(X))\colon \bigcap_{J'\in R}J '\subseteq J\right\}\]
for some $R\subseteq \mathrm{Prim}(\cstu(X))$.

As $\ker(\pi_\xi)\not\in W$, $\ker(\pi_\xi)$ cannot contain $\cap_{J'\in R}J '$. Fix $a\in \cap_{J'\in R}J '$ such that $a\not\in \ker(\pi_{\xi})$. By Corollary \ref{Kernels}, there is $\eta\in \Orb(\xi)$ such that
\begin{equation}\label{EqE(a)(zezta)notzero}
E(a^*a)(\eta)\neq 0,
\end{equation}
where $E\colon \cstu(X)\to \ell_\infty(X)$ is the canonical conditional expectation and $E(a)$ is interpreted as a function in $C(\beta X)$. On the other hand, as each $\ker(\pi_{\xi_i})$ is not in  $W$, we must have that $a\in \ker(\pi_{\xi_i})$ for all $i\in I$. So,
\begin{equation}\label{EqE(a)(zezta)notzero2}
E(a^*a)\restriction \Orb(\xi_i)=0\ \text{ for all }\ i\in I.
\end{equation}
As $\eta\in \Orb(\xi)$ and $\xi=\lim_{i,I}\xi_i$, there is a net $(\eta_i)_{i \in I'}$ such that $\eta=\lim_{i,I'}\eta_i$ and each $\eta_i$ is in the orbit of some $\xi_j$, $j\in I$ (see Lemma \ref{LemmaConvergingNetsOrbitEquiv}). By the continuity of $E(a^*a)$ and  \eqref{EqE(a)(zezta)notzero2}, we have
\[E(a^*a)(\eta)=\lim_{i,I'}E(a)(\eta_i)=0,\]
which  contradicts \eqref{EqE(a)(zezta)notzero}.
\end{proof}

\begin{theorem}
    Let $X$ be a u.l.f.\ metric space. Both the maps \begin{equation*}  \Orb(\xi)\in \calOrb(\beta X) \mapsto [\pi_\xi]\in \mathrm{Irr}(\cstu(X))
    \end{equation*}
    and \begin{equation}\label{EqMapOrbToPrim.2Intro} \qOrb(\xi)\in \qcalOrb(\beta X) \mapsto \ker(\pi_\xi)\in \mathrm{Prim}(\cstu(X))
    \end{equation}
    are homeomorphisms with their images. \label{ThmMapFromOrbToPrimIsNice.InTheText}
\end{theorem}

\begin{proof}
Let us first show that  map
\begin{equation}\label{eq.Psi.24}\Psi\colon \Orb(\xi)\in \calOrb(\beta X)\mapsto \ker(\pi_\xi)\in \Psi(\calOrb(\beta X))\subseteq \mathrm{Prim}(\cstu(X))
\end{equation}
is continuous and open.
The continuity of $\Psi$ follows from Proposition \ref{PropMapFromBetaXToPrim} and the fact that the quotient map  $q\colon \beta X\to \calOrb(\beta X)$ is open (see Corollary \ref{CorQuotientOpen}). Indeed, letting $\psi$ be as in Proposition \ref{PropMapFromBetaXToPrim}, the diagram if Figure \ref{Diagram1}. So, $\Psi^{-1}(W)=q(\psi^{-1}(W))$ is open for all open $W\subseteq \mathrm{Prim}(\cstu(X))$.
\begin{figure}[h] \[\xymatrix{\beta X \ar[rr]^\psi\ar[dr]_q& & \mathrm{Prim}(\cstu(X))\\
& \calOrb(\beta X)\ar[ru]_\Psi&
}\]
\caption{}\label{Diagram1}
\end{figure}

We us   notice $\Psi\colon \calOrb(\beta X)\to \Psi(\calOrb(\beta X))$ is an open map. Suppose not and pick $U\subseteq \calOrb(\beta X)$ such that $\Psi(U)$ is not open in $\Psi(\calOrb(\beta X))$. Then,  we can find $\xi\in \beta X$ and a net $(\xi_i)_{i\in I}$ in $\beta X$ such that
\[\ker(\pi_\xi)=\lim_{i,I}\ker(\pi_{\xi_i}), \ \Orb(\xi)\in U, \text{ and }\ \Orb(\xi_i)\not\in U\ \text{ for all }\ i\in I.\]
From the last two conditions above and the definition of the quotient topology of $\calOrb(\beta X)$, there is an open orbit-invariant subset $V\subseteq \beta X$ such that $\xi \in V$ and $\xi_i\not\in V$ for all $i\in I$. In particular, this implies that $\xi$ does not belong to $\COrb(\{\xi_i\colon i\in I\})$. Hence, by Corollary \ref{Kernels}, we have that
\[\bigcap \ker(\pi_{\xi_i})\not\subseteq \ker(\pi_\xi).\]

Set
\[W=\mathrm{Prim}(\cstu(X))\setminus \left\{J\in \mathrm{Prim}(\cstu(X))\colon \bigcap \ker(\pi_{\xi_i})\subseteq J\right\}.\]
By the previous paragraph, $\ker(\pi_\xi)\in W$. So, $W$ is an open neighborhood of $\ker(\pi_\xi)$ and, as $\ker(\pi_\xi)=\lim_{i,I}\ker(\pi_i)$, there must be some $i\in I$ such that $\ker(\pi_{\xi_i})\in W$. This clearly conflicts with the definition of $W$, so we reached a contradiction.

We now show
\[\Psi_1\colon \Orb(\xi)\in \calOrb(\beta X)\mapsto [\pi_\xi]\in \mathrm{Irr}(\cstu(X))\]
 is a homeomophism with its image. Indeed, this follows since the map in \eqref{eq.Psi.24} is continuous and open, the quotient map
\[\theta\colon \mathrm{Irr}(\cstu(X))\to \mathrm{Prim}(\cstu(X))\]
is open, and the  diagram in Figure \ref{Diagram2} commutes.
\begin{figure}[h] \[\xymatrix{\calOrb(\beta X) \ar[rr]^{\Psi}\ar[dr]_{\Psi_1}& & \mathrm{Prim}(\cstu(X))\\
& \mathrm{Irr}(\cstu(X))\ar[ru]_\theta&
}\]\caption{}\label{Diagram2}
\end{figure}

We are left to show
\[\Phi_1\colon \qOrb(\xi)\in \qcalOrb(\beta X)\mapsto \ker(\pi_\xi)\in \mathrm{Prim}(\cstu(X))\]
is a homeomorphism with its image. This will follow analogously to what we did for  $\Psi_1$ but considering the map \[\Phi\colon \qOrb(\xi)\in \qcalOrb(\beta X)\mapsto \ker(\pi_\xi)\in \Phi(\qcalOrb(\beta X))\subseteq \mathrm{Prim}(\cstu(X))\] instead of $\Psi$ and using that the  quotient map $\colon\beta X\to \qcalOrb(\beta X)$ is open (Corollary \ref{CorQuotientOpen}). We leave the details to the reader.
\end{proof}

\section{Irreducible subsets of $\partial X$ and prime versus primitive ideals}

\subsection{Primitive and prime ideals of $\cstu(X)$}\label{SubsectionIrreducibleSet} We shall now relate problem of whether prime ideals in $\cstu(X)$ are primitive to the problem of whether irreducible subsets of $\beta X$ are of the form $\COrb(\xi)$ for some $\xi\in \beta X$. Applications of the results in this section will be obtained in Section \ref{SectionIrred}.

In this section, we use the maps \[F\in \cF(\beta X)\mapsto I_F\in \cI(\cstu(X))\]
and
\[  I\in \cI(\cstu(X))\mapsto F_I\in \cF(\beta X)\] introduced in Section \ref{SubSubSectionIdeals}  (see Proposition \ref{PropBijPropA}).

The next result shows that, if irreducible sets are orbit closures, then primes of $\cstu(X)$ are primitive under property A.

\begin{theorem}\label{Thm.IrredCLosureofOrbImplisPrimeIsPrimitive}
    Let $X$ be a u.l.f.\ metric space with property A. Suppose that every irreducible subset $F\subseteq \beta X$ is of the form $\COrb(\xi)$. Then every prime ideal of $\cstu(X)$ is primitive.
\end{theorem}

\begin{proof}
 Let $I\subseteq \cstu(X)$ be a prime ideal. By Propositions \ref{PropBijPropA} and  \ref{Prop.Irrd.Gives.Prime}, $F_I$ is a irreducible  subset of $\beta X$. By the hypothesis, there is $\xi\in \beta X$ such that   $F_I= \COrb(\xi)$.  Since $X$ has property A, Proposition \ref{PropBijPropA} gives that $I=I_{F_I}$, i.e.,
     \[I=\{a\in \cstu(X)\colon E(a^*a)\restriction F_I=0\}.\] On the other hand, by
     Corollary \ref{Kernels},
     \[\ker(\pi_\xi)=\{a\in \cstu(X)\colon E(a^*a)\restriction \Orb(\xi) =0\}.\]
Therefore, using  \eqref{Eq.IOrbRelation}, we have
\[\ker(\pi_\xi)=I_{\COrb(\xi)}=I.\]  By Proposition \ref{PropGNSIrreducible}, $\pi_\xi$ is an irreducible representation, so, $I$ is primitive   and we are done.
\end{proof}

 \subsection{Spaces where irreducible sets are closures of orbits}\label{SectionIrred}

We shall now provide (nontrivial) examples of u.l.f.\ metric spaces such that the irreducible subsets of $\partial X$ are of the form $\COrb(\xi)$ for some $\xi\in \partial X$. Recall that, if $\calOrb(\partial X)$ is Hausdorff, then this is automatically. So, we must look for examples outside this class of spaces.  We start showing that the orbit space being $T_1$ is already enough  for this.

 \begin{theorem}
      Let $X$ be a u.l.f.\ metric space for which $\calOrb(\partial X)$ is $T_1$. Then every  irreducible subset of $\calOrb(\partial X)$ is the closure of a singleton.  In particular,  the prime and primitive ideals of $\cstu(X)$ coincide.\label{Cor.T1.Prime=Primitive.InTheText}
\end{theorem}

The proof of  Theorem \ref{Cor.T1.Prime=Primitive.InTheText} will be an immediate application of Theorems \ref{ThmCharacterizationOrbClosed.IntheText} and \ref{CorClosedQuasiOrbitsCoarseDisjointUnionSing.Inthetext}, and the next theorem, which is a general result about inverse semigroups acting on compact Hausdorff spaces. For the following result, recall that an inverse semigroup $S$ \emph{acts} on a topological space $M$ if for all $s\in S$ there are open subsets $\dom(s),\mathrm{im}(s)\subseteq M$ such that $s$ is a homeomorphism between $\dom(s)$ and $\mathrm{im}(s)$. Moreover, we assume that this action is so that $M=\bigcup_{s\in S}\dom(s)$.
Similarly as above, we define the \emph{orbit} of an element $x\in M$ as 
\[\Orb(x)=\{s(x)\mid s\in S\ \text{ and }\ x\in \dom(s)\}\]
and a nonempty closed (orbit) invariant subset $F\subseteq M$ is called \emph{irreducible} if it cannot be written as $F_0\cup F_1$ for any closed invariant subsets $F_0,F_1\subseteq M$ which are not equal to $F$. 

\begin{theorem}\label{Thm.Irr.General.Inv.Semi.Group} Let $M$ be a locally compact Hausdorff space and let $S$ be an inverse semigroup
acting on $M$.
  If every orbit is finite, then every  irreducible subset $F\subseteq M$ coincides
with the orbit of a single point.
\end{theorem}

\begin{proof}
For each $n\in {\mathbb N}$, let
  \[
  O_n = \{x\in  M: |\Orb(x)|\geq  n\},
  \]
  and let us prove that $O_n$ is open.  Indeed, given $x$ in $O_n$, let
  $s_1,s_2, \ldots , s_n\in S$,
  be such that  the set $\{s_i(x): 1\leq i\leq n\}$ has $n$ distinct elements.  For each
$i$,  choose an open neighborhood $A_i$ of $s_i(x)$,  such that the ${A_i}'s$ are pairwise
disjoint  and let
  \[
  U = \bigcap_{i=1}^n s_i^{-1}(A_i).
  \]
  So, $U$ is an open neighborhood of $x$ contained in the domain of each $s_i$ and
 so that $s_i(U)\subseteq A_i$.
As the $A_i$'s are disjoint,  for each $y\in U$, the set  $\{s_i(y): 1\leq i\leq n\}$ has $n$ distinct elements. In particular,  $U\subseteq O_n$.

By restricting the action of $S$ to   $F$  we may assume
without loss of generality that 
that $M$ itself is irreducible.

\begin{claim}
The number of elements in each orbit is uniformly bounded from
above.
\end{claim}

\begin{proof} Suppose by contradiction that  $O_n$ is nonempty for every $n$.  Since $O_n$ is
clearly invariant, the irreducibility of $M$ implies that each $O_n$ is dense.  Using
the Baire Category Theorem, we get that $\bigcap_{n\in {\mathbb N}}O_n$ is nonempty. A point in this intersection
would have to have infinite orbit, contradicting the hypotheses.
\end{proof}

\begin{claim} There exists a unique orbit of size
  \begin{equation}\label{Eq.1.01.Jun.26}
  n = \max\{|\Orb(x)| : x\in M\}.
  \end{equation}
\end{claim}

\begin{proof} Supposing by contradiction that  $\Orb(x)$ and $\Orb(y)$ are
disjoint, and both have $n$ elements,  choose
  \[
  s_1,s_2, \ldots , s_n;\   t_1,t_2, \ldots , t_n\in S,
  \]
  such that
  \[
  \Orb(x) = \{s_i(x):1\leq i\leq n\} \ \text{and}\    \Orb(y) = \{t_i(y):1\leq i\leq n\}.
  \]
  Next choose a collection of pairwise disjoint open sets
  \[
  A_1,A_2, \ldots , A_n;\   B_1,B_2, \ldots , B_n,
  \]
  such that $s_i(x)\in A_i$ and $t_i(y)\in B_i$, for all $i\in \{1,\ldots, n\}$, and let
  \[
  U= \bigcap_{i=1}^n s_i^{-1}(A_i)  \ \text{and}\   V= \bigcap_{i=1}^n t_i^{-1}(B_i).
  \]
  Then clearly $x\in U$, $y\in V$, and moreover
  \[  s_i(U)\subseteq A_i \ \text{and}\   t_i(V)\subseteq B_i,
  \]
  for all $i\in \{1,\ldots, n\}$.  Consequently the open sets
  \[
  {\mathcal A}= \bigcup_{i=1}^n s_i(U)\  \text{ and }\ {\mathcal B}=  \bigcup_{i=1}^n t_i(V),
  \]
  are disjoint, given that
  \[
  {\mathcal A}\subseteq  \bigcup_{i=1}^n A_i,  \ \text{and}\    {\mathcal B}\subseteq  \bigcup_{i=1}^n B_i.
  \]
  We next claim that, for every $x'\in U$, one has that
  \[
  \{s_i(x'):1\leq i\leq n\} = \Orb(x').
  \]
  The inclusion ``$\subseteq $'' is clearly true while the reverse inclusion follows from
the fact that the set in the left-hand-side   has exactly $n$ distinct elements,
given that $s_i(x')$ lies in $A_i$, and the  $A_i$ are disjoint,  and also from the
fact that $\Orb(x')$ has at most $n$
points.  This shows that
  \[  {\mathcal A} = \bigcup_{x'\in U}\Orb (x'),
  \]
so
$\mathcal A$ is invariant.
The exact same argument shows that $\mathcal B$ is invariant, so we come to a
contradiction because an irreducible set cannot contain two disjoint open invariant
subsets.  This proves the claim \end{proof}

\begin{claim} $M$ is finite.
\end{claim}

\begin{proof}
    We will prove this by induction on $n$, as defined in \eqref{Eq.1.01.Jun.26}.
If $n=1$, then all points in $M$ are fixed, while there exists a unique fixed point by
claim (3), so $M$ is a singleton.    Assuming that $n>1$, and that the conclusion
holds for any $m<n$, let $\Orb(x)$ be the unique orbit with $n$ elements.  Then
  \[
  M'= M\setminus \Orb(x)
  \]
  is an open invariant subset of $M$, with maximum orbit size less than $n$, and it
is clearly also  locally compact, and irreducible.  Restricting the
action to $M'$, the induction hypothesis implies that $M'$ is finite, so clearly
$X$ is also finite.
\end{proof}

We can now conclude that  $M$ coincides with an orbit. Indeed, being finite, $M$ is discrete, so
every orbit is both open and closed and the conclusion follows easily.
\end{proof}

 \begin{proof}[Proof of Theorem \ref{Cor.T1.Prime=Primitive.InTheText}]
 Suppose   $\calOrb(\partial X)$ is $T_1$ and let   $F\subseteq \partial X$ be an irreducible subset. By  Theorem \ref{CorClosedQuasiOrbitsCoarseDisjointUnionSing.Inthetext}, every $\Orb(\xi)$ is closed and hence, by Theorem \ref{ThmCharacterizationOrbClosed.IntheText}, finite. By Theorem \ref{Thm.Irr.General.Inv.Semi.Group}, $F$ is the closure of an orbit and we are done.\end{proof}

We now show that any
u.l.f.\ metric space which coarsely embeds into
\[\cM_{\leq k}=\bigsqcup_{i=1}^k\cM_i,\]
 for some $k\in\N$, also has the property that the prime ideals of its uniform Roe algebra are primitive. Notice that, for $k\geq 2$, the orbit space of $\cM_{\leq k}$ is not $T_1$  (Theorem \ref{CorClosedQuasiOrbitsCoarseDisjointUnionSing.Inthetext}), so Theorem \ref{Cor.T1.Prime=Primitive.InTheText} cannot be applied here.

\begin{theorem}\label{theo:embinmk}
Let $X$ be a u.l.f.\ metric space which coarsely embeds into $\bigsqcup_{i=1}^k\cM_{i}\subseteq \cM$ for some $k\in\N$. Then every irreducible subset of $\calOrb(\partial X)$ is the closure of a singleton. In particular, the prime and primitive ideals of $\cstu(X)$ coincide.   \label{Thm.M_k.Irreducible.Inthetext}
\end{theorem}

\begin{proof}  First notice that, by Proposition \ref{PropCoarseEmbOrbSpaceHomeo}, we can assume that $X\subseteq \cM_{\leq k}$.    We start setting up some notation.  For $i\in \N$ and $j\in \{1,\ldots, k\}$, let
    \[M_{j}(i)=\{(a_i)_i\in X\colon |\{  i'>i\colon a_{i'}=1\}|=j\}.\]
   For each $\xi\in \partial X$,  define
\[\ell(\xi)=\min\{j\in \{1,\ldots, k\}\colon \exists i\in \N,\ \forall i'\geq i,\  M_{j}(i)\in \xi\}.
\]
The fact such minimum exists follows  from $\xi$ being a nonprincipal ultrafilter and $X$ being a subset of $\cM_{\leq k}$. Indeed, for each $i\in\N$, we can partition $X$ into at most $k+1$ many subsets, say $X_0,\ldots, X_k$, such that each element in $X_j$ has $j$ coordinates with $1$ after the $i$th coordinate. As $\xi$ is an ultrafilter, it must contain an element in this partition. Moreover, as  $X_0$ is a finite set and $\xi$ is nonprincipal, there is $j(i)\in \{1,\ldots, k\}$ with $X_{j(i)}\in \xi$. As $i\to \infty$, $j(i)$ decreases, but it can never reach zero as $\xi$ does not contain finite sets. So, $\ell(\xi)=\lim_{i\to \infty}j(i)$.

 Let $F\subseteq \partial X$ be irreducible and pick $\xi\in F$ such that
\begin{equation}\label{Eq.ellxiDef}\ell(\xi)=\min_{\xi'\in F}\ell(\xi').\end{equation}
 Fix $i(\xi)\in\N$ such that $M_{\ell(\xi)}(i)\in \xi$ for all $i\geq i(\xi)$.  Let us show that $F=\COrb(\xi)$.

Towards a contradiction, suppose that $F\setminus\COrb(\xi)$ is nonempty and pick $\zeta\in F\setminus \COrb(\xi)$ with
\begin{equation}\label{Eq.ellzetaDef}\ell(\zeta)=\min_{\xi'\in F\setminus \COrb(\xi)}\ell(\xi').\end{equation}
  Fix $i(\zeta)\in\N$ such that $M_{\ell(\zeta)}(i)\in \zeta$ for all $i\geq i(\zeta)$ and set \[i_0=\max\{i(\xi),i(\zeta)\}.\]  As $\zeta\not\in \COrb(\xi)$, we can pick  $B\in \zeta$   such  that $\overline{B}$ does not intersect $\Orb(\xi)$.   Let $A= M_{\ell(\xi)}(i_0)$ and, shrinking  $B$   if necessary,  assume that  $  B\subseteq M_{\ell(\zeta)}(i_0) $.
 Define
\begin{align*}
A_0=\{(a_i)_i\in A\colon \exists &(a_1',\ldots, a_{i_0}')\in \{0,1\}^{i_0} \text{ such that }\\ &(a_1',\ldots, a_{i_0}',a_{i_0+1},a_{i_0+2},\ldots)\in B\}.
\end{align*}
Notice that, as $X$ is u.l.f.\ and $\overline B$ does not intersect $\Orb(\xi)$, $A_1=A\setminus A_0\in \xi$. Indeed, by the definition of $A_0$, there is a map $f\colon A_0\to B$ such that
\[\sup_{\bar a\in A_0}d(\bar a,f(\bar a))<\infty.\]
As $f$ is not necessarily injective, $f$ may not be a partial translation. However, there is a partition of $A_0$ into finitely many subsets such that the restriction of $f$ to each of this subsets is injective. If it was the case that $A_0\in \xi$, then, as $\xi$ is an ultrafilter, one of this subsets would belong  to $\xi$ and this would contradict the fact that $\overline B$ does not intersect $\Orb(\xi)$.

 As $F$ is irreducible, the closure of an invariant open subset intersecting $F$ must contain the whole $F$. In particular, $\zeta\in \COrb(\overline{A_1})$ which implies that    $\overline {B}\cap F$ must intersect $\Orb(\overline{A_1})$. So,  there is $\eta\in \overline{B}\cap F$ and a partial translation $f$ of $ X$ such that $\dom(f)\in \eta$, $\dom(f)\subseteq B$,  and $\mathrm{im}(f)\subseteq A_1$.    By the definition of the metric of $M$, there is $i_1\geq i_0$ such that $f$  does not change coordinates with index greater than  $i_1$. More precisely, if $\bar b=(b_i)_i\in \dom(f)$   then $f(\bar b)=(b'_0,\ldots, b_{i_1}', b_{i_1+1},b_{i_1+2} ,\ldots)$ for some $b_0',\ldots, b_{i_1}'\in \{0,1\}$.

We shall now reach a contradiction by looking at $\ell(\eta)$. Firstly, by \eqref{Eq.ellxiDef} and \eqref{Eq.ellzetaDef}, it is clear that
\begin{equation}\label{Eq.Inequalityells}\ell(\eta)\geq \ell(\zeta)\geq \ell(\xi).
\end{equation} Suppose $\ell(\eta)>\ell(\xi)$. As $\eta$ is a  filter and $M_{\ell(\eta)}(i)\in \eta$ for $i$ large enough, we can pick $\bar b\in \dom(f)$ with more than $\ell(\xi)$ coordinates with $1$'s after the  $i_1$th coordinate. As $f$ does not change coordinates indexed after $i_1$,   the same holds for $f(\bar b)$. However, since  $f(\bar b)\in A_1 \subseteq   M_{\ell(\xi)}(i_0)$ and $i_1\geq i_0$, the vector $f(\bar b)$ has at most $\ell(\xi)$ coordinates with $1$'s indexed after $i_1$; contradiction.

We are then left with $\ell(\eta)=\ell(\xi)$.
By \eqref{Eq.Inequalityells},    \[ \ell\coloneq\ell(\eta)=\ell(\zeta)=\ell(\xi).\]
So there is $i_2\geq i_1$ such that  $M_{\ell}(i_2)\in \eta$. As $\eta$ is a filter, we can pick $\bar b \in M_{\ell}(i_2)\cap \dom(f)$.  Since $\dom(f)\subseteq B$ and $\mathrm{im}(f)\subseteq A_1$, and since   $B$ and $A_1$ are contained in $M_\ell(i_0)$,  both $\bar b$ and $f(\bar b)$ have  $\ell$ coordinates indexed after $i_0$ with $1$'s. On the other hand, as $\bar b$ is in  $M_{\ell}(i_2)$  and as $f$ does not chance coordinates indexed after $i_1$, this implies that the coordinates of both  $\bar b$ and $f(\bar b)$ have  $\ell$ coordinates indexed after $i_2$ with $1$'s. Therefore, the coordinates of both $\bar b$ and  $f(\bar b)$ with index in $(i_0,i_2)$ must all be zero.  But then $f(\bar b)\in A_0$; contradiction.

The last sentence in the statement of the theorem claims that the prime ideals of $\cstu(X)$ are primitive. This is a consequence of the above and Theorem \ref{Thm.IrredCLosureofOrbImplisPrimeIsPrimitive}.
\end{proof}

\begin{remark}\label{rem:MkEquivalence}
Note that although Theorem \ref{theo:embinmk} is stated for spaces that coarsely embed into $\bigcup_{i \le k} \mathcal{M}_i$, this class of spaces coincides with those that coarsely embed into $\mathcal{M}_k$. Indeed, there exists a coarse embedding $f\colon  \bigcup_{i \le k} \mathcal{M}_i \to \mathcal{M}_k$ constructed as follows: for any $x \in \mathcal{M}_i$ (with $i \le k$), let $f(x) = y^\smallfrown x$ be the concatenation of a fixed-length prefix $y$ and the sequence $x$, where $y$ is chosen to have exactly $k-i$ ones. By choosing the length of $y$ sufficiently large (say, of length $k$), this map corresponds to shifting the support of $x$ to the right. This shift scales the metric by a factor of $3^k$ and the prefix $y$ contributes only a bounded error to the distance, making $f$ a coarse embedding. In particular, the property of having all prime ideals primitive is equivalent for both classes of spaces.
\end{remark}

 We finish the paper stating two main problems we left unsolved:
 \begin{problem}\label{Prob.Irred.FORB}
     Let $X$ be a u.l.f.\ metric space and $F\subseteq \partial X$ an irreducible subset. Is there $\xi\in \partial X$ for which $F=\COrb(\xi)$?
 \end{problem}

 \begin{remark}\label{RemarkIrreducibleCORB}
 Let $X$ be a u.l.f.\ metric space and $F\subseteq \partial X$ an irreducible subset. Then, if $\xi\in F$, either $F=\COrb(\xi)$ or $F\setminus \COrb(\xi)$ must be dense in $F$. Therefore, if $F$ could be partitioned into countably many  orbit closures, it would follow from Baire category theorem that the intersection of all their complements in $F$ is nonempty; contradicting that we started with a partition of $F$. Unfortunately, as shown in Proposition \ref{prop:uncountably_many_disjoint_orbits}, irreducible sets can contain a continuum of pairwise disjoint orbit closures. So, it is not clear how to use Baire category in order to solve Problem \ref{Prob.Irred.FORB}. \end{remark}
 \begin{problem}\label{Prob.PrimeArePrimitivaUra}
     Let $X$ be a u.l.f.\ metric space. If $I$ is a prime ideal of $\cstu(X)$, is $I$ primitive? What about under the extra assumption of $X$ having property A?
 \end{problem}

  From Theorem \ref{Thm.IrredCLosureofOrbImplisPrimeIsPrimitive}, a positive answer to Problem \ref{Prob.Irred.FORB} implies a positive answer to Problem \ref{Prob.PrimeArePrimitivaUra} under the presence of property A.

\begin{acknowledgments}
 The authors would like to thank Rufus Willett and Ján \v{S}pakula for discussions about this paper. B. M. Braga would also like to thank Ilijas Farah, Alejandro Kocsard, and Davi Obata. The authors would also like to thank  Aleksander Fedorynski for finding an error in a previous version of this paper. More precisely, an earlier version of this paper had an extra item in Theorem \ref{ThmHigsonRel} asserting that, for u.l.f.\ metric spaces coarsely embeddable in finitely generated groups, the items in the theorem would also be equivalent to the existence of $\zeta\in \partial X$ such that $\xi,\eta\in \COrb(\zeta)$. Not only the proof was incorrect, but also this result is false; see \cite[Theorem 5.4]{Protasov2005TopAndAppl}.
 \end{acknowledgments}
 \bibliographystyle{amsalpha}
 \bibliography{bibliography}

\end{document}